\documentclass[10pt]{article}

\usepackage{amsmath, amsthm, amsfonts}
\usepackage{mathtools}
\usepackage{hyperref}
\hypersetup{citecolor=blue, colorlinks=true, linkcolor=black}
\usepackage{tikz}
\usetikzlibrary{arrows,decorations.pathmorphing,backgrounds,positioning,fit,matrix}
\usepackage{color}
\usepackage{enumitem}
\setlist[itemize]{leftmargin=.5in}
\setlist[enumerate]{leftmargin=.5in,topsep=3pt,itemsep=3pt,label=(\roman*)}
\usepackage{geometry}
\newgeometry{margin=2.5cm}
\usepackage{algorithm, algpseudocode}

\theoremstyle{plain}
\newtheorem{theorem}{Theorem}[section]

\newtheorem{lemma}[theorem]{Lemma}

\theoremstyle{definition}
\newtheorem{definition}[theorem]{Definition}

\theoremstyle{remark}
\newtheorem{remark}[theorem]{Remark}

\numberwithin{equation}{section}

\ifpdf
\DeclareGraphicsExtensions{.eps,.pdf,.png,.jpg}
\else
\DeclareGraphicsExtensions{.eps}
\fi

\renewcommand{\epsilon}{\varepsilon}
\def\eps{\epsilon}

\def\Ten{{\mathrm{Ten}}}
\def\Sym{{\mathrm{Sym}}}
\def\per{{\mathrm{per}}}

\def\forae{\text{ for a.e. }}
\def\Teps{{T^\epsilon}}
\def\Om{\Omega}
\newcommand{\sureps}[1]{\tfrac{#1}{\epsilon}}
\def\xsureps{\sureps{x}}
\def\ofxsureps{\big(\xsureps\big)}

\def\R{\mathbb{R}}
\def\N{\mathbb{N}}

\def\Npos{\N_{>0}}
\newcommand{\closure}[1]{\bar{#1}}

\def\fcnC{\mathcal{C}}
\def\fcnL{\mathrm{L}}
\def\fcnH{\mathrm{H}}
\def\fcnW{\mathrm{W}}

\def\calW{\mathcal{W}}
\def\calL{\mathcal{L}}
\def\calO{\mathcal{O}}

\def\calV{\mathcal{V}}
\def\calH{\mathcal{H}}

\def\famEfEq{\mathcal{E}}

\newcommand{\Lp}[1]{\fcnL^{#1}}
\def\Ld{{\Lp2}}
\def\Ldz{{\Lp2_0}}
\def\Li{{\Lp\infty}}
\def\Lu{{\Lp1}}
\def\Ldper{\fcnL^2_\per}

\newcommand{\Hk}[1]{\fcnH^{#1}}
\def\Hu{{\Hk1}}

\def\Huper{{\Hk1_\per}}

\newcommand{\of}[1]{(#1)}

\def\ofOm{\of\Omega}
\def\ofY{\of{Y}}

\newcommand{\Lti}[1]{(0,\Teps;#1)}

\def\Wper{{\fcnW_{\!\per}}}
\def\WperStar{{\fcnW_{\!\per}^*}}

\def\calLd{{\calL^2}}
\def\calWper{{\calW_{\hspace{-0.5pt}\per}}}
\def\calWperStar{{\calW_{\hspace{-0.5pt}\per}^*}}

\def\LdOm{{\Ld\ofOm}}
\def\LdzOm{{\Ldz\ofOm}}

\newcommand{\HkOm}[1]{{\Hk{#1}\ofOm}}

\def\calLdOm{{\calLd\ofOm}}
\def\calWperOm{{\calWper\ofOm}}
\def\calWperStarOm{{\calWperStar\ofOm}}
\def\WperOm{{\Wper\ofOm}}
\def\WperStarOm{{\WperStar\ofOm}}

\def\WperY{{\Wper\ofY}}
\def\WperStarY{{\WperStar\ofY}}

\def\LdY{{\Ld\ofY}}

\def\calWnorm{\mathcal{W}}
\def\Wnorm{W}

\def\bbrackl{\scalebox{2.1}[1.1]{\ensuremath{[}}}
\def\bbrackr{\scalebox{2.1}[1.1]{\ensuremath{]}}}
\newcommand{\eqc}[1]{\boldsymbol{#1}}
\newcommand{\eqcbis}[1]{\bbrackl{#1}\bbrackr}

\def\bbrackls{\scalebox{2.1}[1]{\ensuremath{[}}}
\def\bbrackrs{\scalebox{2.1}[1]{\ensuremath{]}}}
\newcommand{\eqcbiss}[1]{\bbrackls{#1}\bbrackrs}

\newcommand{\norm}[1]{\|#1\|}
	\newcommand{\normbig}[1]{\big\|#1\big\|}
	
\newcommand{\seminorm}[1]{|#1|}

\newcommand{\ps}[1]{(#1)}
	\newcommand{\psbig}[1]{\big(#1\big)}
	\newcommand{\psBig}[1]{\Big(#1\Big)}
\newcommand{\intMean}[1]{\langle{#1}\rangle}
	\newcommand{\intMeanbig}[1]{\big\langle{#1}\big\rangle}
	\newcommand{\intMeanBig}[1]{\Big\langle{#1}\Big\rangle}

\def\contEmb{\hookrightarrow}

\def\dt{\partial_t}
\def\dtt{{\partial^2_t}}
\def\d{\mathrm{d}}

\def\pa{\partial}

\def\calA{\mathcal{A}}

\makeatletter
\newcommand*{\shifttext}[2]{%
  \settowidth{\@tempdima}{#2}%
  \makebox[\@tempdima]{\hspace*{#1}#2}%
}
\makeatother

\newcommand{\middleintspec}[3]{ 
	\ensuremath{ \scalebox{1.2}[1.6]{$\int$}\raisebox{-6pt}{\shifttext{-6pt}{$\scriptstyle #1$}}
											\raisebox{10pt}{\shifttext{-6.5pt}{$\scriptstyle #2$}}
				\hspace{#3} 	}
}

\newcommand{\nabx}{\ensuremath{\nabla_{\shifttext{-3pt}{$\scriptstyle x$}}\kern-1.5pt}}
\newcommand{\nabxn}[1]{\nabla^{#1}_{x}}
\newcommand{\naby}{\ensuremath{\nabla_{\shifttext{-3pt}{$\scriptstyle y$}}\kern-1.5pt}}

\newcommand{\nabxk}[1]{\ensuremath{\nabla^{#1}_{\shifttext{-3pt}{$\scriptstyle x$}}\kern-1.5pt}}

\def\aeps{a^\epsilon}

\def\tildeub{\tilde{u}}

\def\ueps{u^\epsilon}

\def\Aeps{\mathcal{A}^\eps}
\def\Beps{\mathcal{B}^\eps}
\def\BBeps{\eqc{\mathcal{B}}^\eps}

\def\Reps{\mathcal{R}^\eps}

\def\cini{u_{0}}		
\def\dini{u_1}		

\DeclareFontFamily{U}{mathx}{\hyphenchar\font45}
\DeclareFontShape{U}{mathx}{m}{n}{
      <5> <6> <7> <8> <9> <10>
      <10.95> <12> <14.4> <17.28> <20.74> <24.88>
      mathx10
      }{}
\DeclareSymbolFont{mathx}{U}{mathx}{m}{n}
\DeclareMathSymbol{\bigtimes}{1}{mathx}{"91}

\newcommand\twodigits[1]{%
   \ifnum#1<10 0#1\else #1\fi
}

\newenvironment{proofof}[1]
{\par\indent\textit{Proof of #1.~}}
{\hspace*{\fill}$\square$}
\makeatletter
\newenvironment{breakablealgorithm}
{
	\vspace{0.8em}
	\begin{center}
		\refstepcounter{algorithm}
		\hrule height.8pt depth0pt \kern2pt
		\renewcommand{\caption}[2][\relax]{
			{\raggedright\textbf{\ALG@name~\thealgorithm} ##2\par}%
			\ifx\relax##1\relax 
			\addcontentsline{loa}{algorithm}{\protect\numberline{\thealgorithm}##2}%
			\else 
			\addcontentsline{loa}{algorithm}{\protect\numberline{\thealgorithm}##1}%
			\fi
			\kern2pt\hrule\kern2pt
		}
	}{
		\kern2pt\hrule\relax
	\end{center}
	\vspace{0.8em}
}
\makeatother

\def\cdotss{\cdot\cdot}
\def\otimess{\!\otimes\!}
\def\Cten{c}
\def\Pten{p}
\def\Dten{g}
\def\Gten{g}
\def\Hten{h}
\def\Kten{k}

\def\Yperfunc{\gamma}
\def\genTen{q}
\def\inds{n}

\def\chapalpha{5}

\newcommand{\ma}[1]{\textcolor{blue}{#1}\index {#1}}

\newcommand{\email}[1]{\href{#1}{#1}}
\newcommand{\TheTitle}{Effective models and numerical homogenization for wave propagation in heterogeneous media on arbitrary timescales} 
\newcommand{\TheAuthors}{Assyr Abdulle and Timoth\'ee Pouchon}

\title{{\TheTitle}}
\author{
	Assyr Abdulle\footnotemark[1] \footnotemark[3]
	\and 
	Timoth\'ee Pouchon\footnotemark[2] \footnotemark[3]
}
\date{}

\usepackage{fancyhdr}
\ifpdf
\hypersetup{
	pdftitle={\TheTitle},
	pdfauthor={\TheAuthors}
}
\fi

\begin{document}

	
	\maketitle
	\thispagestyle{plain}
	
	\renewcommand{\thefootnote}{\fnsymbol{footnote}}
	\footnotetext[1]{%
		\email{assyr.abdulle@epfl.ch}\\
		ANMC, Institute of Mathematics, \'Ecole Polytechnique F\'ed\'erale de Lausanne,	Station 8, CH-1015 Lausanne, Switzerland
	}
	\footnotetext[2]{%
		\email{timothee.pouchon@ed.ac.uk}\\
		University of Edinburgh,
		School of Mathematics,
		James Clerk Maxwell Building,
		Peter Guthrie Tait Road,
		Edinburgh EH9 3FD, United Kingdom.
	}
	\footnotetext[3]{%
		The authors are partially supported by the Swiss National Foundation, grant No. 200020\_172710.
	}
	\renewcommand{\thefootnote}{\arabic{footnote}}

\begin{abstract}
A family of effective equations for wave propagation in periodic media for arbitrary timescales $\mathcal{O}(\eps^{-\alpha})$, where $\eps\ll1$ is the period of the tensor describing the medium, is proposed.
The well-posedness of the effective equations of the family is ensured without requiring a regularization process as in previous models 
[A. Benoit and A. Gloria, arXiv:1701.08600, 2017], 
[G. Allaire, A. Lamacz, and J. Rauch, arXiv:1803.09455, 2018].
The effective solutions in the family are proved to be $\eps$ close to the original wave in a norm equivalent to the $\Li(0,\eps^{-\alpha}T;\LdOm)$ norm.
In addition, a numerical procedure for the computation of the effective tensors of arbitrary order is provided.
In particular, we present a new relation between the correctors of arbitrary order, which allows to substantially reduce the computational cost of the effective tensors of arbitrary order. 
This relation is not limited to the effective equations presented in this paper and can be used to compute the effective tensors of alternative effective models.
\end{abstract}




\section{Introduction}

The wave equation in heterogeneous media is widely used in many applications such as seismic inversion, medical imaging or the manufacture of composite materials.
We consider the following model problem: let $\Omega\subset\R^d$ be a hypercube and let $\ueps:[0,T]\times\Omega\to\R$ be the solution of
\begin{equation} 	\label{eq:waveEquation_ueps_alpha_intro}
\dtt\ueps(t,x) - \nabx\cdot\big(\aeps(x)\nabx\ueps(t,x)\big) = f(t,x)
\quad\text{in }(0,T]\times\Omega,
\end{equation}
where we require $x\mapsto\ueps(t,x)$ to be $\Omega$-periodic and the initial conditions $\ueps(0,x)$ and $\dt\ueps(0,x)$ are given.
As we allow the domain $\Omega$ to be arbitrarily large, \eqref{eq:waveEquation_ueps_alpha_intro} can be used to model wave propagation in infinite media. 
We assume here that the tensor $\aeps$ varies at the scale $\eps\ll1$ while the initial conditions and the source $f$ have wavelength of order $\mathcal{O}(1)$.
In such multiscale situations, standard numerical methods such as the finite element (FE) method or the finite difference (FD) method are accurate only if the size of the grid resolves the microscopic scale $\mathcal{O}(\eps)$.
Hence, as $\eps\to0$ or as the domain $\Omega$ grows the computational cost of the method becomes prohibitive
and multiscale numerical methods are needed.

Several multiscale methods for the approximation of \eqref{eq:waveEquation_ueps_alpha_intro} are available in the literature.
They can be divided into two groups (see \cite{AbH17} for a review).
First, the methods suited when the medium does not have scale separation: \cite{OwZ08}, \cite{JEG10,JiE12}, \cite{OwZ11}, and \cite{AbH14c}.
These methods rely on sophisticated finite element spaces relying on the solutions of localized problems at the fine scale.
Second, the methods suited when the medium has scale separation (i.e., a special structure of the medium is required).
These methods are built in the framework of the heterogeneous multiscale method (HMM):
the FD-HMM \cite{EHR11,ArR16} and the FE-HMM \cite{AbG11}.
In both methods, the effective behavior of the wave is approximated by solving micro problems in small sampling domains.

The FD-HMM and the FE-HMM rely on homogenization theory \cite{BLP78,San80,BaP89,JKO94,CiD99,MuT97}:
they are built to approximate the homogenized equation and thus provide approximations of $\ueps$ in an $\Li(0,T;\LdOm)$ sense.
The homogenization of the wave equation \eqref{eq:waveEquation_ueps_alpha_intro} is provided in \cite{BFM92}.
For a given sequence of tensors $\{\aeps\}_{\eps>0}$, we have the existence of a subsequence of $\{\ueps\}_{\eps>0}$ that converges weakly$^*$ in $\Li(0,T;\WperOm)$ to $u^0$ as $\eps\to0$ (definitions of the functional spaces are provided below). The limit $u^0$, called the \textit{homogenized solution}, solves the \textit{homogenized equation}
\begin{equation}	\label{eq:waveEquationHomo_intro}
\dtt u^0(t,x)-\nabx\cdot \big( a^0(x)\nabx u^0(t,x) \big) = f(t,x) \quad\text{in }(0,T]\times\Omega,
\end{equation} 
with the same initial conditions as for $\ueps$.
The homogenized tensor $a^0$ in \eqref{eq:waveEquationHomo_intro} is obtained as the $G$-limit of a subsequence of $\{\aeps\}_{\eps>0}$ (see \cite{Spa68,DeS73}).
In general, $a^0$ depends on the choice of the subsequence and thus no formula is available for its computation.
In this paper, we consider periodic media, i.e., we assume that the medium is described by 
\begin{equation}	\label{eq:assumPeriodic}
\aeps(x)=a\ofxsureps,\qquad \text{where } a(y) \text{ is $Y$-periodic},
\end{equation}
where $Y$ is a reference cell (typically $Y=(0,1)^d$).
Under assumption \eqref{eq:assumPeriodic}, $a^0$ is proved to be constant and an explicit formula is obtained (see e.g. \cite{BLP78,BaP89,JKO94,CiD99}):
it can be computed by means of the \textit{first order correctors}, which are defined as the solutions of \textit{cell problems} (i.e., elliptic equations in $Y$ based on $a(y)$ with periodic boundary conditions).
Therefore, in the periodic case the homogenized solution $u^0$ can be accurately approximated independently of $\eps$.

However, for wave propagation on large timescales, $\ueps$ develops dispersive effects at the macroscopic scale that are not captured by $u^0$.
Furthermore, if the initial conditions or the source have high spatial frequencies (in between $\mathcal{O}(1)$ and $\mathcal{O}(\eps)$), the dispersion appears at shorter times.
Hence, to develop numerical homogenization methods for long-time propagation, or in high frequency regimes, new effective models are required.

The study of this dispersion phenomenon has recently been the subject of considerable interest.
Analyses for periodic media and timescales $\mathcal{O}(\eps^{-2})$ are provided in \cite{SaS91,Lam11b,DLS14a,DLS15,AbP16b,ABV15,AbP16} and numerical approaches are studied in \cite{AGS14,ArR14}.
A result for locally periodic media for timescales $\mathcal{O}(\eps^{-2})$ was also obtained in \cite{AbP18}.
For arbitrary timescales $\mathcal{O}(\eps^{-\alpha})$, $\alpha\in\N$, effective equations were proposed in \cite{ALR18} and \cite{BeG17}.
The well-posedness of these equations is obtained using regularization techniques:
in \cite{BeG17}, the regularization relies on the tuning of an unknown parameter, which poses problems in practice;
in \cite{ALR18} a filtering process is introduced (yet not tested in practice).

%

In this paper, we present two main results first reported in \cite[Chap. \chapalpha]{Pou17}.
The first main result is the definition of a family of effective equations that approximate $\ueps$ for arbitrary timescales $\mathcal{O}(\eps^{-\alpha})$.
The effective equations, derived by generalizing the technique introduced for timescales $\mathcal{O}(\eps^{-2})$ in \cite{AbP16}, have the form 
\footnote{In the whole paper we use the shorthand $q\nabxn{n} v$ to denote the operator $q_{i_1\cdotss i_n} \pa^n_{i_1\cdotss i_n}v$, see \eqref{eq:notationDiffOperator_alpha}.}
\begin{equation}	\label{eq:effectiveEquation_res_alpha_intro}
\dtt\tildeub - a^0\nabxn 2\tildeub 
-\displaystyle\sum_{r=1}^{\lfloor \alpha/2\rfloor}(-1)^r\eps^{2r}
\big(a^{2r}\nabxn{2r+2}\tildeub 
- b^{2r}\nabxn{2r}\dtt\tildeub\big) 
=	f+\displaystyle\sum_{r=1}^{\lfloor \alpha/2\rfloor}(-1)^r\eps^{2r} b^{2r} \nabxn{2r} f
\quad\text{in }(0,\eps^{-\alpha}T]\times\Omega,
\end{equation}
where $a^0$ is the homogenized tensor and $a^{2r},b^{2r}$ are pairs of nonnegative, symmetric tensors of order $2r+2$ and $2r$, respectively, which satisfy constraints based on high order correctors, solutions of cell problems.
Note that the correction of the right-hand side generalizes the one introduced in the case $\alpha=2$ in \cite{ABV15} and discussed in \cite{AbP16}.
For all effective solutions $\tildeub$ in the family, we prove an error estimate that ensures $\tildeub$ to be $\eps$ close to $\ueps$ in the $\Li(0,\eps^{-\alpha}T;\Wnorm)$ norm (see \eqref{eq:definitionNormW}).
In contrast to the effective equations proposed in \cite{ALR18} and \cite{BeG17}, the well-posedness of \eqref{eq:effectiveEquation_res_alpha_intro} does not rely on regularization but is naturally ensured by the non-negativity of the tensors.
The unregularized versions of the {effective} equations from \cite{ALR18} and \cite{BeG17}
do not belong to the family {\eqref{eq:effectiveEquation_res_alpha_intro}} but are closely related:
they have the form \eqref{eq:effectiveEquation_res_alpha_intro} with 
$a^{2r}=g^{2r}$, $b^{2r}=0$ {and} these pairs of tensors satisfy the correct constraints. {The issue is however that} the sign of $g^{2r}$ happens to be negative for some $r$ (this is proved for $g^{2}$ \cite{COV06}) which results in the ill-posedness of the {unregularized version of these effective equations}.

The second main result of the paper is an explicit procedure for the computation of the high order effective tensors $\{a^{2r},b^{2r}\}$ in \eqref{eq:effectiveEquation_res_alpha_intro}, for which we provide a new relation between the high order correctors.
In particular, while the natural formula to compute $a^{2r},b^{2r}$ requires to solve the cell problems of order $1$ to $2r+1$, this relation ensures that only the cell problems of order $1$ to $r+1$ are in fact necessary.
The consequence is a significant reduction of the computational cost needed to compute the effective tensors of arbitrary order.
We emphasize that this result can also directly be used to reduce the computational cost for the tensors of the effective equations from \cite{ALR18} and \cite{BeG17}.

The paper is organized as follows.
In Section \ref{sec:mainresult}, we present our first main result: we derive the family of effective equations and state the error estimate. We then compare the obtained effective equations with the ones from \cite{ALR18} and \cite{BeG17}.
In Section \ref{sec:secondMainResult}, we construct a numerical procedure to compute the tensors of effective equations. 
We then present our second main result: a relation between the correctors which allows to reduce the computational cost of the effective tensors.
In Section \ref{sec:numexp_alpha}, we illustrate our theoretical findings in various numerical experiments.
Finally, in Section \ref{sec:proofs} we provide the proofs of the main results.


\subsection*{Definitions and notation}

Let us start by introducing some definitions and notations used in the paper.
Let $\Huper(\Om)$ be the closure of the space $\fcnC^\infty_\per(\Om)$ for the $\Hu$ norm.
We denote the quotient spaces $\calLd\ofOm = \Ld\ofOm/\R$
and $\calWper\of\Om = \Huper(\Om)/\R$.
The space $\Wper\of\Om$ (resp. $\LdzOm$) is composed of the zero mean representatives of the equivalence classes in $\calWper\of\Om$ (resp. $\calLd\ofOm$).
The dual space of $\Wper\of\Om$ (resp. $\calWper\of\Om$) is denoted $\WperStar\of\Om$ (resp. $\calWperStar\of\Om$).
The integral mean of $v\in\Lu(\Om)$ is denoted $\intMean{v}_\Omega = \frac{1}{|\Omega|} \int_\Omega v$
and $\ps{\cdot,\cdot}_\Omega$ denotes the standard inner product in $\Ld(\Omega)$.
We define the following norm on $\Wper\of\Om$
\begin{equation}	\label{eq:definitionNormW}
\norm{w}_{\Wnorm} 
= \inf_{\substack{w=w_1+w_2\\w_1,w_2\in\Wper\of\Om}}
\Big\{ \norm{w_1}_{\Ld\of\Om} + \norm{\nabla w_2}_{\Ld\of\Om} \Big\}
\quad \forall w\in\Wper\of\Om.
\end{equation}
Using the Poincar\'e--Wirtinger inequality, we verify that $\norm{\cdot}_{\Wnorm}$ is equivalent to the $\Ld$ norm:
$\norm{w}_{\Wnorm} \leq \norm{w}_{\Ld\of\Om} \leq \max\{1,C_\Om\} \norm{w}_{\Wnorm}$ 
where $C_\Om$ is the Poincar\'e constant.

We denote $\Ten^n(\R^d)$ the vector space of tensors of order $n$. 
In the whole text, we drop the notation of the sum symbol for the dot product between two tensors and use the convention that repeated indices are summed.
The subspace of $\Ten^n(\R^d)$ of symmetric tensors is denoted $\Sym^n(\R^d)$, i.e., $\genTen\in\Sym^n(\R^d)$ iff
$\genTen_{i_1\cdots i_n} = \genTen_{i_{\sigma(1)}\cdots i_{\sigma(n)}}$ for any permutation of order $n$ $\sigma\in\mathbb{S}_n$.
We define the \textit{symmetrization operator} $S^n:\Ten^n(\R^d)\to\Sym^n(\R^d)$ as 
\begin{equation}	\label{eq:definitinosymmetrizationoperator}
\big(S^n(\genTen)\big)_{i_1\cdots i_n} = \frac{1}{n!} \sum_{\sigma\in \mathbb{S}_n} \genTen_{i_{\sigma(1)}\cdots i_{\sigma(n)}}.
\end{equation}
The coordinate $\big(S^n(\genTen)\big)_{i_1\cdots i_n}$ is denoted $S^n_{i_1\cdots i_n}\{\genTen_{i_1\cdots i_n}\}$.
We denote $=_S$ an equality holding up to symmetries, i.e., for $p,\genTen\in\Ten^n(\R^d)$ 
we have
\begin{equation}	\label{eq:equalityS}
p=_S\genTen  \qquad \Leftrightarrow \qquad S^n(p) = S^n(\genTen).
\end{equation}
A colon is used to denote the inner product of two tensors in $\Ten^{n}(\R^d)$,
$p:\genTen = p_{i_1\cdots i_n}\genTen_{i_1\cdots i_n}$.
We say that a tensor $q\in\Ten^{2n}(\R^d)$ is \textit{major symmetric} if it satisfies
\begin{equation}	\label{eq:majorSymDef_alpha}
\genTen_{i_1\cdots i_{n}i_{n+1}\cdots i_{2n}} = \genTen_{i_{n+1}\cdots i_{2n}i_1\cdots i_{n}}
\quad
1\leq {i_1\cdots i_{2n}} \leq d.
\end{equation}
We say that a tensor $\genTen\in\Ten^{2n}(\R^d)$ is \textit{positive semidefinite} if
\begin{equation}	\label{eq:tensorPositiveSemidef_alpha}
\genTen\xi:\xi
=
\genTen_{i_1\cdots i_{2n}} \xi_{i_1\cdots i_n} \xi_{i_{n+1}\cdots i_{2n}} \geq 0\quad
\forall \xi\in\Sym^n(\R^d),
\end{equation}
and it is \textit{positive definite} if 
the equality in \eqref{eq:tensorPositiveSemidef_alpha} holds only for $\xi=0$.
The tensor product of $p\in\Ten^m(\R^d)$ and $\genTen\in\Ten^n(\R^d)$ is 
the tensor of $\Ten^{m+n}(\R^d)$ defined as
$(p\otimes q)_{i_1\cdots i_{m+n}} = p_{i_1\cdots i_m} \genTen_{i_{m+1}\cdots i_{m+n}}$.
Note that up to symmetries the tensor product is commutative, i.e.,
$p\otimes \genTen =_S \genTen\otimes p$.
We use the shorthand notation 
\[
\otimes^s \genTen = \underbrace{\genTen\otimes \cdots\otimes \genTen}_{s \text{ times}}.
\]
The derivative with respect to the $i$-th space variable $x_i$ is denoted $\pa_{i}$ and the derivation with respect to any other variable is specified.
For $q\in\Ten^n(\R^d)$, we denote the differential operator 
\begin{equation}	\label{eq:notationDiffOperator_alpha}
\genTen\nabxn n \coloneqq \genTen_{i_1\cdots i_n}\pa^n_{i_1\cdots i_n}.
\end{equation}



\subsection*{Settings of the problem}

Recall assumption \eqref{eq:assumPeriodic}:  $\aeps(x) = a\big(\xsureps\big)$, where $y\mapsto a(y)$ is a $d\times d$ symmetric, $Y$-periodic tensor.
In addition, we assume that $a(y)$ is uniformly elliptic and bounded, i.e., there exists $\lambda,\Lambda>0$ such that
\begin{equation}	\label{eq:uniformEllipticity_lpd}
\lambda |\xi|^2 \leq a(y)\xi\cdot\xi \leq \Lambda |\xi|^2	\quad \forall\xi\in\R^d~ \forae y\in Y.
\end{equation}
Without loss of generality, let the reference cell be $Y=(0,\ell_1)\times\cdots\times(0,\ell_d)$.
We assume that the hypercube $\Omega=(\omega^l_{1},\omega^r_{1})\times\cdots\times (\omega^l_{d},\omega^r_{d})$ satisfies
\begin{equation}	\label{eq:assumptionOmega}
n_i = \frac{\omega^r_{i}-\omega^l_{i}}{\ell_i\eps} \in \Npos	\quad \forall i=1,\hdots,d.
\end{equation}
In particular, \eqref{eq:assumptionOmega} ensures that for a $Y$-periodic function $\gamma$, the map $x\mapsto\gamma\ofxsureps$ is $\Omega$-periodic ($\gamma$ is extended to $\R^d$ by periodicity).
Note that the integers $n_i$ in \eqref{eq:assumptionOmega} can be arbitrarily large. In particular $n_i$ can be of order $\mathcal{O}(\eps^{-\alpha})$.

Given an integer $\alpha\geq0$, we consider the wave equation: $\ueps:[0,\eps^{-\alpha}T]\times\Omega\rightarrow\R$ such that
\begin{equation}	\label{eq:waveEquationEpsilon_lpd}
\begin{array}{ll}
\dtt \ueps(t,x)-\nabx\cdot \big(a\big(\xsureps\big)\nabx\ueps(t,x)\big) = f(t,x) 
&\text{in }(0,\eps^{-\alpha} T]\times\Omega,\\[1pt]
x\mapsto\ueps(t,x)~~\Omega\text{-periodic}				&\text{in }[0,\eps^{-\alpha} T],\\[1pt]
\ueps(0,x)=\cini(x), \quad \dt\ueps(0,x) = \dini(x) &\text{in }\Omega,
\end{array}
\end{equation} 
where $\cini,\dini$ are given initial conditions and $f$ is a source.
The well-posedness of \eqref{eq:waveEquationEpsilon_lpd} is proved in \cite{LiM68,Eva98}: if $\cini\in\WperOm$, $\dini\in\LdzOm$, and $f\in\Ld\Lti\LdzOm$, then there exists a unique weak solution $\ueps\in\Li\Lti\WperOm$ with $\dt\ueps\in\Li\Lti\LdzOm$ and $\dtt\ueps\in\Ld\Lti\WperStarOm$.


\section{First main result: family of effective equations and a priori error estimate}	\label{sec:mainresult}

In this section, we present the family of effective equations and provide the corresponding a priori error estimate. 
In \ref{subsec:derivationFamily_subsec}, we derive the family in three steps:
(i) we discuss the ansatz on the form of the effective equations; 
(ii) using asymptotic expansion we derive the high order cell problems;
(iii) we obtain the constraints on the effective tensors by investigating the well-posedness of the cell problems.
In \ref{subsec:mainresult_subsec}, we define rigorously the family of effective equations and state the a priori error estimate.
Finally, in \ref{subsec:comparisonGloria_alpha} we compare the obtained equations with the other effective equations available in the literature. 
For the sake of readability we postpone the technical proofs to Section \ref{sec:proofs}.

\subsection{Derivation of the family of effective equations}	\label{subsec:derivationFamily_subsec}

In the whole derivation, we assume that the data are as regular as necessary. The specific requirements are stated in Theorem \ref{thm:est_epsm_alpha}.
Note that we consider here timescales $\eps^{-\alpha}T$ with $\alpha\geq 2$.
For timescales  $\eps^{-\alpha}T$ with $\alpha<2$ it can be shown following similar techniques that the standard homogenized equation is a valid effective model (see \cite[section \chapalpha.1.1]{Pou17}).

\subsubsection*{Ansatz on the form of the effective equations}
We first discuss the ansatz on the form of the effective equations, which has a major importance in the derivation.
We assume that the effective equations have the form
\begin{equation}	\label{eq:effectiveEquation_alpha}
\dtt\tildeub  - a^0\nabxn 2\tildeub 
-\displaystyle\sum_{r=1}^{\lfloor\alpha/2\rfloor}(-1)^r\eps^{2r}
\Big(a^{2r}\nabxn {2r+2}\tildeub 
- b^{2r}\nabxn {2r}\dtt\tildeub\Big)
= Qf
\quad\text{in }(0,\eps^{-\alpha}T]\times\Omega,
\end{equation}
where $a^0$ is the homogenized tensor \eqref{eq:defhomogenizedTensor_alpha},
$a^{2r}\in\Ten^{2r+2}(\R^d)$, $b^{2r}\in\Ten^{2r}(\R^d)$ are tensors to be defined
and $Q$ is a differential operator to be defined (the construction of $Qf$ is discussed in Remark \ref{rem:explanationQf2}).
As discussed in \cite{AbP16}, if the set of considered equations is too small, we end up with ill-posed equations.
In particular, without the operators $-b^{2r}\nabxn {2r}\dtt$ in \eqref{eq:effectiveEquation_alpha}, our derivation would lead to the same ill-posed effective equations obtained in \cite{ALR18} and \cite{BeG17} (the unregularized versions).

Following the classical Faedo--Galerkin method (see \cite[Chap. 7]{Eva98}), we prove the following well-posedness result for \eqref{eq:effectiveEquation_alpha}.
We define the bilinear forms
\begin{align*}
\psbig{v,w}_\calH &= 
\psbig{v,w}_\Ld 
+ \sum_{r=1}^{\lfloor\alpha/2\rfloor} \eps^{2r} 
\psbig{b^{2r}_{i_1\cdots i_{2r}}\pa^{r}_{i_1\cdots i_{r}} v, \pa^{r}_{i_{r+1}\cdots i_{2r}}w}_\Ld ,	
\\
A\big( v,w\big) &=
\psbig{ a^0\nabla v,\nabla w}_\Ld
+ \sum_{r=1}^{\lfloor\alpha/2\rfloor} \eps^{2r} 
\psbig{a^{2r}_{i_1\cdots i_{2r+2}}\pa^{r+1}_{i_1\cdots i_{r+1}} v, \pa^{r+1}_{i_{r+2}\cdots i_{2r+2}}w}_\Ld,
\end{align*}
and the associated Banach spaces
\[
\calH = \big\{ v\in\LdzOm\,:\,
\ps{v,v}_\calH  <\infty \big\},
\qquad
\calV = \big\{ v\in\WperOm\,:\,
\ps{v,v}_\calV  <\infty \big\}.
\]
We call a function $\tildeub\in\Li\Lti{\calV}$, with $\dt\tildeub\in\Li\Lti{\calH}$, a weak solution of \eqref{eq:effectiveEquation_alpha} if for all test functions $v \in\fcnC^2([0,\Teps];\calV)$, with $v(\Teps)=\dt v(\Teps)=0$,
$\tildeub$ satisfies
\begin{equation}	\label{eq:BoussinesqEquationWeakForm_app}
\begin{aligned}	
\middleintspec{0}{\Teps}{-10pt} 
\psbig{\tildeub(t),\dt^2 v(t)}_\calH + A\big(\tildeub(t),v(t)\big)\, \d t
= &\middleintspec{0}{\Teps}{-10pt} \psbig{Qf(t),v(t)}_\LdOm\, \d t
- \psbig{ \dini,v(0)}_\calH + \psbig{ \cini,\dt v(0)}_\calH.
\end{aligned}
\end{equation}

\begin{theorem}	\label{thm:wp}
	Assume that the tensors $a^{2r}\in\Ten^{2r+2}(\R^d)$, $b^{2r}\in\Ten^{2r}(\R^d)$ are positive semidefinite \eqref{eq:tensorPositiveSemidef_alpha} and satisfy the major symmetries \eqref{eq:majorSymDef_alpha}. 
	Furthermore, assume that the data satisfy the regularity
	\[
	\cini\in\WperOm \cap\fcnH^{\lfloor\alpha/2\rfloor+1}(\Omega),
	\quad
	\dini\in\LdzOm \cap\fcnH^{\lfloor\alpha/2\rfloor}(\Omega),
	\quad
	Qf\in\Ld(0,\Teps;\LdzOm),
	\]
	Then there exists a unique weak solution of 
	\eqref{eq:effectiveEquation_alpha}.
\end{theorem}

Let us provide a short sketch of the proof.
We look for successive approximations of a weak solution in the form $u^m(t) = \sum_{k=0}^m u^m_k(t)\varphi_k$, where
$\{\varphi_k\}_{k\in\N}$ is a smooth basis of $\WperOm$.
For each $m$, $u^m(t)$ is obtained as the solution of a well-posed ordinary differential equation.
We then prove that the sequence $\{u^m\}_{m\geq0}$ is bounded in $\Li\Lti{\calV}$.
In particular, note that the sign assumptions on the tensors ensure that $\ps{v,v}_\calH \geq \norm{v}_\Ld^2$ for any $v\in\calH$ and 
and $A(v,v)\geq \lambda \norm{\nabla v}_\Ld^2$ for all $v\in\calV$.
We thus obtain the existence of a subsequence that weakly$^*$ converges in $\Li\Lti{\calV}$.
We can then prove that the weak$^*$ limit is the unique weak solution.

%
%

\subsubsection*{Asymptotic expansion, inductive Boussinesq tricks}
We make the ansatz that $\ueps$ can be approximated by an adaptation of $\tildeub$ of the form
\begin{equation}	\label{eq:adaptation1_alpha} 
\Beps\tildeub(t,x) 
= 	\tildeub(t,x) 	+ \sum_{k=1}^{\alpha+2} \eps^k u^k\big(t,x,\xsureps\big),
\end{equation}
where $u^k$ are to be defined and the map $y\mapsto u^k(t,x,y)$ is $Y$-periodic.  
We split the error as
\begin{equation*} 
\norm{\ueps - \tildeub}_{\Li(0,\eps^{-\alpha}T;\Wnorm)}
\leq
\norm{\ueps - \Beps\tildeub}_{\Li(0,\eps^{-\alpha}T;\Wnorm)}
+ \norm{\Beps\tildeub-\tildeub }_{\Li(0,\eps^{-\alpha}T;\Wnorm)},
\end{equation*}
and follow the argument presented in \cite{AbP16} based on the error estimate of Lemma \ref{lem:errorestimateindepoftimeanddomain} (see also \cite[Section 4.2.2]{Pou17}):
for $\norm{\ueps - \Beps\tildeub}_{\Li(0,\eps^{-\alpha}T;\Wnorm)}$ to be of order $\mathcal{O}(\eps)$, we need the terms involving $\tildeub$ in the remainder 
\begin{equation}	\label{eq:defSmallreps}
r^\eps=(\dtt + \Aeps) (\Beps\tildeub -\ueps),
\qquad
\Aeps v = -\nabx\cdot\big(a\ofxsureps\nabx v\big),
\end{equation}
to be of order $\mathcal{O}(\eps^{\alpha+1})$ in the $\Li(0,\eps^{-\alpha}T;\WperStarOm)$ norm
(it is sufficient that this holds in the stronger $\Li(0,\eps^{-\alpha}T;\LdOm)$ norm).
We now expand $r^\eps$:
using the equation for $\ueps$ \eqref{eq:waveEquationEpsilon_lpd} and the form of the adaptation 
\eqref{eq:adaptation1_alpha}, we obtain
\begin{equation}	\label{eq:waveAsymptoticDeveloppment_alpha}
\begin{aligned}
r^\eps =\,
&\phantom{+}\,\eps^{-1} \big(
\phantom{\dtt\tildeub}
\phantom{{\,+\,}}
{\calA_{yy}u^1}
{\,+\,}{\calA_{xy}\tildeub}
\phantom{^1\,+\,}\phantom{\calA_{xx}\tildeub}\phantom{\,-\,f}	\,
\big)\\
&{+}\,\eps^{0\phantom{-}} \big(	
{\dtt\tildeub}
{\,+\,}{\calA_{yy}u^2}
{\,+\,}{\calA_{xy}u^1} 
{\,+\,}{\calA_{xx}\tildeub}	
{\,-\,f\,}
\big) \\
&{+} \sum_{k=1}^\alpha \eps^k\big(	\dtt u^k + \calA_{yy}u^{k+2} + \calA_{xy}u^{k+1} + \calA_{xx}u^{k}\, \big)
~~+~~\mathcal{R}^\eps_{\mathrm{ini}},
\end{aligned}
\end{equation}
where the operators $\calA_{yy},\calA_{xy},\calA_{xx}$ are defined as
\begin{eqnarray*}
	\calA_{yy} v= -\naby\cdot\big( a(y)\naby v\big),
	\quad
	\calA_{xy} v= -\naby\cdot\big( a(y)\nabla_x v\big) -\nabla_x\cdot\big( a(y)\naby v\big),
	\quad
	\calA_{xx} v= -\nabla_x\cdot\big(a(y)\nabla_x v\big),
\end{eqnarray*}
and the remainder is
\begin{equation} \label{eq:Repsini}
\mathcal{R}^\eps_{\mathrm{ini}}
= 
\,\eps^{\alpha+1} 	 \big(  \dtt u^{\alpha+1} + \calA_{xy}u^{\alpha+2} +\calA_{xx}u^{\alpha+1} 	\big)
\,{+}\,\,\eps^{\alpha+2} \big(  \dtt u^{\alpha+2} + \calA_{xx}u^{\alpha+2}\big).
\end{equation}
Classical two scale asymptotic expansion \cite{BLP78} advises to look for $u^k$ of the form
\begin{equation}	\label{eq:adaptation2_alpha} 
u^k(t,x,y) = \chi^k(y) \nabxn k\tildeub(t,x)
= \chi^k_{i_1\cdotss i_k}(y) \pa^k_{i_1\cdotss i_k}\tildeub(t,x),
\end{equation}
where the components of the tensor $\chi^k$ are $Y$-periodic functions to be defined.$  $

The next step is the main difficulty of the derivation: we must use the effective equation \eqref{eq:effectiveEquation_alpha} to substitute all the time derivatives in the terms of order $\mathcal{O}(\eps^0)$ to $\mathcal{O}(\eps^{\alpha})$ in \eqref{eq:waveAsymptoticDeveloppment_alpha}.
Various versions of this process have been used in related works \cite{AbP16,AbP16b,ABV15,DLS15,FCN02a,Lam11} to obtain well-posed effective equations from ill-posed ones.	To refer to this kind of manipulation the term \textit{Boussinesq trick} was coined in \cite{ABV15} (see \cite{CMV96} where such tricks are used for several Boussinesq's type of equations).

As these \textit{inductive Boussinesq tricks} represent a technical challenge,
let us explain here the case $\alpha=2$ and $f=0$
and postpone the general case to Section \ref{subsec:proofbousstrick} (Theorem \ref{thm:rewritedtttildeub_2_alpha} and Lemma \ref{lem:inductiveBoussinesqTricks_alpha}).
For $\alpha=2$ and $f=0$, the effective equation \eqref{eq:effectiveEquation_alpha} can be written as
\[
\dtt\tildeub  =
a^0\nabxn 2\tildeub -\eps^{2}\Big(a^{2}\nabxn{4}\tildeub - b^{2}\nabxn{2}\dtt\tildeub\Big).
\]
We now use the equation to substitute $\dtt\tildeub$ in the last term and obtain
\begin{align*}
\dtt\tildeub  =
a^0\nabxn 2\tildeub -\eps^{2}\Cten^1\nabxn{4}\tildeub
+ \mathcal{R}_2^\eps\tildeub ,
\qquad
\Cten^1 = a^2 -a^0\otimes b^2,
\quad
\mathcal{R}_2^\eps\tildeub 
= \eps^4 \big( a^2\otimes b^2 \nabxn{6}\tildeub - b^2\otimes b^2\nabxn{4}\dtt\tildeub\big).
\end{align*}
Using \eqref{eq:adaptation2_alpha} and the two last equalities, we find that the time derivatives in \eqref{eq:waveAsymptoticDeveloppment_alpha} can be written as
\begin{align*}
\dtt\tildeub + \eps\dtt u^1 + \eps^2 \dtt u^2
&= \dtt\tildeub + \eps\chi^1\nabx\dtt\tildeub + \eps^2 \chi^2\nabxn{2}\dtt\tildeub
\\&= a^0\nabxn{2}\tildeub 
+ \eps a^0\otimes\chi^1\nabxn{3}\tildeub
+ \eps^2 \big(-\Cten^1 + a^0\otimes\chi^2\big)\nabxn{4}\tildeub
+ \mathcal{R}^\eps\tildeub,
\end{align*}
where the remainder $\mathcal{R}^\eps\tildeub$ has order $\mathcal{O}(\eps^3)$ in the $\Li(0,\eps^{-\alpha}T;\LdOm)$ norm (for $\tildeub$ sufficiently regular).
Using this expression in \eqref{eq:waveAsymptoticDeveloppment_alpha}, we obtain the desired development in the case $\alpha=2$ and $f=0$.

This process is generalized in Theorem \ref{thm:rewritedtttildeub_2_alpha} and Lemma \ref{lem:inductiveBoussinesqTricks_alpha}.
With these results, we are able to rewrite $r^\eps$ in \eqref{eq:waveAsymptoticDeveloppment_alpha} without time derivative in the terms of order $\mathcal{O}(\eps^0)$ to $\mathcal{O}(\eps^{\alpha})$:
defining the tensors $\Cten^r\in\Ten^{2r+2}(\R^d)$ inductively and $\Pten^{k}\in\Ten^{k+2}(\R^d)$ as
\begin{equation}	\label{eq:definitionCrPk_first}
\begin{gathered}
\Cten^0 = a^0,
\qquad
\Cten^r = a^{2r} - \sum_{\ell=0}^{r-1} \Cten^\ell \otimes b^{2(r-\ell)} 
\quad 1\leq r\leq \lfloor\alpha/2\rfloor,
\\
\Pten^{2r} = (-1)^r c^r,
\qquad
\Pten^{2r+1} = 0,
\qquad
0\leq r\leq \lfloor\alpha/2\rfloor,
\end{gathered}
\end{equation}
we obtain
\begin{equation}	\label{eq:waveAsymptoticDeveloppment_2_alpha} 
\begin{aligned}
r^\eps=
&\phantom{+}\,\eps^{-1} \big( \calA_{yy}u^1 + \calA_{xy}\tildeub \big)\\
&+ \sum_{k=0}^{\alpha}	\eps^{k}
\Big(\calA_{yy}u^{k+2} + \calA_{xy}u^{k+1} + \calA_{xx}u^{k} 
+	\textstyle\sum\limits_{j=0}^k \Pten^{j}\otimes\chi^{k} 
\Big) \nabxn{k+2}\tildeub\\
& +\mathcal{S}^\eps f + \mathcal{R}^\eps_{\mathrm{ini}}\tildeub + \mathcal{R}^\eps \tildeub.
\end{aligned}
\end{equation}
Provided sufficient regularity of $\tildeub$, the remainder $\mathcal{R}^\eps \tildeub$ has order $\mathcal{O}(\eps^{\alpha+1})$ in the $\Li(0,\eps^{-\alpha}T;\LdOm)$ norm.
Furthermore, provided sufficient regularity of $f$, the remainder $\mathcal{S}^\eps f$ has order $\mathcal{O}(\eps)$ in the $\Li(0,\eps^{-\alpha}T;\LdOm)$ norm (this is sufficient for $\tildeub$ and $\ueps$ to be $\eps$ close, see Theorem \ref{thm:est_epsm_alpha}).  
Using the definition of $u^k$ in \eqref{eq:adaptation2_alpha} and of $\calA_{yy},\calA_{xy},\calA_{xx}$, we verify that (for $k=0$ use $u^0 = \tildeub$ and $\chi^0=1$)
\begin{align*}
\calA_{yy}	u^{k+2}	&= -\naby \cdot\big( a\naby \chi^{k+2}_{i_1\cdotss i_{k+2}} \big) \pa^{k+2}_{i_1\cdotss i_{k+2}} \tildeub,\\
\calA_{xy}	u^{k+1}	&= \Big(
-\naby \cdot\big( ae_{i_1}\chi^{k+1}_{i_2\cdotss i_{k+2}} \big) 
- e_{i_1}^T a \naby \chi^{k+1}_{i_2\cdotss i_{k+2}}
\Big) \pa^{k+2}_{i_1\cdotss i_{k+2}} \tildeub,\\
\calA_{xx}	u^k	&= - e_{i_1}^Tae_{i_2}\chi^{k}_{i_3\cdotss i_{k+2}}  \pa^{k+2}_{i_1\cdotss i_{k+2}} \tildeub,
\end{align*}
where $e_i$ denotes the $i$-th vector of the canonical basis of $\R^d$.
Hence, canceling successively the terms of order $\calO(\eps^{-1})$ to $\calO(\eps^{\alpha})$ in \eqref{eq:waveAsymptoticDeveloppment_2_alpha}, 
we obtain the \textit{cell problems}:
the correctors $\{\chi^{k}_{i_1\cdotss i_{k}}\}_{k=1}^{\alpha+2}$ are the functions in $\WperY$ such that 
\begin{subequations}		\label{eq:cellProblemsWave_compact}
	\begin{align}
	\psbig{ a\naby \chi^1_i,\naby w}_Y = \,&-\psbig{ a e_i, \naby w}_Y, 
	\qquad\forall w\in\WperY,
	\label{eq:cellProblemsWave_compact_1}\\
	\begin{split}	\label{eq:cellProblemsWave_compact_k}
	\psbig{ a\naby \chi^{k+1}_{i_1\cdotss i_{k+1}},\naby w}_Y 
	= \,&S^{k+1}_{i_1\cdotss i_{k+1}}\Big\{	
	- \psbig{ a e_{i_1}\chi^k_{i_2\cdotss i_{k+1}}, \naby w}_Y
	+ \psbig{ a (\naby\chi^k_{i_2\cdotss i_{k+1}} + e_{i_2}\chi^{k-1}_{i_3\cdotss i_{k+1}}),e_{i_1}w }_Y
	\\&\phantom{S^{k+1}_{i_1\cdotss i_{k+1}}\big\{}
	- \textstyle\sum\limits_{j=0}^{k-1} 	
	\psbig{ 
		\big(\Pten^{k-1-j}\otimes \chi^{j}\big)_{i_1\cdotss i_{k+1}}
		, w}_Y  
	\Big\}
	\qquad\forall w\in\WperY, 
	\end{split}
	\end{align}
\end{subequations}
where the tensors $\Pten^k$ are defined in \eqref{eq:definitionCrPk_first}
and the symmetrization operator is defined in \eqref{eq:definitinosymmetrizationoperator}.
Note that the cell problems for $\chi^1_i$ and $\chi^2_{ij}$ are the same as {the cell problems} obtained with two-scale asymptotic expansion \cite{BLP78}.

\begin{remark}	\label{rem:correctorsAreSymmetric}
	In \eqref{eq:cellProblemsWave_compact}, we have chosen symmetric right-hand sides.
	This is possible thanks to the symmetry of $\nabxn{n}\tildeub$ in \eqref{eq:waveAsymptoticDeveloppment_2_alpha}, as a term of the form $q\nabxn {n}\tildeub$ can be rewritten as $S^n(q)\nabxn {n}\tildeub$.
	This choice ensures that the correctors are symmetric tensors functions.
	In particular, $\chi^k$ has only $\binom{k+d-1}{k}$ distinct entries instead of $d^k$ if it was not symmetric. 
	When it comes to numerical approximation, each distinct entry of the corrector corresponds to a cell problem to solve and this symmetrization saves computational time.
	In addition, this choice ensures that the odd order cell problems are well-posed unconditionally (see below).
\end{remark}

\subsubsection*{Constraints on the high order effective tensors}

The last step in the derivation of the family of effective equations is to obtain the constraints on the effective tensors by imposing the well-posedness of the cell problems.
In order to investigate the solvability of \eqref{eq:cellProblemsWave_compact}, 
let us state the following classical result
(obtained with the Fredholm alternative or the Lax--Milgram theorem combined with the characterization of $\WperStarY$).
\begin{lemma}	\label{lem:solvcond}
	For an elliptic and bounded tensor $a(y)$, consider the following variational problem:
	find $v\in\WperY$ such that
	\begin{equation}	\label{eq:genericEllipticEquation}
	\psbig{a\naby v, \naby w}_Y = \psbig{f^1,\naby w}_{Y} + \psbig{f^0,w}_{Y}
	\quad\forall w\in\WperY,
	\end{equation}
	where $f^0,f^1_1,\hdots,f^1_d$ are given functions.
	Then \eqref{eq:genericEllipticEquation} has a unique solution $v\in\WperY$
	if and only if
	\begin{equation}	\label{eq:solvabilityConditionCalWper_ch5}
	f^0,f^1_1,\hdots,f^1_d\in\Ld\of Y
	\quad\text{and}\quad
	\psbig{f^0,1}_{Y} = 0.
	\end{equation}
\end{lemma}

We now proceed to the two following tasks:
first, we verify that the odd order cell problems satisfy unconditionally the solvability condition \eqref{eq:solvabilityConditionCalWper_ch5};
second, we impose the solvability condition \eqref{eq:solvabilityConditionCalWper_ch5} on the right-hand sides of the even order cell problems to obtain the constraints on the effective tensors.

First note that \eqref{eq:cellProblemsWave_compact_1} is well-posed as its right-hand side unconditionally satisfies \eqref{eq:solvabilityConditionCalWper_ch5}.
Next, we verify that the cell problem for $\chi^2$ is well-posed as the homogenized tensor satisfies
\begin{equation}	\label{eq:defhomogenizedTensor_alpha}
\Pten^0 _{ij} 
= a^0_{ij} 
= \frac{1}{|Y|}\psbig{a(\naby\chi^1_{j}+e_j),e_i}_{Y}
= \intMeanbig{e_i^Ta(\naby\chi^1_{j}+e_j)}_Y,
\end{equation}
which ensures the solvability condition \eqref{eq:solvabilityConditionCalWper_ch5} to hold.
We then continue this process to derive the constraints on the higher order effective tensors imposed by the well-posedness of the higher order cell problems.

We assume that the cell problems are well-posed up to order $2r$, $r\geq 1$.
Consider the cell problem for $\chi^{2r+1}$.
Recalling that $\Pten^{2r-1}=0$ and as the correctors $\chi^{1},\hdots,\chi^{2r}$ have zero mean,
we verify that the solvability condition \eqref{eq:solvabilityConditionCalWper_ch5} 
is equivalent to the following relation between the correctors.
\begin{lemma}	\label{lem:lemmawpodd_alpha}
	For any $1\leq r \leq \lfloor \alpha/2\rfloor$, the correctors $\chi^{2r}$ and $\chi^{2r-1}$ satisfy the equality
	\[
	\Gten^{2r-1}_{i_1\cdotss i_{2r+1}}
	\coloneqq
	\psBig{ a \big(\naby\chi^{2r}_{i_2\cdotss i_{2r+1}} + e_{i_2} \chi^{2r-1}_{i_3\cdotss i_{2r+1}} \big) , e_{i_1} }_{Y} =_S 0.
	\]
\end{lemma}
\begin{remark}	\label{rem:lemma_odd}
	A proof of Lemma \ref{lem:lemmawpodd_alpha} can be found in \cite[Lemma 5.2.5]{Pou17}.
	A similar result is also known in the context of Bloch wave theory (see e.g. \cite{DLS14a}, \cite{BeG17} and the references therein). 
	As discussed, Lemma \ref{lem:lemmawpodd_alpha} guarantees that the odd order cell problems are well-posed unconditionally.
	Pursuing the reasoning, this result ensures that no operator of odd order is needed in the effective equations.
	It is also the reason why no additional correction is required in the effective equation 
	when increasing the timescale from an even integer $\alpha$ to $\alpha+1$.
\end{remark}

Finally, consider the cell problem for $\chi^{2r+2}$.
Imposing the solvability condition \eqref{eq:solvabilityConditionCalWper_ch5} on the right-hand side, we obtain the following constraint on the tensor $\Pten^{2r}$:
\begin{equation}	\label{eq:constaintcr_alpha}
\Pten^{2r} =_S \Gten^{2r},
\qquad \Gten^{2r}_{i_1\cdotss i_{2r+2}}
\coloneqq \frac{1}{|Y|}\psBig{ a \big(\naby\chi^{2r+1}_{i_2\cdotss i_{2r+2}} + e_{i_2} \chi^{2r}_{i_2\cdotss i_{2r+2}} \big) , e_{i_1}}_{Y}.
\end{equation}
This constraint can be rewritten in terms of the effective tensors $a^{2r},b^{2r}$ 
using the definition of $\Cten^r$ and $\Pten^{2r}$ in \eqref{eq:definitionCrPk_first}:
\begin{equation}	\label{eq:constainta2rb2r_alpha}
a^{2} - b^{2}\otimes a^0  =_S - \Gten^{2},
\qquad
a^{2r} - b^{2r}\otimes a^0  =_S (-1)^r \Gten^{2r} + \sum_{\ell=1}^{r-1} \Cten^\ell \otimes b^{2(r-\ell)}  
\qquad
2\leq r \leq \lfloor \alpha/2\rfloor.
\end{equation}

The constraints \eqref{eq:constainta2rb2r_alpha} characterize the family of effective equations.
Indeed, if the effective tensors $a^{2r},b^{2r}$ satisfy \eqref{eq:constainta2rb2r_alpha}, the cell problems for $\chi^1$ to $\chi^{\alpha+2}$ are well-posed because their right-hand sides satisfy the solvability condition \eqref{eq:solvabilityConditionCalWper_ch5}.
Therefore, the adaptation $\Beps\tildeub$ given by \eqref{eq:adaptation1_alpha} and \eqref{eq:adaptation2_alpha} is well defined and can be used to prove an error estimate between $\ueps -\tildeub$.
This result is presented in the next section and rigorously proved in Section \ref{subsec:summaryproof}.



\subsection{A priori error estimate for the family of effective equations}	\label{subsec:mainresult_subsec}

In the previous section, we derived the constraints on the effective tensors for the effective equations to approximate $\ueps$. 
We present here an error estimate that ensures that the solutions of the derived effective equations are $\eps$ close to $\ueps$ in the $\Li(0,\eps^{-\alpha}T;\Wnorm)$ norm.

Let us first rigorously define the family of effective equations derived in Section \ref{subsec:derivationFamily_subsec}.
\begin{definition}	\label{def:familyofeffectivesolution_alpha}
	The family $\famEfEq$ of effective equations is the set of equations \eqref{eq:effectiveEquation_alpha}, where
	$Qf$ is defined as
	\begin{equation}	\label{eq:defQf_expl}
	Qf 	= f + \sum_{r=1}^{\lfloor \alpha/2 \rfloor}(-1)^r\eps^{2r} b^{2r} \nabxn{2r} f,
	\end{equation}
	and for $1\leq r \leq \lfloor\alpha/2\rfloor$ the tensors $a^{2r},b^{2r}$ satisfy the following requirements:
	\begin{enumerate}[label=\textit{(\roman{*})}]
		\item	\label{en:defr1}	$a^{2r}\in\Ten^{2r+2}(\R^d)$, $b^{2r}\in\Ten^{2r}(\R^d)$;
		\item	\label{en:defr2}	$a^{2r}$ and $b^{2r}$ are positive semidefinite  (see \eqref{eq:tensorPositiveSemidef_alpha});
		\item	\label{en:defr3}	$a^{2r}$ and $b^{2r}$ satisfy the major symmetries \eqref{eq:majorSymDef_alpha};
		\item	\label{en:defr4}	$a^{2r}$ and $b^{2r}$ satisfy the constraints \eqref{eq:constainta2rb2r_alpha}.
	\end{enumerate}
\end{definition}
\begin{remark}	\label{rem:explanationQf1}
	In the case $\alpha=2$, the correction $Qf$ of the right-hand side in the effective equations was introduced in \cite{ABV15} and discussed in \cite{AbP16}.
	We verify that the definition of $Qf$ in \eqref{eq:defQf_expl} leads to a lower constant multiplying $\norm{f}_{\Lu(0,\eps^{-\alpha}T;\HkOm{r(\alpha)})}$ in estimate \eqref{eq:errorEstimateThm}.
	More details on the origin of \eqref{eq:defQf_expl} are given in Remarks \ref{rem:explanationQf2}
	and \ref{rem:explanationQf4}.
\end{remark}
For the effective solutions in family $\famEfEq$, we prove in Section \ref{subsec:summaryproof} the following error estimate.
\begin{theorem}		\label{thm:est_epsm_alpha}
	Let $d\leq 3$, $\alpha\geq2$ and let $\tildeub$ belong to the family of effective equations $\famEfEq$ (Definition \ref{def:familyofeffectivesolution_alpha}).
	Furthermore, assume that $a(y)\in\Li(Y)$ 
	and that the data and $\tildeub$ satisfy the following regularity
	\begin{gather*}
	\tildeub\in\Li(0,\eps^{-\alpha}T;\HkOm{r(\alpha)+2}),
	~~
	\dt\tildeub\in\Li(0,\eps^{-\alpha}T;\HkOm{r(\alpha)+1}),
	~~
	\dtt\tildeub\in\Li(0,\eps^{-\alpha}T;\HkOm{r(\alpha)}),
	\\
	\cini\in\HkOm{\alpha+2},
	~~
	\dini\in\HkOm{\alpha+2},
	~~
	f\in\Ld(0,\eps^{-\alpha}T;\HkOm{r(\alpha)}),
	\end{gather*}
	where $r(\alpha)=\alpha+2\lfloor\alpha/2\rfloor+2$.
	Then the following error estimate holds
	\begin{equation}	\label{eq:errorEstimateThm}
	\begin{aligned}
	\norm{\ueps-\tildeub}_{\Li(0,\eps^{-\alpha}T;\Wnorm)}
	\leq C \eps\Big( &\,\norm{\dini}_{\HkOm{\alpha+2}} + \norm{\cini}_{\HkOm{\alpha+2}} + \norm{f}_{\Lu(0,\eps^{-\alpha}T;\HkOm{r(\alpha)})}
	\\&+ \textstyle\sum_{k=1}^{r(\alpha)+2}\seminorm{\tildeub}_{\Li(0,\eps^{-\alpha}T;\HkOm{k})} 
	+ \norm{\dtt \tildeub}_{\Li(0,\eps^{-\alpha}T;\HkOm{r(\alpha)})}\,\Big),	
	\end{aligned}
	\end{equation}
	where $C$ depends only on $T$, $\lambda$, $\Lambda$, 
	$\{|a^{2r}|_\infty,|b^{2r}|_\infty\}_{r=1}^{\lfloor \alpha/2\rfloor}$, and $Y$, 
	and the norm $\norm{{\cdot}}_\Wnorm$ is defined in \eqref{eq:definitionNormW}.
\end{theorem}

\begin{remark}
	For $d\geq 4$, the result of Theorem \ref{thm:est_epsm_alpha} holds provided a higher regularity on the effective solution and $f$ are assumed.
	Specifically, assuming that $m$ is a sufficiently large integer for the embedding 
	$\Hk{m}_\per(\Omega)\contEmb\fcnC^0(\Omega)$ to hold, the statement of Theorem \ref{thm:est_epsm_alpha} is true for $r(\alpha)=\alpha+2\lfloor\alpha/2\rfloor+m$. 
\end{remark}

\begin{remark}
	From the derivation in Section \ref{subsec:derivationFamily_subsec}, we may hope that $\Beps\tildeub$ is a better approximation of $\ueps$ than $\tildeub$.
	For example, in the elliptic case the error between $\ueps$ and an adaptation can be estimated in the energy norm 
	(see e.g. \cite{JKO94}, note that under \eqref{eq:assumptionOmega} there is no boundary layer).
	However, in the case of the wave equation, the answer is not as simple.
	In fact, an error estimate in the energy norm can be proved only in some specific settings: for example if $f=0$ and {for well prepared initial} 
	position $\cini$ \ma{of the form} $\Beps\bar u_0$ for some $\bar u_0$.
	Note that this issue is related to the lack of convergence of the energy of the fine scale wave toward the energy of the homogenized wave (see \cite{BFM92}).
\end{remark}


\subsection{Comparison with other effective equations in the literature}	\label{subsec:comparisonGloria_alpha}

In \cite{BeG17}, effective equations for arbitrary timescales are derived.
The settings are slightly different as the wave equation \eqref{eq:waveEquationEpsilon_lpd} is considered in the whole space $\R^d$ and with a tensor that can be of a more general nature: periodic, almost periodic, quasiperiodic and random (we refer to \cite{BeG17} for the specific definitions).
In these circumstances, \cite{BeG17} proposes effective equations that have the form
\begin{equation} \label{eq:veps_Gloria}
\dtt \bar{u}	- a^0 \nabxn{2}\bar{u}
+\displaystyle\sum_{r=0}^{\lfloor\alpha/2\rfloor} \eps^{2r} \Dten^{2r} \nabxn {2r+2}\bar{u}
- \epsilon^{2\lfloor{\alpha}/{2}\rfloor+2}R\bar{u}
= 0,
\qquad
R \bar{u}
=
\gamma(-1)^{\lfloor{\alpha}/{2}\rfloor+1} \mathrm{Id} \nabxn {2\lfloor{\alpha}/{2}\rfloor+4} \bar{u}.
\end{equation} 
The effective tensors in this equation are indeed verified to match the tensors $\Gten^{2r}$ defined in \eqref{eq:constaintcr_alpha} (see \cite[section 5.2.6]{Pou17}).
Under some weak regularity assumption on the initial condition, an error estimate between $\ueps$ and $\bar{u}$ is proved in the $\Li(0,\eps^{-\alpha}T;\Ld(\R^d))$ norm.
This estimate is a strong theoretical result as it holds in the norm we expect in the context of homogenization of the wave equation.
However, the use of the effective equation \eqref{eq:veps_Gloria} in practice is problematic
as no procedure for the computation of the regularization parameter $\gamma$ is available.
In fact, numerical tests indicate that the range of acceptable values for $\gamma$ is narrow: if $\gamma$ is too small, the equation is ill-posed and if $\gamma$ is too large, the solution $\bar{u}$ of \eqref{eq:veps_Gloria} does not describe $\ueps$ accurately (see \cite[section 5.4.3]{Pou17}).

In \cite{ALR18}, another effective equation for arbitrary timescales is proposed.
The settings of this result are the following: the wave equation \eqref{eq:waveEquationEpsilon_lpd} is considered in the whole space $\R^d$ with a periodic tensor, it includes an oscillating density and admits nonzero source and trivial initial conditions.
In the particular case of a constant density (as considered in the present paper), the effective equation proposed in \cite{ALR18} reads
\begin{equation} \label{eq:veps_AllaireLamacz}
\bigg(\dtt	
- a^0 \nabxn{2}
+\displaystyle\sum_{r=0}^{\lfloor\alpha/2\rfloor} \eps^{2r} \Dten^{2r} \nabxn {2r+2}\bigg)
S_2^\eps\hat{u}
= S_1^\eps f
\quad\text{in }(0,\eps^{-\alpha}T]\times\R^d,
\end{equation} 
where $S^\eps_1$, $S^\eps_2$ are filtering differential operators that ensure the equation to be well-posed.
The main result is an error estimate in energy norm between $\ueps$ and an adaptation of $\hat{u}$.
In particular, the results enable the adaptation to approximate $\ueps$ as accurately as one wants by increasing $\alpha$ accordingly.
Nevertheless, the filtering process used in \eqref{eq:veps_AllaireLamacz} to obtain a well-posed equation has yet to be tested in practice.

Despite the inherent difference between \eqref{eq:veps_Gloria} and \eqref{eq:veps_AllaireLamacz} and effective equations \eqref{eq:effectiveEquation_alpha} in family $\famEfEq$, their respective effective tensors are tied through the relation (see \eqref{eq:constainta2rb2r_alpha}):
\begin{equation}	\label{eq:relationbaraj_alpha}
\Dten^{2r}
=_S (-1)^r\Big( a^{2r} - \sum_{\ell=0}^{r-1} \Cten^\ell \otimes b^{2(r-\ell)} \Big)
\qquad 1\leq r \leq \lfloor\alpha/2\rfloor.
\end{equation}
Hence, if we let $b^{2r}=0$ for all $r$ in the effective equations in $\famEfEq$, relation \eqref{eq:relationbaraj_alpha} reads $\Dten^{2r} = (-1)^ra^{2r}$ and we end up with the unregularized versions of \eqref{eq:veps_Gloria} and \eqref{eq:veps_AllaireLamacz} (i.e., $\gamma=0$ and $S_1^\eps=S_2^\eps=\mathrm{Id}$).
Similarly, equations \eqref{eq:veps_Gloria} and \eqref{eq:veps_AllaireLamacz} without their
regularization devices satisfy requirement \ref{en:defr4} in Definition \ref{def:familyofeffectivesolution_alpha}.
However, while the well-posednesses of \eqref{eq:veps_Gloria} and \eqref{eq:veps_AllaireLamacz} rely on their respective regularization process,
the well-posedness of the effective equations in $\famEfEq$ is guaranteed by requirements \ref{en:defr2} and \ref{en:defr3}.
To fulfill these requirements, we have an explicit and constructive algorithm described in Section \ref{subsec:constructionEffectiveTensors}.

Finally, before approximating any of the effective equations \eqref{eq:effectiveEquation_alpha}, \eqref{eq:veps_Gloria} or \eqref{eq:veps_AllaireLamacz}, the effective tensors $\Gten^{2r}$ must be computed.
For this calculation, a substantial gain of computational time is obtained by using the formula provided by our second main result in Theorem \ref{thm:lemmawpeven_alpha}, Section \ref{subsec:algorithm_alpha}.


\section{Second main result: computation of effective tensors and reduction of the computational cost}	
\label{sec:secondMainResult}

In this section, we provide a numerical procedure for the computation of the tensors of some effective equations in the family $\famEfEq$.
In particular, in \ref{subsec:algorithm_alpha} we present a new relation between the correctors that allows to reduce significantly the computational cost for the effective tensors.
We emphasize that this result is not limited to the family $\famEfEq$ and can be used to compute the effective tensors of the alternative effective models from \cite{ALR18} and \cite{BeG17} discussed in \ref{subsec:comparisonGloria_alpha}.
The final algorithm is provided in \ref{subsec:algo}.


\subsection{Construction of high order effective tensors}	\label{subsec:constructionEffectiveTensors}

Recall that the tensors $\{a^{2r},b^{2r}\}_{r=1}^{\lfloor\alpha/2\rfloor}$ of an effective equation in family $\famEfEq$ are characterized by the following properties (Definition \ref{def:familyofeffectivesolution_alpha}):
for $1\leq r\leq \lfloor\alpha/2\rfloor$ 
\begin{enumerate}[label=\textit{(\roman{*})}]
	\item	$a^{2r}\in\Ten^{2r+2}(\R^d)$, $b^{2r}\in\Ten^{2r}(\R^d)$;
	\item	$a^{2r}$ and $b^{2r}$ are positive semidefinite, i.e., (see \eqref{eq:tensorPositiveSemidef_alpha})
	\[
	a^{2r}\xi:\xi
	\geq 0\quad
	\forall \xi\in\Sym^{r+1}(\R^d),
	\qquad
	b^{2r}\xi:\xi
	\geq 0\quad
	\forall \xi\in\Sym^{r}(\R^d);
	\]
	\item 	$a^{2r}$ and $b^{2r}$ satisfy the major symmetries, i.e.,
	\[
	a^{2r}_{i_1\cdots i_{r+1}i_{r+2}\cdots i_{2r+2}} = a^{2r}_{i_{r+2}\cdots i_{2r+2}i_1\cdots i_{r+1}},
	\qquad
	b^{2r}_{i_1\cdots i_{r}i_{r+1}\cdots i_{2r}} = b^{2r}_{i_{r+1}\cdots i_{2r}i_1\cdots i_{r}};
	\]
	\item	 $a^{2r}$ and $b^{2r}$ satisfy the constraints 
	\begin{equation}	\label{eq:constrbis}
	a^{2r}-b^{2r}\otimes a^0 =_S  \check{q}^r,
	\end{equation}
	with the tensors $\check{q}^r\in\Sym^{2r+2}(\R^d)$ defined as
	\begin{equation}	\label{eq:defCheckqr}
	\check{q}^1 = S^4\big( -\Gten^2 \big) ,
	\qquad
	\check{q}^r = 
	S^{2r+2}\Big(
	(-1)^r \Gten^{2r} + \sum_{\ell=1}^{r-1} \Cten^\ell \otimes b^{2(r-\ell)} 
	\Big)
	\quad
	2\leq r\leq \lfloor\alpha/2\rfloor,
	\end{equation}
	where the tensor $\Cten^1\in\Ten^{4}(\R^d),\hdots, \Cten^{r-1}\in\Ten^{2r}(\R^d)$, defined in \eqref{eq:definitionCrPk_first},
	are computed with $a^0$, $\{a^{2s},b^{2s}\}_{s=1}^{r-1}$ 
	and the tensor $\Gten^{2r}\in\Ten^{2r+2}(\R^d)$, defined in \eqref{eq:constaintcr_alpha}, can be computed from the correctors $\chi^{2r}$ and $\chi^{2r+1}$ 
	(a much cheaper alternative is to compute only $S^{2r+2}(\Gten^{2r})$ with formula \eqref{eq:decompostionhr}, see Section \ref{subsec:algorithm_alpha}).
\end{enumerate}
Our goal is now to construct tensors $\{a^{2r},b^{2r}\}_{r=1}^{\lfloor\alpha/2\rfloor}$ satisfying the requirements \ref{en:defr1} to \ref{en:defr4}.
Note that if $\check{q}^r$ was positive semidefinite, the pair $a^{2r}=\check{q}^r$ and $b^{2r}=0$ would trivially satisfy \ref{en:defr1} to \ref{en:defr4}.
However, although the sign of $\check{q}^r$ is unknown for $r\geq2$, it is known that $\check{q}^1$ is negative definite (see \cite{COV06,ABV15,AbP16}). 
Hence, the main challenge of the construction lies in the sign of the tensors
and to build valid tensors, we need two basics from the tensor world.
First, we use the following result, proved in Section \ref{subsec:signOfEvenOrderTensor}.
\begin{lemma} \label{lem:tensorProdOfA0}
	For any $n\geq1$, the tensor $S^{2n}(\otimes^{n} a^0)$ is positive definite.
\end{lemma}
Second, we use a ``matricization'' operator, which linearly maps a symmetric tensor $q\in\Sym^{\inds}(\R^d)$
to a symmetric matrix $M(q)$ whose sign is the same as $q$, i.e.,
\begin{equation}	\label{eq:propertyM}
\genTen\xi :\xi = M(\genTen) \nu(\xi) \cdot \nu(\xi)
\quad \forall\xi,\eta\in\Sym^{\inds}(\R^d),
\end{equation}
where $\nu$ is a simple tensor-to-vector transformation. 
One construction for such an operator is provided in Section \ref{subsec:matricizationOperator}.

With Lemma \ref{lem:tensorProdOfA0} and the operator $M$, we are able to build tensors that satisfy \ref{en:defr1} to \ref{en:defr4}.
\begin{lemma}	\label{lem:constructionOfR}
	For $r\geq 1$, assume that $a^0$ and $\{a^{2s},b^{2s}\}_{s=1}^{r-1}$ have already been computed
	and let $\check{q}^r$ be the tensor defined in \eqref{eq:defCheckqr}.
	Define the tensor $R\in\Ten^{2r}(\R^d)$ as
	\begin{equation}	\label{eq:verifPosSemDef_alpha}
	\begin{aligned}
	&
	A^{r} = M\big(  S^{2 r+2}(\otimes^{r+1} a^0) \big),\\
	&\delta^* = \bigg\{ -\frac{\lambda_{\min}\big( M(\check{q}^r) \big)}{\lambda_{\min}(A^{r})}\bigg\}_+,\\
	&R = \delta^* S^{2r}(\otimes^{r} a^0), 
	\end{aligned}
	\end{equation}	
	where $\lambda_{\min}(\cdot)$ denotes the minimal eigenvalue of a symmetric matrix and
	$\{\cdot\}_+ = \max\{0,\cdot\}$.
	Then the tensors 
	\begin{equation}	\label{eq:pairofvalidtensors_alpha}
	a^{2r} = \check{q}^r + S^{2r+2}\big( R\otimes a^0 \big),
	\qquad
	b^{2r} = R ,
	\end{equation}
	satisfy the requirements \ref{en:defr1} to \ref{en:defr4} of Definition \ref{def:familyofeffectivesolution_alpha}.
\end{lemma}
\begin{proof}
	First, note that the orders of the tensors in \ref{en:defr1} are correct.
	Second, as $a^{2r},b^{2r}$ are fully symmetric tensors, they trivially satisfy the major symmetries and \ref{en:defr3} is verified.
	Next, we verify \ref{en:defr4}, i.e., $a^{2r},b^{2r}$ satisfy \eqref{eq:constrbis} (recall the meaning of $=_S$ in \eqref{eq:equalityS}):
	\[
	S^{2r+2}\big(a^{2r}-b^{2r}\otimes a^0\big)
	= S^{2r+2}\big(\check{q}^r\big) + S^{2r+2}\big(R\otimes a^0\big) - S^{2r+2}\big(R\otimes a^0\big)
	= S^{2r+2}\big(\check{q}^r\big) .
	\]
	We are left with \ref{en:defr2}: verifying that $a^{2r},b^{2r}$ are positive semidefinite.
	The positive semidefiniteness of $b^{2r}=R$ follows directly from the non-negativity of $\delta^*$ and Lemma \ref{lem:tensorProdOfA0}.
	Finally, let us verify that $a^{2r}$ is positive semidefinite.
	Using \eqref{eq:propertyM} and the definitions in \eqref{eq:verifPosSemDef_alpha}, we have
	\[
	a^{2r}\xi : \xi 
	= M\big(\check{q}^r +  \delta^* S^{2 r+2}(\otimes^{r+1} a^0)\big)v_\xi\cdot v_\xi 
	= M(\check{q}^r) v_\xi\cdot v_\xi 
	+ \delta^* A^{r} v_\xi\cdot v_\xi
	\geq \big( \lambda_{\min}\big( M(\check{q}^r) \big)
	+\delta^*\lambda_{\min}(A^{r})	\big) |v_\xi|^2,
	\]
	where we denoted $v_\xi=\nu(\xi)$ and $|\cdot|$ the Euclidian norm.
	The definition of $\delta^*$ in \eqref{eq:verifPosSemDef_alpha} ensures that the right-hand side is nonnegative, proving that $a^{2r}$ is positive semidefinite.
	That concludes the verification of \ref{en:defr1} to \ref{en:defr4} and the proof of the lemma is complete.
\end{proof}

Lemma \ref{lem:constructionOfR} allows to compute the tensors of one effective equations in the family $\famEfEq$. We emphasize that this is one possible construction among others.
To obtain other effective equations, we have many options.
For example, note that replacing $\delta^*$ in \eqref{eq:verifPosSemDef_alpha} with any $\delta\geq \delta^*$
provides other valid pairs of tensors.
Alternatively, we can replace $S^{2r}(\otimes^{r} a^0)$ in the definition of $R$ with another positive definite tensor.
Finally, let us note the following alternative (used in \cite{AbP16} in the case $\alpha=2$).
\begin{remark}	\label{rem:deftensors_foralgo}
	Assume that the tensor $\check{q}^r$ in \eqref{eq:defCheckqr} can be decomposed as 
	$\check{q}^r =_S \check{q}^{r,1} - \check{q}^{r,2}\otimess a^0$, where $\check{q}^{r,2}\in\Ten^{2r}(\R^d)$ is positive semidefinite.
	Then we verify that the pairs \eqref{eq:pairofvalidtensors_alpha}, where $R$ is defined as
	\begin{equation}	\label{eq:verifPosSemDef_second_alpha}
	\begin{aligned}
	&A^{r} = M\big(  S^{2 r+2}(\otimes^{r+1} a^0) \big),
	\quad
	\delta^* = \bigg\{ -\frac{\lambda_{\min}\big( M(  S^{2r+2}(\check{q}^{r,1} ))  \big)}{\lambda_{\min}(A^{r})}\bigg\}_+,\\
	&R = S^{2r}\big(\check{q}^{r,2} + \delta^* \otimes^{r} a^0\big),
	\end{aligned}
	\end{equation}
	define effective equations in the family.
\end{remark}



\subsection{A new remarkable relation between the correctors to reduce the cost of computation of the effective tensors}	\label{subsec:algorithm_alpha}

Let us discuss the cost of the procedure described in the previous section.
For an integer $\alpha$, assume that we want to construct the tensors of an effective equation for a timescale $\mathcal{O}(\eps^{-\alpha})$: $a^0, \{a^{2r},b^{2r}\}_{r=1}^{s}$, where $s = \lfloor\alpha/2\rfloor$.
The main computational cost of this construction lies in the calculation of the tensors $\{\check{q}^r\}_{r=1}^{s}$ in \eqref{eq:defCheckqr}, which involves the tensors $\{S^{2r+2}(\Gten^{2r})\}_{r=1}^{s}$, where
\begin{equation}	\label{eq:defbis_hr}
\Gten^{{2r}}_{i_1\cdotss i_{2r+2}} = \frac{1}{|Y|}\psBig{ a \big(\naby\chi^{2r+1}_{i_2\cdotss i_{2r+2}} + e_{i_2} \chi^{2r}_{i_3\cdotss i_{2r+2}} \big) , e_{i_1} }_{Y}.
\end{equation}
Following this natural---but naive---formula, we thus need approximations of the correctors $\chi^{1}$ to $\chi^{2s+1}$.
In this section, we present a new relation between the correctors ensuring that in fact the correctors $\chi^{1}$ to $\chi^{s+1}$ are sufficient (Theorem \ref{thm:lemmawpeven_alpha}).

Let us quantify the computational gain achieved thanks to this result.
As each $\chi^r$ is a symmetric tensor valued function, it has $\binom{r+d-1}{r}$ distinct components (see Remark \ref{rem:correctorsAreSymmetric}).
The number of cell problems to solve to have all the distinct components of $\chi^{1}$ to $\chi^{k}$ is thus
\[
\mathrm{CP}(d,k) 
= \sum_{r=1}^{k} \binom{r+d-1}{r} = \binom{k+d}{d} - 1.
\]
Hence, to compute $\{\check{q}^r\}_{r=1}^{s}$ Theorem \ref{thm:lemmawpeven_alpha} allows to avoid the approximation of $\chi^{s+2},\hdots,\chi^{2s+1}$, i.e., we spare the approximation of
\[
\mathrm{CP}(d,2s+1) - \mathrm{CP}(d,s+1) 
= \binom{2s+1+d}{d} -  \binom{s+1+d}{d}
\]
cell problems.
To fully appreciate this gain, assume that we want to compute the effective tensors of an effective equation for a timescale $\mathcal{O}(\eps^{-6})$, i.e. $s=\lfloor\alpha/2\rfloor=3$:
for $d=2$, we spare $21$ cell problems ($14$ cell problems to solve instead of $35$)
and if $d=3$, we spare $85$ cell problems ($34$ cell problems to solve instead of $119$).

\begin{remark}
	It is important to note that this gain is not specific to the effective equations presented here.
	Indeed, formula \eqref{eq:decompostionhr} can directly be used to compute the (sufficient) symmetric part of the high order tensors in the effective models from \cite{ALR18} and \cite{BeG17} discussed in Section \ref{subsec:comparisonGloria_alpha}.
\end{remark}

The new relation between the correctors is presented in the following result, proved in Section \ref{subsec:prooflemma_even}.
Note that this result was obtained independently in \cite[Remark 2.5]{DuO19} in the context of stochastic homogenization.
\begin{theorem}	\label{thm:lemmawpeven_alpha}
	Let $\{\chi^{k}_{i_1\cdotss i_{k}}\}_{k=1}^{2r+1}$ be the zero mean correctors defined in \eqref{eq:cellProblemsWave_compact}.
	Then, for $1\leq r\leq \lfloor \alpha/2\rfloor$, the tensor $\Gten^{2r}$ defined by
	\eqref{eq:defbis_hr} satisfies the decomposition
	\begin{equation}	\label{eq:decompostionhr}
	\begin{aligned}
	\Gten^{2r}
	=_S
	(-1)^r\Kten^r + \Hten^r,
	\end{aligned}
	\end{equation}
	where the tensors $\Kten^r, \Hten^r \in\Ten^{2r+2}(\R^d)$ are defined as
	\begin{equation}	\label{eq:defpq_alpha}
	\begin{aligned}
	&
	\Kten^r_{i_1\cdotss i_{2r+2}}
	= 
	-\intMeanBig{ a\naby\chi^{r+1}_{i_1\cdotss i_{r+1}}\cdot \naby\chi^{r+1}_{i_{r+2}\cdotss i_{2r+2}} }_Y
	+ \intMeanBig{a e_{i_2}\chi^r_{i_3\cdotss i_{r+2}}\cdot e_{i_1} \chi^r_{i_{r+3}\cdotss i_{2r+2}} }_Y,
	\\
	&\Hten^r =
	\sum_{j=1}^{\lceil{r}/{2}\rceil}
	\sum_{\ell=1}^{\lceil{r}/{2}\rceil}
	\intMeanBig{ \Dten^{2(r-j-\ell+1)} \otimess \chi^{2j-1} \otimess \chi^{2\ell-1} }_Y
	-\sum_{j=1}^{\lfloor{r}/{2}\rfloor}
	\sum_{\ell=1}^{\lfloor{r}/{2}\rfloor}
	\intMeanBig{ \Dten^{2(r-j-\ell)} \otimess \chi^{2j} \otimess \chi^{2\ell} }_Y
	,
	\end{aligned}
	\end{equation}
	where 
	the second double sum in the definition of $\Hten^r$ vanishes for $r=1$.
\end{theorem}
Observe that the tensor $\Hten^r$ only depends on the correctors $\chi^1$ to $\chi^{r}$
so that decomposition \eqref{eq:decompostionhr} indeed guarantees that $S^{2r+2}(\Gten^{2r})$ can be computed from $\chi^1$ to $\chi^{r+1}$.

\begin{remark}
	It can be verified that the homogenized tensor \eqref{eq:defhomogenizedTensor_alpha} satisfies $a^0 = \Gten^0 = \Kten^0$.
	Hence, decomposition \eqref{eq:decompostionhr} also holds in the case $r=0$ with $\Hten^0=0$ (recall that $\chi^0=1$). 
\end{remark}


\subsection{Matrix associated to a symmetric tensor of even order}	\label{subsec:matricizationOperator}

In this section, we construct the ``matricization'' operator used in the construction of Section \ref{subsec:constructionEffectiveTensors}.
This operator maps a given symmetric tensor of even order $\genTen$ to a matrix whose sign is the same as $\genTen$.

We consider the bilinear map
\begin{equation} \label{eq:mapdoubledotproduct_alpha}
\Sym^\inds(\R^d)\times\Sym^\inds(\R^d)\to \R,
\qquad (\xi,\eta)\mapsto 
\genTen\xi : \eta = \genTen_{i_1\cdots i_{\inds}i_{\inds+1}\cdots i_{2\inds}} 
\xi_{i_1\cdots i_{\inds}}\eta_{i_{\inds+1}\cdots i_{2\inds}}.
\end{equation}
Denote
$I(d,\inds)$ the set of multiindices of the distinct entries of a tensor in $\Sym^{\inds}(\R^d)$, i.e.,
\begin{equation}	\label{eq:defIdn}
I(d,\inds) = 
\big\{	i = (i_1,\hdots,i_\inds) : 1\leq i_1\leq \hdots \leq i_\inds\leq d	\big\}.
\end{equation}
We verify that the cardinality of $I(d,\inds)$ is
$ 
N(d,\inds) = |I(d,\inds)| = \binom{d+\inds-1}{\inds}.
$ 
We denote $J(d,\inds) = \{ 1,\hdots, N(d,\inds)\}$
and let $\ell:J(d,\inds)\to I(d,\inds)$ be a bijection.
We define then the bijective mapping 
\[
\nu : \Sym^\inds(\R^d) \to \R^{N(d,\inds)},
\qquad \xi\mapsto \nu(\xi),
\quad
\big(\nu(\xi)\big)_r = \xi_{\ell(r)} \quad r\in J(d,\inds).
\]
For $i\in I(d,\inds)$, let $z(i)$ be the number of multiindices in $\{1,\hdots,d\}^\inds$ that are equal to $i$ up to symmetries, i.e.,
\[
z(i) = \big| \{ j\in\{1,\hdots ,d\}^\inds : \text{ there exists a permutation } \sigma \text{ s.t. } \sigma(j) = i\} \big|.
\]
With these notations, we rewrite the map defined in \eqref{eq:mapdoubledotproduct_alpha} as
\[
\genTen\xi :\eta
= \sum_{i,j\in I(d,\inds)} z(i) z(j) \genTen_{ij} \xi_i \eta_j
= \sum_{r,s=1}^{N(d,\inds)}  z\big(\ell(r)\big) z\big(\ell(s)\big) \genTen_{\ell(r)\ell(s)} \xi_{\ell(r)} \eta_{\ell(s)}.
\]
Finally, we define the matrix associated to a tensor as
\begin{equation}	\label{eq:matrixassociatedtotensor_alpha}
\begin{aligned}	
M:\Sym^{2\inds}(\R^d)&\to \Sym^2(\R^{N(d,\inds)}),
\\
\genTen &\mapsto M(\genTen)
\quad
\big(M(\genTen)\big)_{rs} = z\big(\ell(r)\big) z\big(\ell(s)\big) \genTen_{\ell(r)\ell(s)}
\quad m,n\in J(d,\inds).
\end{aligned}
\end{equation}
We verify that for any $\xi,\eta\in\Sym^{n}(\R^d)$, $\genTen\xi :\eta = M(\genTen) \nu(\xi) \cdot \nu(\eta)$.
Hence, $\genTen$ is positive definite (resp. semidefinite) if and only if $M(\genTen)$ is positive definite (resp. semidefinite).
In particular, property \eqref{eq:propertyM} holds.


\subsection{Algorithm for the computation of the high order effective tensors}	\label{subsec:algo}

We present here the algorithm for the computation of the effective tensors of one equation in the family $\famEfEq$. 
The procedure relies on the construction explained in Lemma \ref{lem:constructionOfR}, the formula provided by Theorem \ref{thm:lemmawpeven_alpha},
and the ``matricization'' operator $M$ defined in \eqref{eq:matrixassociatedtotensor_alpha}.
The set of index $I(d,n)$ is defined in \eqref{eq:defIdn}.

\setcounter{algorithm}{\value{theorem}}
\begin{breakablealgorithm}	
	\caption{Compute the tensors of an effective equation \eqref{eq:effectiveEquation_alpha} in the family $\famEfEq$.}
	\label{algo:alpha}
	\begin{algorithmic}[1]
		\Require{$Y$-periodic, $d\times d$ symmetric tensor $a(y)$ satisfying \eqref{eq:uniformEllipticity_lpd}; timescale $\alpha$.}
		\Ensure{effective tensors $a^0, \{a^{2r},b^{2r}\}_{r=1}^{\lfloor \alpha/2 \rfloor}$.}
		\vspace{5pt}
		\State{for all $i\in I(d,1)$ solve the cell problem for $\chi^1_i$ in \eqref{eq:cellProblemsWave_compact_1} 
			with $\intMeanbig{\chi^1_i}_Y = 0$}	
		\State{for all $i\in I(d,2)$\quad 
			$\Cten^0_{i_1i_2} = a^0_{i_1 i_2} = - \intMeanbig{a\naby\chi_{i_2}^1\cdot\naby\chi_{i_1}^1}_Y + \intMeanbig{ae_{i_2}\cdot e_{i_1}}_Y$ }
		\For{$r =1,\hdots,\lfloor \alpha/2 \rfloor$}
		\State{
			for all $i \in I(d,r+1)$ solve the cell problems for $\chi^{r+1}_{i_1\cdotss i_{r+1}}$ in \eqref{eq:cellProblemsWave_compact_k} with $\intMeanbig{\chi^{r+1}_{i_1\cdotss i_{r+1}}}_Y = 0$
		} 
		\State{
			compute the tensor $S^{2r+2}\big(\Gten^{2r}\big)$ with formula \eqref{eq:decompostionhr}
		}
		\State{
			compute the tensor $\check{q}^r$ with formula \eqref{eq:defCheckqr}
		}
		\State{
			build the matrices $M(\check{q}^{r})$, $A^{r}= M\big( S^{2r+2}(\otimes^{r+1} a^0)\big)$
			and	compute $\delta^* =  \bigg\{ -\displaystyle\frac{\lambda_{\min}(M(\check{q}^{r}))}{\lambda_{\min}(A^r)}\bigg\}_+$
		}
		\State{
			compute
			$a^{2r} = \check{q}^{2r} + \delta^*S^{2r+2}\big(\otimes^{r+1} a^0\big)$,
			\quad $b^{2r} = \delta^*S^{2r}\big(\otimes^{r} a^0\big)$,\quad
			$\Cten^r = a^{2r} - \displaystyle\sum_{\ell=0}^{r-1} \Cten^{\ell} \otimess b^{2(r-\ell)}$
		}
		\EndFor
	\end{algorithmic}
\end{breakablealgorithm}
\setcounter{theorem}{\value{algorithm}}


\section{Numerical experiments} \label{sec:numexp_alpha}

In this section, we present numerical experiments to illustrate the result of Theorem \ref{thm:est_epsm_alpha}:
the effective equations in the family capture the long-time behavior of $\ueps$.
Note that reporting numerical error in the approximation of $\ueps$ in a pseudo-infinite medium for very large timescales is not conceivable as computing a reference solution is out of reach even for one-dimensional problems.
We can however consider a small periodic domain as this setting is covered by our theory. 
In addition, we will illustrate that high order effective models are also useful when we deal with high frequency regimes. 

We consider the one-dimensional model problem \eqref{eq:waveEquationEpsilon_lpd} given by the data
$\cini(x) = e^{-4{x^2}}$, $\dini=f=0$,
$a(y) = \sqrt{2} -\cos(2\pi y)$ (we verify that $a^0=1$)
with $\eps=1/10$, and the periodic domain $\Omega= (-L,L)$, $L=84$.
Let us give some insight on the macroscopic evolution of $\ueps$:
the central pulse $\cini$ separates into left- and right-going waves packets with speed $a^0=\pm1$; these packets meet at $x=L$ (equivalently $x=-L$) for $t=L+2kL$, $k\in\N$, and at $x=0$ for $t=2kL$, $k\in\N$. 
As time increases, dispersion appears in each packet. 
As the domain has a periodic boundary, at some point in time the packets start to superpose. In our settings, $t=\eps^{-4}$ is the larger timescale until the bulk of the dispersion spans approximately half of the domain $\Omega$.
For $\alpha=0,2,4$, we compare $\ueps$ with the corresponding effective solutions of order $s=\lfloor{\alpha/2}\rfloor$, denoted $\tilde{u}^{\lfloor{\alpha/2}\rfloor}$ ($\tilde{u}^{0}$ is the homogenized solution).
We approximate $\ueps$ using a spectral method grid of size $h=\eps/16$ and a leap frog scheme for the time integration with $\Delta t = h/16$.
The effective solutions $\{\tilde{u}^s\}_{s=0}^2$ are approximated with a Fourier method (the effective coefficients are computed using Algorithm \ref{algo:alpha}) on a grid of size $h=\eps/16$ (no time integration). 
Details on the numerical methods can be found in \cite[section 5.3]{Pou17}.

The results are displayed in Figure \ref{fig:alpha_1d}.
In the top-left plot, we compare $\ueps$, $\tildeub^0$, $\tildeub^1$ at $t=\eps^{-2}=10^2$.
We observe that $\tildeub^0$ does not capture the macroscopic dispersion developed by $\ueps$, while $\tildeub^1$ accurately describes it.
In the top-right plot, we compare $\ueps$, $\{\tildeub^s\}_{s=0}^2$, at $t=\eps^{-4}=10^4$
({observe} that right- and left- going wave packets are in fact superposed).
At this timescale, the first order effective solution $\tildeub^1$ does not capture accurately the dispersion anymore, as expected, but the second order $\tildeub^2$ does.
In the bottom plot, we compare the normalized errors
$
\mathrm{err} (\tilde{u}^{s})(t) = \norm{(u^\eps-\tilde{u}^s)(t)}_\LdOm / \norm{u^\eps(t)}_\LdOm
$
for the effective solutions $\{\tildeub^s\}_{s=0}^2$ on the time interval $[0,\eps^{-4}]$ (the $x$-axis is in log scale).
We observe that $\tildeub^0$ is accurate up to $\eps^{-1}$ and then deteriorates.
The first order $\tildeub^1$ is accurate up to $\eps^{-3}$, then deteriorates.
Finally, the second order $\tildeub^2$ has a satisfying accuracy over the whole time interval $(0,\eps^{-4})$.
The observations of this experiment corroborate the result of Theorem \ref{thm:est_epsm_alpha}:
the effective solution $\tilde{u}^{\lfloor{\alpha/2}\rfloor}$ accurately describe $\ueps$ up to $\mathcal{O}(\eps^{-\alpha})$ timescales.
\begin{figure}[!ht]
	\begin{center}
		\begin{tabular}{@{}c@{\hspace{1cm}}c@{}}
			\includegraphics[width=6cm]{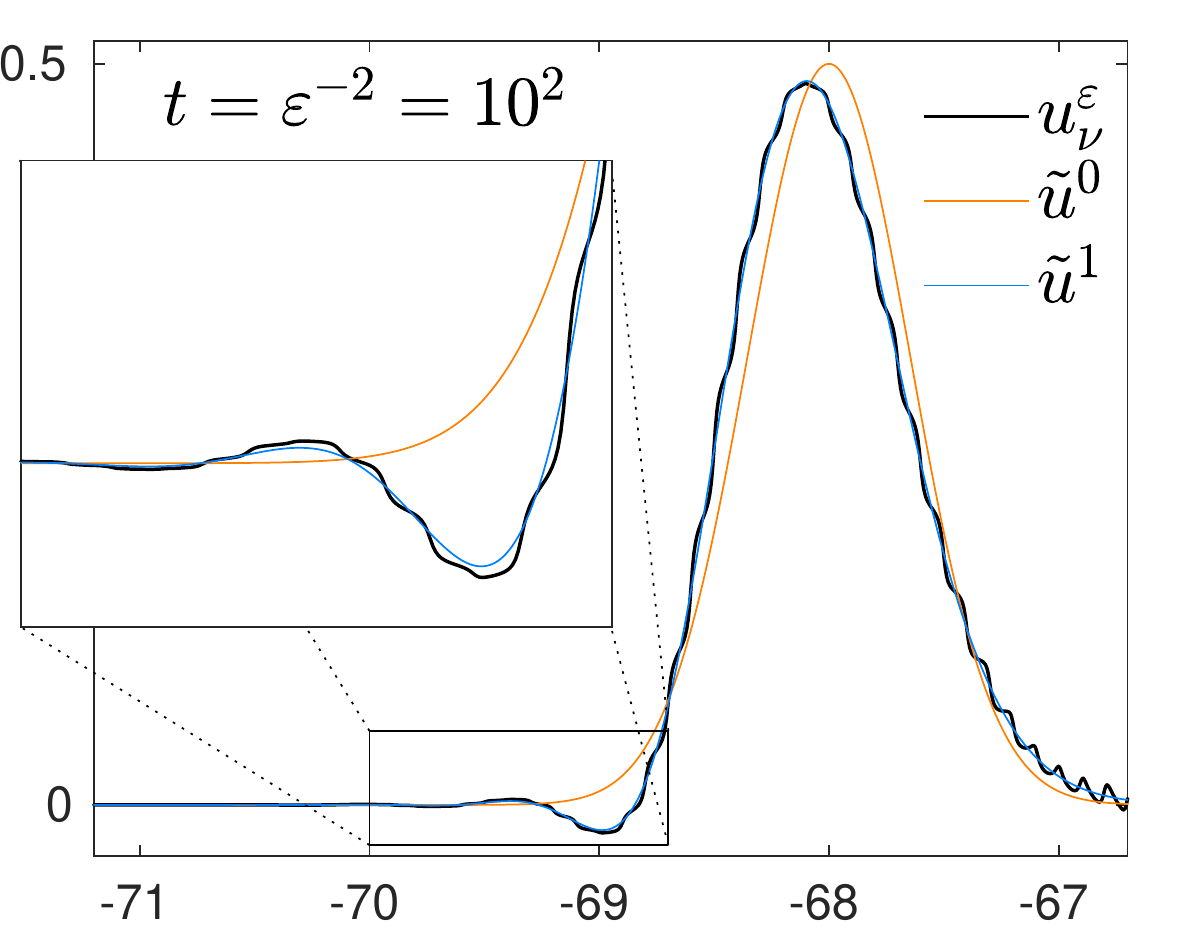}
			&
			\includegraphics[width=6cm]{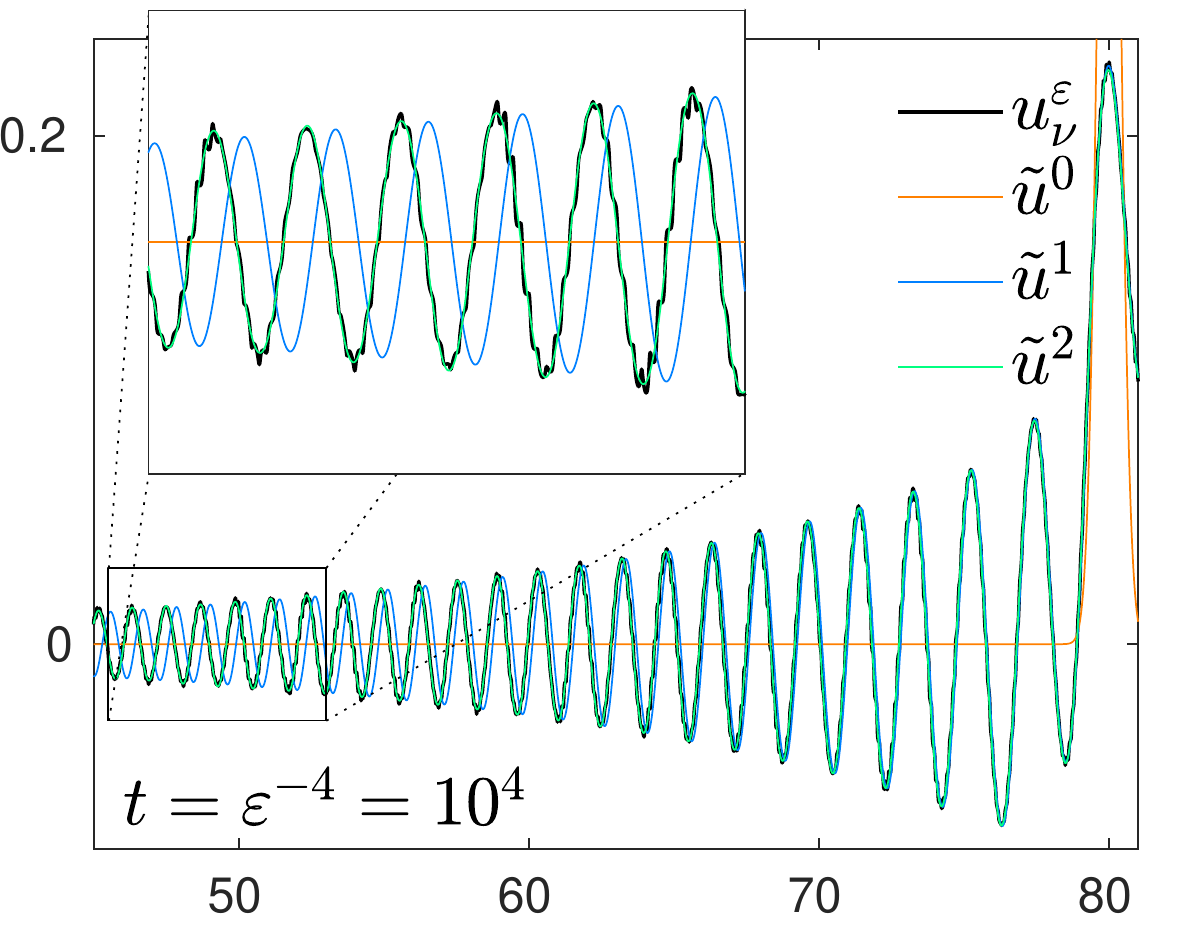}
		\end{tabular}
		\includegraphics[width=7cm]{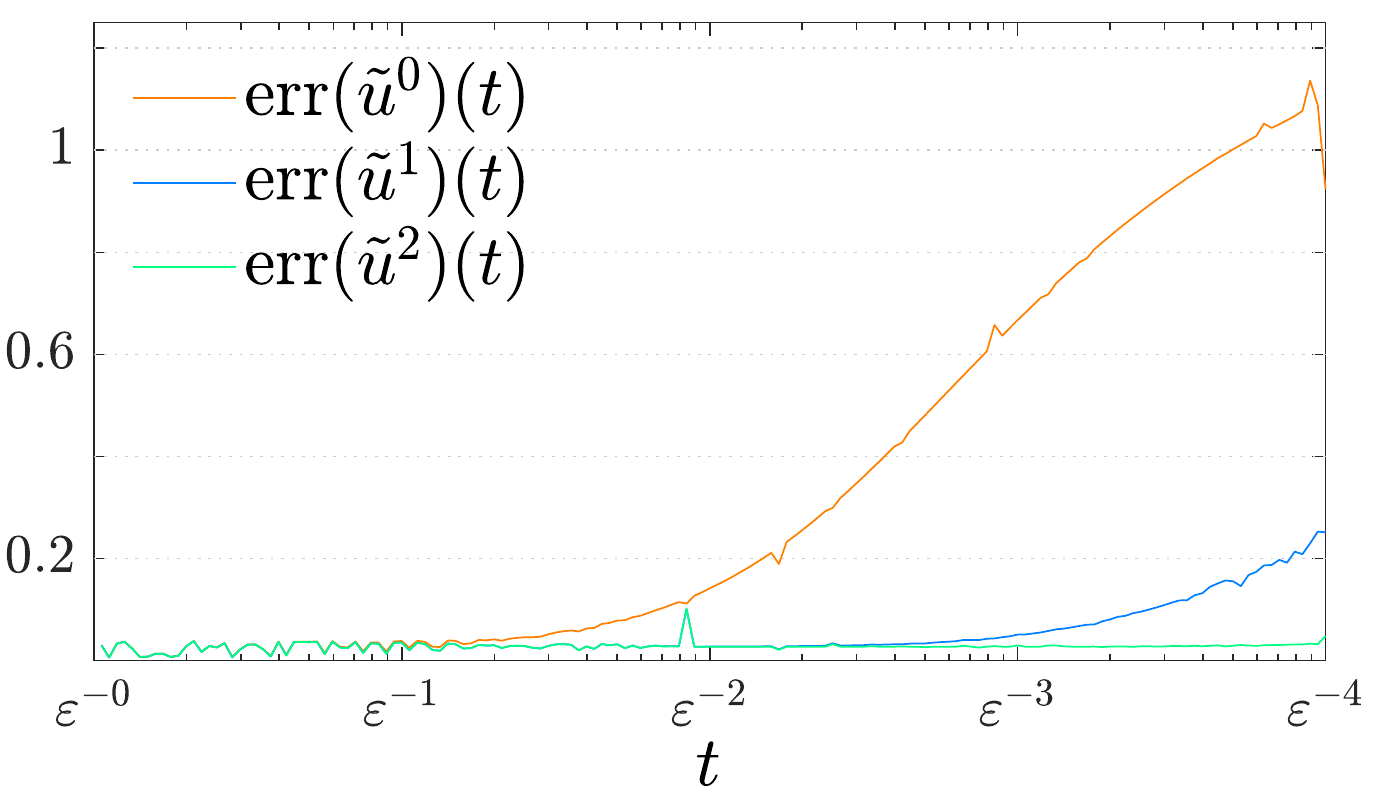}
		\caption{
			Top-left: Comparison of $\ueps$ with the homogenized solution $\tildeub^0$ and the order $1$ effective solution $\tildeub^1$ at $t=\eps^{-2}$.
			Top-right: Comparison of $u^\eps$ with $\tildeub^0$, $\tildeub^1$, and an order $2$ effective solution $\tildeub^2$ at $t=\eps^{-4}$.
			Bottom: Plot of the errors between $\ueps$ and $\tildeub^s$, $0\leq s\leq 3$, over the time interval $[0,\eps^{-4}]$.
		}
		\label{fig:alpha_1d}
	\end{center}
\end{figure}

In two dimensions, computing a reference solution for the previous experiment has a huge computational cost.
However, we can illustrate an interesting fact: high order effective models are useful in high frequency regimes.
Let us first give some insights on this fact.
We consider the following specific settings:
let $a\ofxsureps$ be a given tensor with $\eps>0$ fixed, $g(x)=e^{\beta^2|x|^2}$ be a Gaussian with $\beta=\mathcal{O}(1)$ and $\nu>0$ a scaling parameter.
In \eqref{eq:waveEquationEpsilon_lpd}, we let $f=0$ and the initial conditions $\cini(x) = g(\nu x)$ and $\dini=0$ (we assume that $\Omega$ is arbitrarily large).
To specify the dependence of $\ueps$ on the parameter $\nu$, we denote it as $u^\eps_\nu$.
In the equation for $u^\eps_\nu$, making the changes of variables
$\hat x =\nu x$ 
and
$\hat t =\nu t$
and introducing $\hat u^\eps_\nu(\hat t,\hat x) = u^\eps_\nu(\hat t/\nu,\hat x/\nu)$, we find
\begin{equation*}	
\pa_{\hat t}^2 \hat u^\eps_\nu(\hat t,\hat x) = \nabla_{\hat x}\big(a\big(\tfrac{\hat x}{\nu\eps}\big)\nabla_{\hat x} \hat u^\eps_\nu(\hat t,\hat x)\big)	,
\qquad
\hat u^\eps_\nu(0,\hat x) = g(\hat x),\quad\pa_{\hat t} \hat u^\eps_\nu(0,\hat x) = 0,		
\end{equation*}
and we conclude that $u^\eps_\nu(\hat t/\nu,\hat x/\nu) = \hat{u}^\eps_\nu(\hat t,\hat x) = u^{\nu\eps}_1(\hat t,\hat x)$.
In other words, for $\nu>1$ (i.e., an increase in the frequencies of the initial wave) the long-time effects of $u^{\nu\eps}_1$ can be observed at a shorter time in $u^\eps_\nu$ (modulo a contraction of space). 
However, for high values of $\nu$ we meet situations where Theorem \ref{thm:est_epsm_alpha} does not provide a satisfactory error estimate:
on the one hand, in the estimate for $u^\eps_\nu$, the increase of $\nu$ deteriorates the error constant in Theorem \ref{thm:est_epsm_alpha};
one the other hand, in the estimate for $u^{\nu\eps}_1$, the constant is good, but the period of the tensor $\eps'=\nu\eps$ is to far from the asymptotic regime for homogenization to be meaningful.
In practice, we observe that the increase of the frequencies of the initial wave (increase of $\nu$)
leads to additional dispersive effects in $u^\eps_\nu$ (or $u^{\nu\eps}_1$). 
Furthermore, the use of higher order effective models allows to capture these additional effects. 

To illustrate this, we consider the two dimensional model problem given by the data $f=\dini=0$ and
\[
\cini(x) = g(\nu x),
~
g(x)  = e^{-20{|x|^2}},
~
\nu = 5^{1/3},
\quad
a(y)
= 
\begin{pmatrix}
1 - 0.5\cos(2\pi y_2) & 0 \\
0 & 1 - 0.5\cos(2\pi y_2)
\end{pmatrix}
,
\quad
\eps=1/10.
\]
We compute the effective tensors using Algorithm \ref{algo:alpha} and approximate $u^\eps_\nu$ and $\{\tilde{u}^s_\nu\}_{s=1}^3$ at time $t=20$.
As above, $\ueps$ is approximated with a spectral method on a grid of size $h=\epsilon/16$ and leap frog scheme with $\Delta t = h/100$ and $\{\tildeub^s\}_{s=0}^2$ are computed with a Fourier method on a grid of size $h=\eps/16$ (see \cite[section 5.3]{Pou17}).
In Figure \ref{fig:alpha_2d}, we compare the solution in the periodic medium $u^\eps_\nu$ (top-left)
with the effective solutions of order $1$, $\tilde{u}^1_\nu$ (top-right), and order $2$, $\tilde{u}^2_\nu$ (bottom-left).
We observe that $\tilde{u}^2_\nu$ captures more accurately the dispersion developed by $u^\eps_\nu$ than $\tilde{u}^1_\nu$.
This is even better seen in the 1d cut at $\{x_1=0\}$ in the bottom-right plot of Figure \ref{fig:alpha_2d}.
Furthermore, even though distinguishing the higher order dispersion from the $\eps$-scale oscillation is not easy in this regime, in the zoom we can guess that the model of order $3$ is better than the order $2$.

\begin{figure}[!ht]
	\begin{center}
		\begin{tabular}{@{}c@{\hspace{0.0cm}}c@{}}
			\includegraphics[width=7.32cm]{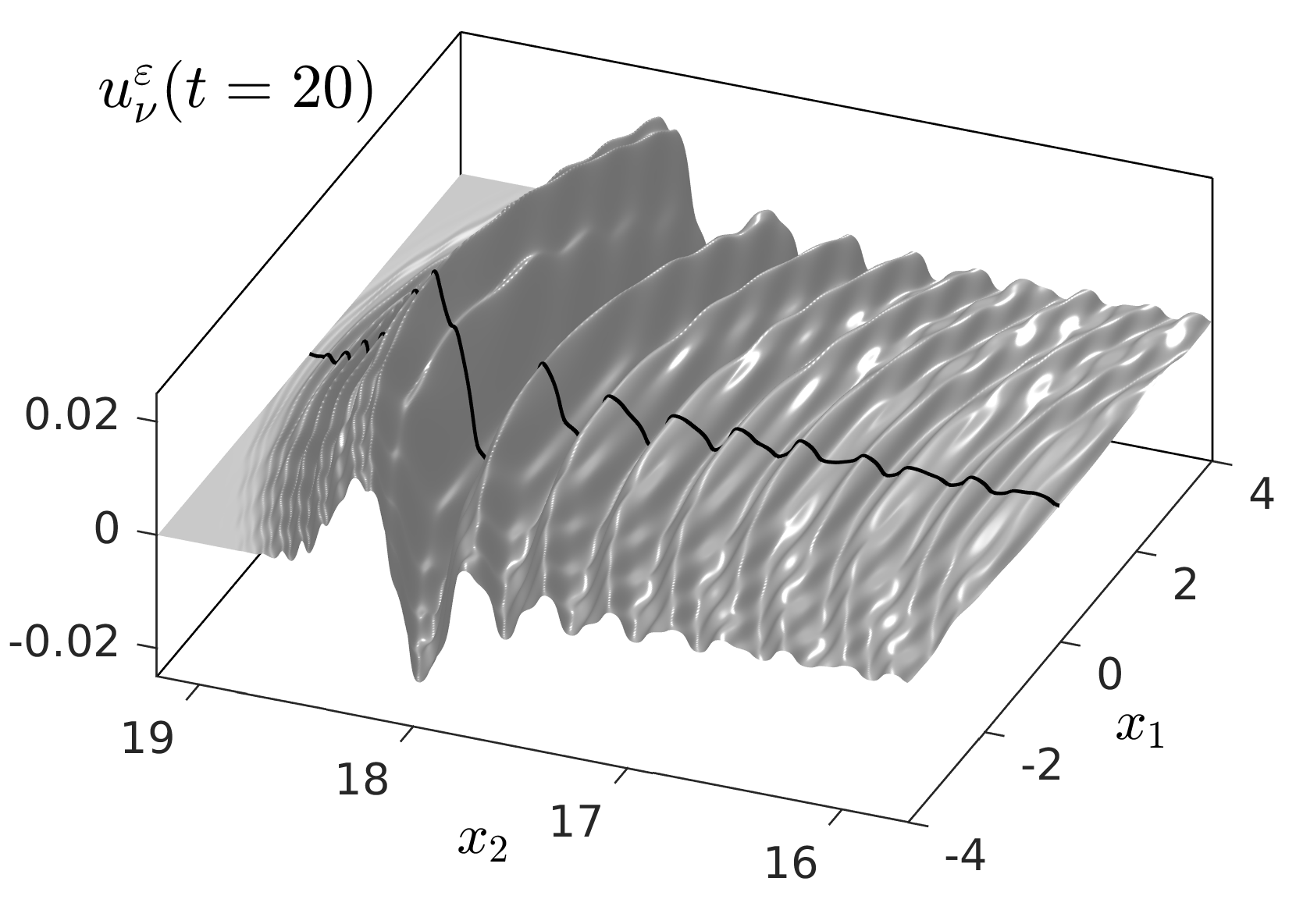}
			&
			\includegraphics[width=7.32cm]{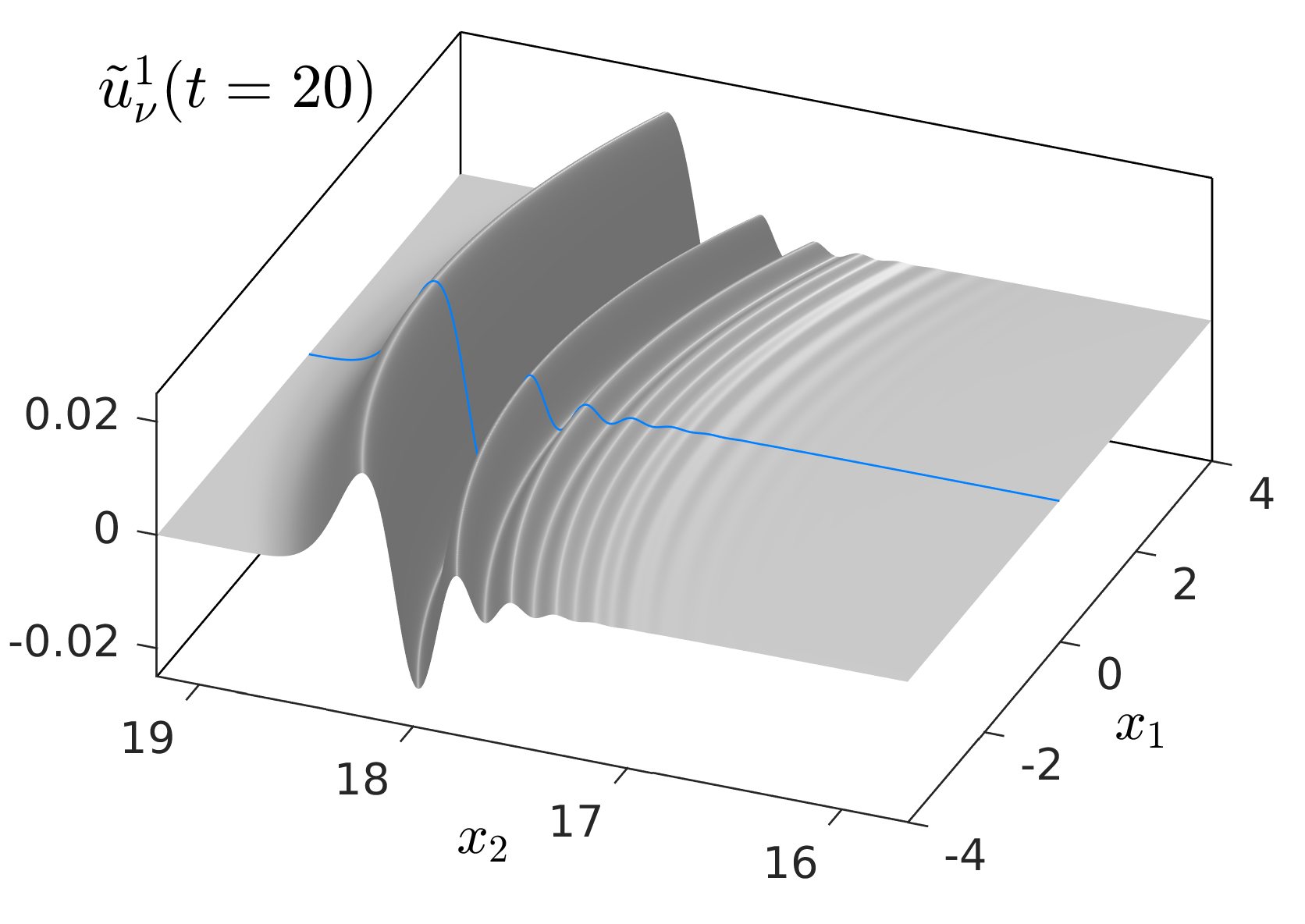}
			\\
			\includegraphics[width=7.32cm]{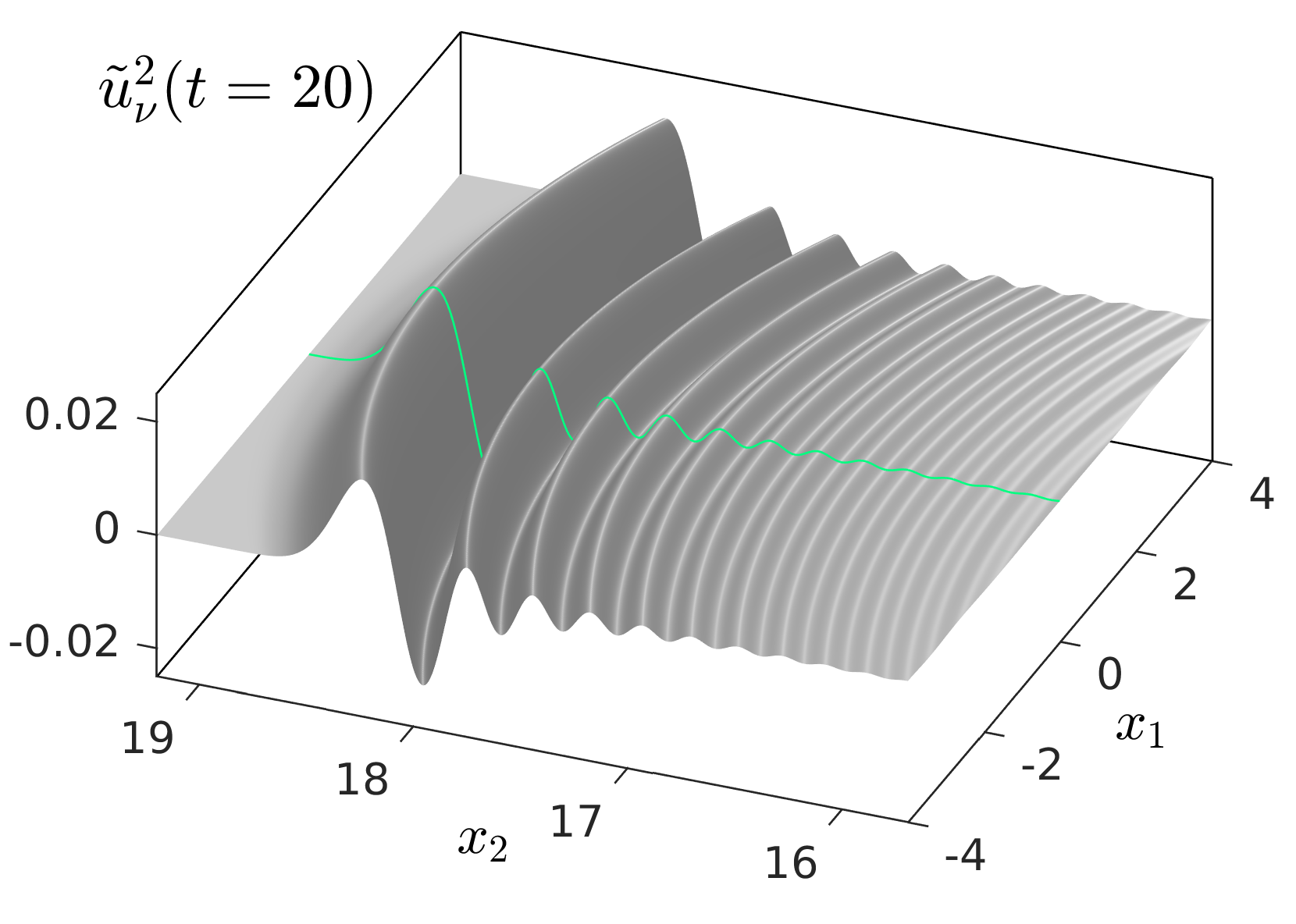}
			&
			\includegraphics[width=6cm]{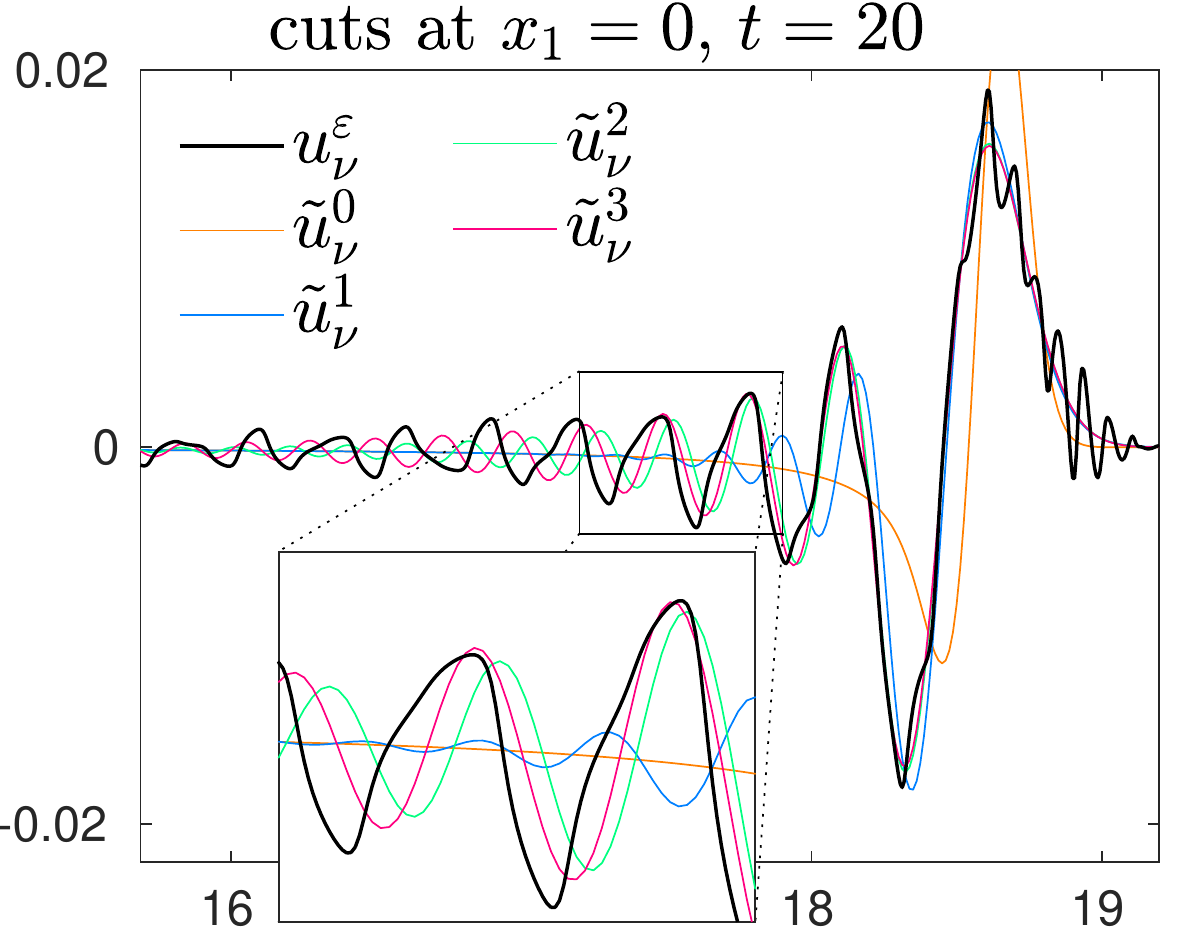}
		\end{tabular}
		\caption{
			Comparison of $\ueps_\nu$ and $\tilde{u}^s_\nu$ at $t=20$ on the subdomains $[-4,4]\times[\sqrt{a^0_{22}}t-5/\nu,\sqrt{a^0_{22}}t+1/\nu]$,
			and the corresponding cuts along $x_1=0$.
		} 
		\label{fig:alpha_2d}
	\end{center}
\end{figure}

\section{Proofs of the main results}	\label{sec:proofs}

In this section, we provide the proofs of the main results of the paper.
In \ref{subsec:proofbousstrick_both}, we prove the inductive Boussinesq tricks: Theorem \ref{thm:rewritedtttildeub_2_alpha} and Lemma \ref{lem:inductiveBoussinesqTricks_alpha}.
In \ref{subsec:summaryproof}, the a priori error estimate for the family of effective equations of Theorem \ref{thm:est_epsm_alpha} is proved.
In \ref{subsec:signOfEvenOrderTensor}, we prove the result on positive definite tensors from Lemma \ref{lem:tensorProdOfA0} and in \ref{subsec:prooflemma_even}, we prove the new relation between the correctors from Theorem \ref{thm:lemmawpeven_alpha}.


\subsection{Inductive Boussinesq tricks for the derivation of the family of effective equations}	\label{subsec:proofbousstrick_both}

In this section, we present the technical task that was postponed in the derivation of the family of effective equations in Section \ref{subsec:derivationFamily_subsec}.
Specifically, we provide the result allowing to substitute the time derivatives in the terms of order $\mathcal{O}(\eps^0)$ to $\mathcal{O}(\eps^{\alpha})$ in $r^\eps$ \eqref{eq:waveAsymptoticDeveloppment_alpha}.
To that end, the main challenge is to proceed to inductive Boussinesq tricks (Theorem \ref{thm:rewritedtttildeub_2_alpha}).
This result is then used in Lemma \ref{lem:inductiveBoussinesqTricks_alpha}, which provides the specific relation used in Section \ref{subsec:derivationFamily_subsec}.

\subsubsection{Inductive Boussinesq tricks}	\label{subsec:proofbousstrick}

The following theorem is the key result of this section.
\def\indk{N}
\begin{theorem}	\label{thm:rewritedtttildeub_2_alpha}	
	Let $\tildeub$ be the solution of \eqref{eq:effectiveEquation_alpha} and let $\indk$ be an even integer such that $0\leq \indk\leq 2\lfloor\alpha/2\rfloor$.
	If the right-hand side of \eqref{eq:effectiveEquation_alpha} is defined as
	\begin{equation}	\label{eq:defQf}
	Qf 	= f + \sum_{r=1}^{\lfloor \alpha/2 \rfloor}(-1)^r\eps^{2r} b^{2r} \nabxn{2r} f,
	\end{equation}
	then $\tildeub$ satisfies
	\begin{equation}	\label{eq:rewritedtttildeub_2_alpha}
	\dtt\tildeub = f + \sum_{r=0}^{\indk/2} \epsilon^{2r} (-1)^r\Cten^r \nabxn {2r+2}\tildeub 
	+ \mathcal{S}^\eps_{\indk} f
	+ \mathcal{R}^\eps_{\indk} \tildeub,
	\end{equation}
	where $\Cten^r$ is the tensor defined inductively as
	\begin{equation}	\label{eq:definitionCr}
	\Cten^0 = a^0,
	\qquad
	\Cten^r = a^{2r} - \sum_{\ell=0}^{r-1} \Cten^\ell \otimes b^{2(r-\ell)} 
	\quad 1\leq r\leq \lfloor\alpha/2\rfloor,
	\end{equation}
	and the remainders $\mathcal{R}^\eps_{\indk}\tildeub$ and 
	$\mathcal{S}^\eps_{\indk} f$ are defined in \eqref{eq:defReps0} and satisfy the following estimates:
	for any integer $n$ such that $\indk+n\leq \alpha$
		\begin{equation}	\label{eq:estimateR1eps}
		\begin{aligned}
		\norm{\nabxn{n}(\mathcal{R}^\eps_{\indk}\tildeub)}_{\Li\Lti\LdOm}
		&\leq C \eps^{\indk+2}
		\sum_{j=\indk+n+2}^{r(\alpha,N,n)} \Big(\norm{\nabxn{j+2}\tildeub}_{\Li\Lti\LdOm} + \norm{\nabxn{j}\dtt\tildeub}_{\Li\Lti\LdOm}	\Big),\\
		\norm{\nabxn{n}(\mathcal{S}^\eps_{\indk}f)}_{\Lu\Lti\LdOm}
		&\leq C \eps^{\indk+2} \sum_{j=\indk+n+2}^{r(\alpha,N,n)}\norm{\nabxn{j}f}_{\Lu\Lti\LdOm},
		\end{aligned}
		\end{equation}
	where 
	$r(\alpha,\indk,n) = 2\lfloor\alpha/2\rfloor+\indk+n$
	and the constants depend only on the tensors $a^0,\{a^{2r},b^{2r}\}_{r=1}^{\lfloor\alpha/2\rfloor}$.
\end{theorem}

\begin{remark}	\label{rem:explanationQf2}
	We verify that the definition \eqref{eq:defQf} of $Qf$ ensures the remainder $\mathcal{S}^\eps_{\indk} f$ to have maximal order in terms of $\eps$ in the second estimate in \eqref{eq:estimateR1eps}.
	More specifically, among the right-hand side of the form $Qf = \sum_{s=0}^{\lfloor\alpha/2\rfloor} (-1)^s \eps^{2s} q^s \nabxn{2s} f$, $q^s\in\Ten^{2s}(\R^d)$, \eqref{eq:defQf} is the only one that ensures 
	$\norm{\mathcal{S}^\eps_{\indk}f} = \mathcal{O}(\eps^{N+2})$ for all even $\indk$ such that $0\leq \indk\leq 2\lfloor\alpha/2\rfloor$.
	This affirmation is proved in Remark \ref{rem:explanationQf3}.
\end{remark}

The result of Theorem \ref{thm:rewritedtttildeub_2_alpha} relies on inductive Boussinesq tricks.
Let us summarize this process.
We start from the expression of $\dtt\tildeub$ given by the effective equation \eqref{eq:effectiveEquation_alpha}, which we rewrite as
\begin{equation}	\label{eq:effectiveEquation_pr0_alpha}
\dtt\tildeub =
Qf +  \displaystyle\sum_{r=0}^{\lfloor\alpha/2\rfloor}(-1)^r\eps^{2r}
a^{2r}\nabxn {2r+2}\tildeub 
+ \displaystyle\sum_{r=1}^{\lfloor\alpha/2\rfloor}(-1)^r\eps^{2r}\big(-b^{2r}\big)\nabxn{2r}\dtt\tildeub,
\end{equation}
where $Qf$ is defined in \eqref{eq:defQf} (see Remark \ref{rem:explanationQf2}).
The result for $N=0$ trivially follows \eqref{eq:effectiveEquation_pr0_alpha} with $\mathcal{S}^\eps_{0} f= G^0(0)$ and $\mathcal{R}^\eps_{0} \tildeub=T^0(0)$ (defined in \eqref{eq:defT0_alpha} below).
Let us then assume that $\indk$ is an even number such that $2 \leq\indk\leq \alpha$.
In the right-hand side of \eqref{eq:effectiveEquation_pr0_alpha}, we inductively substitute $\dtt\tildeub$ (using \eqref{eq:effectiveEquation_pr0_alpha} itself)
in the terms of order $\mathcal{O}(\eps^{2})$ to $\mathcal{O}(\eps^{\indk})$.
To understand this technical process at best, let us define inductively some quantities and functions. 
Note that these definitions arise in the result of Lemma \ref{lem:bousstricklemma1}, which exhibits the decomposition process of one Boussinesq trick.

We define the tensors
\begin{equation}	\label{eq:defBri_alpha}
\begin{aligned}
&B^r(0) = -b^{2r} && 1\leq r\leq \lfloor\alpha/2\rfloor,\\
&B^r(j) = -\sum_{s=1}^{r-j}	b^{2s}\otimes B^{r-s}(j-1) 	 
&& j+1\leq r\leq \lfloor\alpha/2\rfloor,&&\quad 1\leq j\leq \lfloor\alpha/2\rfloor-1,
\end{aligned}
\end{equation}
and
\begin{equation}	\label{eq:defAri_alpha}
\begin{aligned}
&A^r(0) = a^{2r} && 0\leq r\leq \lfloor\alpha/2\rfloor,\\
&A^r(j) = \sum_{s=0}^{r-j}	a^{2s}\otimes B^{r-s}(j-1)	 
&& j\leq r\leq \lfloor\alpha/2\rfloor,&&\quad1\leq j\leq \lfloor\alpha/2\rfloor.
\end{aligned}
\end{equation}
Associated with the tensors $A^r(j)$ and $B^r(j)$, we define for the given $N$ the functions 
\begin{subequations}	\label{eq:defRSF_alpha}
	\begin{align}
	R^\indk(j) &= \sum_{r=j}^{\indk/2} (-1)^r \epsilon^{2r} A^r(j)\nabxn {2r+2}\tildeub
	&& 0\leq j \leq \indk/2,	\label{eq:defRSF1_alpha}\displaybreak[0]\\
	S^\indk(j) &= \sum_{r=j+1}^{\indk/2} (-1)^r \epsilon^{2r} B^r(j) \nabxn {2r}\dtt\tildeub
	&& 0\leq j \leq \indk/2-1,\label{eq:defRSF2_alpha}\displaybreak[0]\\
	\tilde{F}^\indk(j) &= \sum_{r=j+1}^{\indk/2} (-1)^r \epsilon^{2r} B^r(j)\nabxn {2r} (Qf)
	&& 0\leq j \leq \indk/2-1,\label{eq:defRSF3_alpha}\displaybreak[0]\\
	F^\indk(j) &= \sum_{r=j+1}^{\indk/2} (-1)^r \epsilon^{2r} \big(- B^r(j)\big) \nabxn {2r} f
	&& 0\leq j \leq \indk/2-1.\label{eq:defRSF4_alpha}
	\end{align}
\end{subequations}
With these notations, we verify that \eqref{eq:effectiveEquation_pr0_alpha} reads (recall the definition of $Qf$ in \eqref{eq:defQf})
\begin{equation}	\label{eq:effectiveEquation_ter_alpha}
\dtt\tildeub 
= f + R^\indk(0) + S^\indk(0) + T^\indk(0) + F^\indk(0) + G^\indk(0),
\end{equation}
where the remainders are 
$T^\indk(0)=G^\indk(0)=0$ if $\indk=\alpha$ and 
\begin{equation}	\label{eq:defT0_alpha}
G^\indk(0) = \hspace{-7pt}
\sum_{r=\indk/2+1}^{\alpha/2} (-1)^r \epsilon^{2r} b^{2r} \nabxn{2r} f,
\quad
T^\indk(0) = \hspace{-7pt}
\sum_{r=\indk/2+1}^{\alpha/2} (-1)^r \epsilon^{2r} 
\Big(a^{2r}\nabxn {2r+2}\tildeub -b^{2r} \nabxn{2r} \dtt\tildeub\Big),
\end{equation}	
otherwise.

The proof is divided into three steps.
In the first step (Lemma \ref{lem:bousstricklemma1}), we apply one Boussinesq trick:
in $S^\indk(j)$, we use \eqref{eq:effectiveEquation_pr0_alpha} to substitute $\dtt\tildeub$ 
and decompose the result into terms of interest plus remainders.
In the second step (Lemma \ref{lem:bousstricklemma2}), the decomposition is used inductively to obtain a first version of Theorem \ref{thm:rewritedtttildeub_2_alpha}, which involves the tensor $\tilde\Cten^r$ instead of $\Cten^r$.
The final step is to prove that $\tilde\Cten^r$ and $\Cten^r$ are equal (Lemma \ref{lem:bousstricklemma3}).

\begin{lemma}	\label{lem:bousstricklemma1}
	The functions defined in \eqref{eq:defRSF_alpha} satisfy the relations
	(\eqref{eq:rewritedtttildeub_1a_alpha} and \eqref{eq:rewritedtttildeub_1c_alpha} hold only if $N\geq4$)
	\begin{subequations}\label{eq:rewritedtttildeub_1_alpha}
		\begin{align}		
		S^\indk(j) =\,& R^\indk(j+1) + S^\indk(j+1) + T^\indk(j+1) + \tilde{F}^\indk(j)
		&&0\leq j\leq \indk/2-2,
		\label{eq:rewritedtttildeub_1a_alpha}\\
		\tilde{F}^\indk(j) =\,& F^\indk(j+1)-F^\indk(j) + G^\indk(j+1)
		&&0\leq j\leq \indk/2-2,
		\label{eq:rewritedtttildeub_1c_alpha}\\
		S^\indk(\indk/2-1)
		=\,& R^\indk(\indk/2) + T^\indk(\indk/2) + \tilde{F}^\indk(\indk/2-1),
		\label{eq:rewritedtttildeub_1b_alpha}\\
		\tilde{F}^\indk(\indk/2-1) =\,& -F^\indk(\indk/2-1)+G^\indk(\indk/2),
		\label{eq:rewritedtttildeub_1d_alpha}
		\end{align}
	\end{subequations}
	where the remainders 
	$\{T^\indk(j)\}_{j=1}^{N/2}$ and $\{G^\indk(j)\}_{j=1}^{N/2}$ 
	are defined in \eqref{eq:defT_alpha}, \eqref{eq:defGN_1}, \eqref{eq:defTlast_alpha}, and \eqref{eq:defGlast}.
\end{lemma}

\begin{lemma}	\label{lem:bousstricklemma2}
	The effective solution of \eqref{eq:effectiveEquation_alpha} $\tildeub$ satisfies the equality
	\begin{equation}	\label{eq:bousstricklemma2}
	\dtt\tildeub = f + \sum_{r=0}^{\indk/2} \epsilon^{2r} (-1)^r\tilde\Cten^{r} \nabxn {2r+2}\tildeub 
	+ \mathcal{S}^\eps_{\indk} f
	+ \mathcal{R}^\eps_{\indk} \tildeub,
	\end{equation}
	where $\tilde\Cten^{r}$ is the tensor defined as 
	\begin{equation}	\label{eq:definitionCr_alpha}
	\tilde\Cten^{0} = a^0,
	\qquad
	\tilde\Cten^{r} = a^{2r} + \sum_{j=1}^{r} \sum_{s=0}^{r-j} a^{2s}\otimes  B^{r-s}(j-1) 
	\quad 1\leq r\leq \lfloor\alpha/2\rfloor,
	\end{equation}
	and the remainders $\mathcal{R}^\eps_{\indk}\tildeub$ and $\mathcal{S}^\eps_{\indk} f$ are defined in \eqref{eq:defReps0}
	and satisfy the estimates \eqref{eq:estimateR1eps}.
\end{lemma}

\begin{lemma}	\label{lem:bousstricklemma3}
	For $1\leq r\leq \lfloor\alpha/2\rfloor$, the tensor $\tilde\Cten^r$, defined in \eqref{eq:definitionCr_alpha}, equals the tensor $\Cten^r$, defined in \eqref{eq:definitionCr}.	
\end{lemma}

\begin{remark}	\label{rem:explanationQf3}
	Let us explain how the definition of $Qf$ is obtained.
	Assuming that $Qf$ is unknown, \eqref{eq:effectiveEquation_ter_alpha} reads
	$\dt\tildeub = Qf + R^N(0)+ S^N(0)+ T^N(0)$.
	Using inductively \eqref{eq:rewritedtttildeub_1a_alpha} and \eqref{eq:rewritedtttildeub_1b_alpha}, which hold independently of $Qf$, we obtain \eqref{eq:bousstricklemma2} with $\mathcal{S}^\eps_{\indk} f$ defined as
	\[
	\mathcal{S}^\eps_{\indk} f 
	= -f + Qf + \sum_{j=0}^{N/2-1}\sum_{r=j+1}^{N/2-1} (-1)^r \eps^{2r} B^r(j) \nabxn{2r}(Qf).
	\]
	We want to build $Qf$ so that $\mathcal{S}^\eps_{\indk} f$ has order $\mathcal{O}(\eps^{N+2})$ for all even $0\leq N\leq 2\lfloor\alpha/2\rfloor$.
	Inserting the ansatz $Qf = \sum_{s=0}^{\lfloor\alpha/2\rfloor} (-1)^s \eps^{2s} q^s \nabxn{2s} f$, where $q^s\in\Ten^{2s}(\R^d)$ are unknown, we obtain (after some work)
	\[
	\mathcal{S}^\eps_{\indk} f 
	= (q^0-1)f
	+  \sum_{m=1}^{N/2} (-1)^m \eps^{2m} 
	\Big( q^m + \sum_{s=0}^{m-1} \sum_{j=0}^{m-s-1} q^s \otimes B^{m-s}(j) \Big)
	\nabxn{2m}f
	+ \mathcal{O}(\eps^{N+2}).
	\]
	Canceling the terms of order $\mathcal{O}(1)$ to $\mathcal{O}(\eps^{N})$, we obtain $q^0=1$ and an inductive definition for $q^m$, $m\geq1$.
	We then verify by induction that $q^m = b^{2m}$ for $1\leq m\leq \lfloor\alpha/2\rfloor$,
	i.e., $Qf$ is defined as \eqref{eq:defQf}.
\end{remark}

\begin{proofof}{Lemma \ref{lem:bousstricklemma1}}
	To slightly simplify the notation, Let us denote $\bar\alpha= 2\lfloor\alpha/2\rfloor$.
	We first prove \eqref{eq:rewritedtttildeub_1a_alpha}.
	Using \eqref{eq:effectiveEquation_pr0_alpha}, we substitute $\dtt\tildeub$ in $S^\indk(j)$ and obtain
	\begin{equation}	\label{eq:bousstricklemma1_pr1}
	\begin{aligned}
	S^\indk(j)
	= 
	\phantom{ + } &\sum_{r=j+1}^{\indk/2} \sum_{ s=0}^{\indk/2}(-1)^{ r+ s}\epsilon^{2( r+ s)} \big( a^{2 s}\otimes B^{ r}(j) \big) \nabxn {2( r+ s)+2}\tildeub
	+ \tilde{F}^\indk(j)\\
	+ &	\sum_{ r=j+1}^{\indk/2} \sum_{ s=1}^{\indk/2}(-1)^{ r+ s}\epsilon^{2( r+ s)} \big( - b^{2 s} \otimes B^{ r}(j) \big) \nabxn {2( r+ s)}\dtt \tildeub
	+ T^\indk_1(j+1),
	\end{aligned}
	\end{equation}
	where the remainder is $T^\indk_1(j+1) =0$ if $\indk=\bar\alpha$ and
	\begin{equation*}
	T^\indk_1(j+1) = 
	\sum_{r=j+1}^{\indk/2} \sum_{ s=\indk/2+1}^{\bar\alpha/2}
	(-1)^{r+s}\eps^{2(r+s)}
	\Big( \big (a^{2 s}\otimes B^{ r}(j)\big) \nabxn {2( r+ s)+2}\tildeub
	- \big( b^{2 s} \otimes B^{ r}(j) \big) \nabxn {2( r+ s)}\dtt \tildeub
	\Big),
	\end{equation*}
	otherwise.
	In order to decompose the terms in \eqref{eq:bousstricklemma1_pr1}, we study the two index sets (denoting $\ell=\indk/2$)
	\begin{equation*}
	\begin{aligned}
	I^{\ell}_m(j) &= \{ (r,s)\,:\,j+1\leq r\leq {\ell},\; 0\leq s\leq {\ell} \;:\; r+s=m\},
	\\
	J^{\ell}_m(j) &= \{ (r,s)\,:\,j+1\leq r\leq {\ell},\; 1\leq s\leq {\ell} \;:\; r+s=m\},
	\end{aligned}
	\end{equation*}
	where we recall that $0\leq j\leq \ell-2$.
	We verify that these sets can be written as 
	\begin{equation*}
	\begin{aligned}
	&I^\ell_m(j) = 
	\left\{
	\begin{array}{ll}
	\big\{ (r,s)\,:\, 0\leq s \leq m-(j+1),~r=m-s\big\}
	& j+1\leq m\leq \ell,\\[2pt]
	H^\ell_m(j) :=\big\{(r,s)\,:\,	\max\{j+1, m-\ell\}\leq r\leq \ell,~s=m-r	\big\}
	& \ell+1\leq m\leq 2\ell,
	\end{array} 
	\right.
	\\
	&J^\ell_m(j) = 
	\left\{
	\begin{array}{ll}
	\big\{(r,s)\,:\, 1\leq s \leq m-(j+1),~r=m-s\big\}
	& j+2\leq m\leq \ell,\\[2pt]
	H^\ell_m(j) 
	& \ell+1\leq m\leq 2\ell.
	\end{array}
	\right.
	\end{aligned}
	\end{equation*}
	Hence, we rewrite \eqref{eq:bousstricklemma1_pr1} as 
	\begin{equation}	\label{eq:bousstricklemma1_pr2}
	\begin{aligned}
	S^\indk(j)
	= 
	&\phantom{ + } \sum_{m = j+1}^{\indk/2}(-1)^{m} \epsilon^{2m} \Big( \sum_{s = 0}^{m-(j+1)}a^{2 s}\otimes B^{ m-s}(j) \Big) \nabxn {2m+2}\tildeub
	~~+ \tilde{F}^\indk(j)\\
	& + \sum_{m = j+2}^{\indk/2}(-1)^{m} \epsilon^{2m} \Big( -\hspace{-7pt}\sum_{s = 1}^{m-(j+1)}b^{2 s}\otimes B^{ m-s}(j) \Big) \nabxn {2m}\dtt \tildeub
	~~+ T^\indk(j+1) ,
	\end{aligned}
	\end{equation}
	where the remainder is
	\begin{equation}	\label{eq:defT_alpha}
	\begin{aligned}
	T^\indk(j+1) = T^\indk_1(j+1) &+
	\sum_{m=\indk/2+1}^{\indk} (-1)^m \epsilon^{2m}
	\Big(\sum_{r,s\in H^{\indk/2}_m(j)} a^{2s}\otimes B^r(j)	\Big)
	\nabxn {2m+2}\tildeub
	\\& + 
	\sum_{m=\indk/2+1}^{\indk} (-1)^m \epsilon^{2m}
	\Big(-\hspace{-12pt}\sum_{r,s\in H^{\indk/2}_m(j)} 	b^{2s}\otimes B^r(j)	\Big)
	\nabxn {2m}\dtt\tildeub
	.
	\end{aligned}
	\end{equation}
	Using the definitions of $A^r(j)$ in \eqref{eq:defAri_alpha} and $B^r(j)$ in \eqref{eq:defBri_alpha}, we verify that the two sums in \eqref{eq:bousstricklemma1_pr2} match $R^\indk(j+1)$ and $S^\indk(j+1)$, respectively.
	That proves \eqref{eq:rewritedtttildeub_1a_alpha}.
	
	Next, let us prove \eqref{eq:rewritedtttildeub_1c_alpha}.
	Using the definition of $Qf$ in \eqref{eq:defQf} and recalling \eqref{eq:defRSF3_alpha} and \eqref{eq:defRSF4_alpha}, we have
	\begin{equation*}
	\tilde{F}^\indk(j) 
	= -F^\indk(j) 
	+ \sum_{r=j+1}^{\indk/2} \sum_{ s=1}^{\bar\alpha/2}(-1)^{ r+ s}\epsilon^{2( r+ s)} \big( b^{2 s}\otimes B^{r}(j) \big)
	\nabxn {2( r+ s)} f,
	\end{equation*}
	As done above, we decompose the sum and use the expression of $J^{\ell}_m(j)$ to obtain
	\begin{equation*}
	\tilde{F}^\indk(j) 
	= -F^\indk(j) +
	\sum_{m = j+2}^{\indk/2}(-1)^{m} \epsilon^{2m} 
	\Big( \sum_{s = 1}^{m-(j+1)}	b^{2 s}\otimes B^{ m-s}(j) \Big)
	\nabxn {2m}f + G^N(j+1),
	\end{equation*}
	where the remainder is
	\begin{equation}	\label{eq:defGN_1}
	\begin{aligned}
	G^\indk(j+1)
	= &\sum_{r=j+1}^{\indk/2} \sum_{ s=\indk/2+1}^{\bar\alpha/2}
	(-1)^{r+s}\eps^{2(r+s)} \big(b^{2 s}\otimes B^{ m-s}(j)\big)\nabxn {2m}f
	\\&+ \sum_{m = \indk/2+1}^{\indk/2}	
	(-1)^{m} \epsilon^{2m} \Big(\hspace{-10pt}\sum_{r,s\in H^{\indk/2}_m(j)} 	b^{2s}\otimes B^r(j)	\Big)
	\nabxn {2m}f.
	\end{aligned}
	\end{equation}
	Using the definitions of $B^{m}(j+1)$ in \eqref{eq:defBri_alpha} and $F^N(j+1)$ in \eqref{eq:defRSF3_alpha}, we obtain \eqref{eq:rewritedtttildeub_1c_alpha}. 
	
	Next, we prove \eqref{eq:rewritedtttildeub_1b_alpha}. 
	From the definition of $S^\indk(j)$ in \eqref{eq:defRSF2_alpha}, we use \eqref{eq:effectiveEquation_pr0_alpha} to get
	\begin{equation}	\label{eq:bousstricklemma1_pr3}
	S^\indk(\indk/2-1)
	= 	
	(-1)^{\indk/2} \epsilon^{\indk} \big( a^0 \otimes B^{\indk/2}(\indk/2-1) \big) \nabxn {\indk+2}\tildeub 
	+T^\indk(\indk/2)
	+\tilde F^\indk(\indk/2-1)
	,
	\end{equation} 
	where the remainder is 
	\begin{equation}	\label{eq:defTlast_alpha}
	T^\indk(\indk/2) =
	\sum_{r=1}^{\bar\alpha/2} (-1)^{\indk/2+r} \epsilon^{\indk+2r}
	\Big(
	\big(a^{2r}\otimes B^{\indk/2}(\indk/2-1) \big)
	\nabxn {\indk +2r+2}\tildeub
	-\big( b^{2r}\otimes B^{\indk/2}(\indk/2-1)\big)
	\nabxn {\indk +2r}\dtt\tildeub
	\Big),
	\end{equation}
	The first term of the right-hand side of \eqref{eq:bousstricklemma1_pr3} matches the definition of $R^\indk(\bar\alpha/2)$ in \eqref{eq:defRSF1_alpha}.
	That proves \eqref{eq:rewritedtttildeub_1b_alpha}.
	
	Finally, let us prove \eqref{eq:rewritedtttildeub_1d_alpha}.
	From the definition of $\tilde{F}^\indk(j)$, we verify that \eqref{eq:rewritedtttildeub_1d_alpha} holds
	with $G^\indk(\indk/2)$ defined as
	\begin{equation}	\label{eq:defGlast}
	G^\indk(\indk/2) 
	= 	\sum_{r=1}^{\bar\alpha/2} \eps^{N+2r}\Big( b^{2r}\otimess B^{N/2}(N/2-1)\Big) \nabxn{N+2r} f. 
	\end{equation}
	That concludes the proof of Lemma \ref{lem:bousstricklemma1}.
\end{proofof}

\begin{proofof}{Lemma \ref{lem:bousstricklemma2}}
	Starting from \eqref{eq:effectiveEquation_ter_alpha}, we use inductively the decompositions in \eqref{eq:rewritedtttildeub_1_alpha} and obtain
	\begin{equation}	\label{eq:bousstricklemma2_pr1}
	\dtt\tildeub 
	= f + \sum_{j=0}^{\indk/2} R^\indk(j)
	+ \sum_{j=0}^{\indk/2} T^\indk(j) 
	+ \sum_{j=0}^{\indk/2} G^\indk(j).
	\end{equation}
	We define the remainders as  
	\begin{equation}	\label{eq:defReps0}
	\mathcal{R}^\eps_{\indk}\tildeub = \sum_{j=0}^{\indk/2} T^\indk(j),
	\qquad
	\mathcal{S}^\eps_{\indk}f = \sum_{j=0}^{\indk/2} G^\indk(j),
	\end{equation}
	where $\{T^\indk(j)\}_{j=0}^{\indk/2}$ and $\{G^\indk(j)\}_{j=0}^{\indk/2}$ are defined in 
	\eqref{eq:defT0_alpha},
	\eqref{eq:defT_alpha},
	\eqref{eq:defGN_1},
	\eqref{eq:defTlast_alpha},
	and \eqref{eq:defGlast}.
	From the definitions of $T^\indk(j)$ and $G^\indk(j)$, we verify that $\mathcal{R}^\eps_{\indk}\tildeub$ and $\mathcal{S}^\eps_{\indk}f$ satisfy the estimates in \eqref{eq:estimateR1eps}.
	Next, we have to develop $\sum_{j=0}^{\indk/2} R^\indk(j)$ in \eqref{eq:bousstricklemma2_pr1}.
	Using the definition of $R^\indk(j)$ in \eqref{eq:defRSF1_alpha} and exchanging the sums, we find that
	\begin{equation}	\label{eq:bousstricklemma2_pr2}
	\sum_{j=0}^{\indk/2} R^\indk(j)
	= \sum_{j=0}^{\indk/2} \sum_{r=j}^{\indk/2} (-1)^r \eps^{2r} A^{r}(j) \nabxn {2r+2}\tildeub
	= \sum_{r=0}^{\indk/2}  \eps^{2r} (-1)^r \bigg( \sum_{j=0}^r A^{r}(j) \bigg) \nabxn {2r+2}\tildeub.
	\end{equation}
	Using the definition of $A^r(j)$ in \eqref{eq:defAri_alpha}, we verify that 
	for $r=0$, $\sum_{j=0}^r A^{r}(j)=a^0=\tilde\Cten^0$, and for $1\leq r\leq \alpha/2$,
	\begin{equation}	\label{eq:bousstricklemma2_pr3}
	\sum_{j=0}^r A^{r}(j)
	= A^{r}(0)  + \sum_{j=1}^r A^{r}(j)
	= a^{2r} +  \sum_{j=1}^r\sum_{s=0}^{r-j} a^{2s}	\otimes B^{r-s}(j-1)
	= \tilde\Cten^r.
	\end{equation}	
	Combining \eqref{eq:bousstricklemma2_pr1}, \eqref{eq:defReps0}, \eqref{eq:bousstricklemma2_pr2}, and \eqref{eq:bousstricklemma2_pr3},
	we obtain \eqref{eq:bousstricklemma2} and the proof of Lemma \ref{lem:bousstricklemma2} is complete.
\end{proofof}

\begin{proofof}{Lemma \ref{lem:bousstricklemma3}}
	We prove by induction on $r$ that the tensor $\tilde\Cten^r$, defined in \eqref{eq:definitionCr_alpha},
	equals the tensor $\Cten^r$, defined in \eqref{eq:definitionCr}.
	The base case is trivially verified as $\tilde\Cten^0 = a^0 = \Cten^0$.  
	Then, for $r\geq 2$, we assume that $\tilde\Cten^s=\Cten^s$ for $1\leq s \leq r-1$
	and we have to verify that the tensor
	\begin{equation*}
	\Cten^r = a^{2r} - \sum_{\ell=0}^{r-1} \Cten^\ell\otimess b^{2(r-\ell)} 
	=a^{2r} - \Cten^0\otimess b^{2r}  - \sum_{\ell=1}^{r-1} \Cten^\ell\otimess b^{2(r-\ell)} ,
	\end{equation*}
	equals $\tilde\Cten^r$, defined in \eqref{eq:definitionCr_alpha}.
	Using the induction assumption and \eqref{eq:definitionCr_alpha}, we have
	\begin{equation}	\label{eq:proofdefCr_bis_alpha}
	\Cten^r
	=  a^{2r} - \sum_{\ell=0}^{r-1} a^{2\ell} \otimess b^{2(r-\ell)}
	- \sum_{\ell=1}^{r-1} \sum_{j=1}^\ell \sum_{s=0}^{\ell-j}
	a^{2s}\otimess  B^{\ell-s}(j-1) \otimess b^{2(r-\ell)}   .
	\end{equation}
	Let us denote the triple sum $T$ and its summand 
	$x^r_{\ell,j,s} = a^{2s}\otimess  B^{\ell-s}(j-1) \otimess b^{2(r-\ell)}$.
	Applying the change of indices $m=r-\ell$ and exchange the sums twice to get
	\begin{equation*}
	T 	= \sum_{m=1}^{r-1} \sum_{j=1}^{r-m} \sum_{s=0}^{r-m-j} x^r_{r-m,j,s}
	= \sum_{j=1}^{r-1} \sum_{m=1}^{r-j} \sum_{s=0}^{r-m-j} x^r_{r-m,j,s}
	= \sum_{j=1}^{r-1} \sum_{s=0}^{r-j-1} \sum_{m=1}^{r-j-s}  x^r_{r-m,j,s}. 
	\end{equation*}
	Using the definition of $B^r(j)$ in \eqref{eq:defBri_alpha}, we then have
	\[
	T 
	= \sum_{j=1}^{r-1} \sum_{s=0}^{r-j-1} a^{2s}\otimess 
	\bigg(  \sum_{m=1}^{r-s-j}  B^{r-s-m}(j-1)\otimess  b^{2m} \bigg) 
	= -\sum_{j=1}^{r-1} \sum_{s=0}^{r-j-1} a^{2s}  \otimess B^{r-s}(j).
	\] 
	We change the index $k=j+1$ in this equality to obtain from \eqref{eq:proofdefCr_bis_alpha} (recall that $B^r(0)=-b^{2r}$)
	\[
	\Cten^r
	= a^{2r} + \sum_{s=0}^{r-1} a^{2s} \otimess  B^{r-s}(0) + \sum_{k=2}^{r} \sum_{s=0}^{r-k} a^{2s}\otimess B^{r-s}(k-1) 
	= a^{2r} + \sum_{k=1}^{r} \sum_{s=0}^{r-k}a^{2s}\otimess  B^{r-s}(k-1) .
	\]
	This expression matches the definition of $\tilde\Cten^r$ in \eqref{eq:definitionCr_alpha}.
	Hence, we have proved that $\tilde\Cten^r=\Cten^r$ for all $0\leq r\leq \lfloor\alpha/2\rfloor$ and the proof of Lemma \ref{lem:bousstricklemma3} is complete.
\end{proofof}

\subsubsection{Use of the inductive Boussinesq tricks}	\label{subsec:proofbousstrick2}

In the asymptotic expansion in Section \ref{subsec:derivationFamily_subsec}, we need to substitute the terms involving $\dtt\tildeub$ in the development of $r^\eps$ in \eqref{eq:waveAsymptoticDeveloppment_alpha}.
This task is performed in the following result, obtained thanks to Theorem \ref{thm:rewritedtttildeub_2_alpha}.

\begin{lemma}	\label{lem:inductiveBoussinesqTricks_alpha}
	Define tensors $\Pten^{k}\in\Ten^{k}(\R^d)$ as
	\begin{equation}	\label{eq:definitionPk}
	\Pten^{2r} = (-1)^r c^r,
	\qquad
	\Pten^{2r+1} = 0,
	\qquad
	0\leq r\leq \lfloor\alpha/2\rfloor,
	\end{equation}
	where $c^r$ is defined in \eqref{eq:definitionCr}
	and let $\chi^0=1$ and $\{\chi^k\}_{k=1}^\alpha$ be the tensor functions in \eqref{eq:adaptation2_alpha}.
	Then the following equality holds
	\begin{equation}	\label{eq:inductiveBoussinesqTricks_alpha}	
	\begin{aligned}
	\dtt\tildeub 
	+\sum_{k=1}^\alpha \eps^k\chi^k\nabxn{k}\dtt \tildeub
	=  f 
	+\sum_{k=0}^{\alpha}\eps^{k} \bigg(  \sum_{j=0}^{k}\Pten^{k-j}\otimes\chi^j \bigg) \nabxn {k+2}\tildeub 
	+\mathcal{S}^\eps f+ \Reps\tildeub,
	\end{aligned}
	\end{equation}
	where the remainders $\mathcal{S}^\eps f$ and $\Reps\tildeub$ are defined in \eqref{eq:defSReps}, \eqref{eq:defSReps_Nk}.
	In particular, provided sufficient regularity of $\tildeub$, $\Reps\tildeub$ has order $\mathcal{O}(\eps^{\alpha+1})$ in the $\Li(0,\eps^{-\alpha}T; \LdOm)$ norm.
\end{lemma}

\begin{proofof}{Lemma \ref{lem:inductiveBoussinesqTricks_alpha}}
	Let us define $T \coloneqq\sum_{k=0}^\alpha \eps^k\chi^k\nabxn{k}\dtt\tildeub$.
	Using Theorem \ref{thm:rewritedtttildeub_2_alpha} and the definition of $\Pten^k$, we obtain for any choice of even integers $0\leq N(k)\leq\alpha$:
	\begin{equation}	\label{eq:inducBousTrickProof_1}
	T 
	= \sum_{k=0}^{\alpha} \epsilon^{k} \chi^k \nabxn{k}
	\bigg(f + \sum_{r=0}^{N(k)/2} \epsilon^{2r} \Pten^{2r} \nabxn{2r+2}\tildeub 
	+ \mathcal{S}^\eps_{N(k)}f + \mathcal{R}^\eps_{N(k)}\tildeub\bigg)
	=
	f +
	\sum_{k=0}^{\alpha} \sum_{r=0}^{N(k)/2} z^{2r}_k
	+\mathcal{S}^\eps f +\mathcal{R}^\eps \tildeub,
	\end{equation}
	where $z^j_k$ and the remainders are defined as
	\begin{equation}	\label{eq:defSReps}
	z^i_k = \epsilon^{k+i} \Pten^{i}\otimes\chi^k \nabxn{k+i+2}\tildeub,
	\quad
	\mathcal{S}^\eps f = -f+\sum_{k=0}^{\alpha} \epsilon^{k} \chi^k \nabxn{k}\big( f + \mathcal{S}^\eps_{N(k)} f\big),
	\quad
	\mathcal{R}^\eps\tildeub = \sum_{k=0}^{\alpha} \epsilon^{k} \chi^k \nabxn{k}\big(\mathcal{R}^\eps_{N(k)}\big) \tildeub,
	\end{equation}
	with $\mathcal{S}^\eps_{N(k)} f$ and $\mathcal{R}^\eps_{N(k)} \tildeub$ the remainders defined in \eqref{eq:defReps0}.
	The first estimate in \eqref{eq:estimateR1eps} implies that  
	\begin{equation*}
	\norm{\mathcal{R}^\eps\tildeub}_{\Li(\eps^{-\alpha}T; \LdOm)}
	\leq C(\tildeub) \sum_{k=0}^\alpha \eps^{k+N(k)+2}.
	\end{equation*}
	Hence, as our goal is for this quantity to have order $\mathcal{O}(\eps^{\alpha+1})$ (see the discussion in Section \ref{subsec:derivationFamily_subsec}), we need to set $N(k)$ such that
	\[
	k+N(k)+2 \geq \alpha+1
	\qquad
	\Leftrightarrow
	\qquad
	N(k) \geq \alpha-k-1.
	\]
	Dealing with even and odd indices separately, we verify that the smallest even integer satisfying the above inequality is 
	\begin{equation}	\label{eq:defSReps_Nk}
	N(2\ell) = 2\big(\lfloor\alpha/2\rfloor - \ell\big),
	\qquad
	N(2\ell-1) = 2\big(\lceil\alpha/2\rceil - \ell\big)
	\end{equation}
	We now rewrite the double sum in the right-hand side of \eqref{eq:inducBousTrickProof_1} separately for even and odd index $k$.
	For the sum on even index $\big\{k=2\ell \,:\, 0\leq \ell\leq \lfloor\alpha/2\rfloor\big\}$, we have
	\[
	\sum_{\ell=0}^{\lfloor\alpha/2\rfloor} \sum_{r=0}^{\lfloor\alpha/2\rfloor-\ell} z^{2r}_{2\ell}
	= \sum_{\ell=0}^{\lfloor\alpha/2\rfloor} \sum_{m=\ell}^{\lfloor\alpha/2\rfloor} z^{2(m-\ell)}_{2\ell}
	= \sum_{m=0}^{\lfloor\alpha/2\rfloor} \bigg(	
	\sum_{\ell=0}^{m}z^{2m-2\ell}_{2\ell}
	+ \sum_{\ell=1}^{m} z^{2m-(2\ell-1)}_{2\ell} \bigg)
	.
	\]
	where in the first equality we changed the index $m=r+\ell$ and in the second we exchanged the order of summation and used that $z^{2s+1}_k=0$ for any $s,k$.
	Similarly, for the sum on odd index $\big\{k=2\ell-1 \,:\, 1\leq \ell\leq \lceil\alpha/2\rceil\big\}$, we have
	\[
	\sum_{\ell=1}^{\lceil\alpha/2\rceil} \sum_{r=0}^{\lceil\alpha/2\rceil-\ell} z_{2\ell-1}^{2r}
	= \sum_{\ell=1}^{\lceil\alpha/2\rceil} \sum_{m=\ell}^{\lceil\alpha/2\rceil} z_{2\ell-1}^{2(m-\ell)}
	= \sum_{m=1}^{\lceil\alpha/2\rceil} \bigg(
	\sum_{\ell=1}^{m} z_{2\ell-1}^{2m-1-(2\ell-1)} + \sum_{\ell=0}^{m-1} z_{2\ell}^{2m-1-2\ell}
	\bigg)
	.
	\]
	Using the two last equalities in \eqref{eq:inducBousTrickProof_1}, we gather the sums to find
	\[
	T
	= f+\sum_{m=0}^{\lfloor\alpha/2\rfloor} \sum_{j=0}^{2m} z_{j}^{2m-j}+
	\sum_{m=1}^{\lceil\alpha/2\rceil} \sum_{j=0}^{2m-1} z_{j}^{2m-1-j}
	+\mathcal{S}^\eps f +\mathcal{R}^\eps \tildeub
	= 
	f+
	\sum_{k=0}^{\alpha} \sum_{j=0}^{k} z_{j}^{k-j}
	+\mathcal{S}^\eps f +\mathcal{R}^\eps \tildeub.
	\]
	Replacing $z^{k-j}_j$ by its definition in \eqref{eq:defSReps}, we obtain \eqref{eq:inductiveBoussinesqTricks_alpha} and that concludes the proof of Lemma \ref{lem:inductiveBoussinesqTricks_alpha}.
\end{proofof}

\subsection{Proof of the a priori error estimate for the family of effective equations (Theorem \ref{thm:est_epsm_alpha})}	\label{subsec:summaryproof}

In this section, we prove Theorem \ref{thm:est_epsm_alpha}.
Note that the main ingredient of the proof is an adaptation operator based on the adaptation constructed in Section \ref{subsec:derivationFamily_subsec}.

Let us first introduce some notations.
We use a bracket $\eqcbiss v$ to denote the equivalence class of $v\in\Ld\ofOm$ in $\calLd\ofOm$,
and a bold face letter $\eqc v$ to denote elements of $\calWper\of\Om$.
Note that the Hilbert space $\calLd\ofOm$ is equipped with the inner product
$\psbig{\eqcbiss v,\eqcbiss w}_{\calLd\ofOm} 
= \psbig{v- \intMean{v}_\Om,w -\intMean{w}_\Om}_{\Ld\ofOm}
$.
We define the following norm on $\calWper\of\Om$
\begin{equation*}	
\norm{\eqc{w}}_{\calWnorm} =
\inf_{\substack{\eqc{w}=\eqc w_1 +\eqc w_2\\ \eqc w_i = [w_i] \in\calWper\of\Om}}
\Big\{ 		\norm{\eqcbiss{w_1}}_{\calLd\of\Om} + \norm{\nabla w_2}_{\Ld\of\Om} 	\Big\}
\quad \forall \eqc w\in\calWper\of\Om.
\end{equation*}
In particular, we verify that a function $w\in\WperOm$ satisfies 
$\norm{w}_{\Wnorm} = \norm{\eqcbiss{w}}_{\calWnorm}$, 
where $\norm{{\cdot}}_\Wnorm$ is defined in \eqref{eq:definitionNormW}.

The proof of Theorem \ref{thm:est_epsm_alpha} is structured as follows.
First, based on the adaptation $\Beps\tildeub$ defined by \eqref{eq:adaptation1_alpha} and \eqref{eq:adaptation2_alpha}, we define the adaptation operator $\BBeps\tildeub$.
Then, using the triangle inequality, we split the error as
\begin{equation} \label{eq:errorSplit}
\norm{\ueps - \tildeub}_{\Li(\Wnorm)}
= \norm{\eqcbiss{\ueps - \tildeub}}_{\Li(\calWnorm)}
\leq 
\norm{\eqcbiss{\ueps} - \BBeps\tildeub}_{\Li(\calWnorm)}
+ \norm{\BBeps\tildeub-\eqcbiss{\tildeub} }_{\Li(\calWnorm)},
\end{equation}
and estimate the two terms of the right-hand side separately.

Let us first discuss the consequences of the assumptions made in the theorem.
The fact that the effective equation belongs to the family $\famEfEq$ (Definition \ref{def:familyofeffectivesolution_alpha}) has two major implications:
first, the equation is well-posed;
second, the tensors $\{a^{2r},b^{2r}\}_{r=1}^{\lfloor\alpha/2\rfloor}$ satisfy the constraints \eqref{eq:constainta2rb2r_alpha} and thus the cell problems \eqref{eq:cellProblemsWave_compact} are well-posed.
Hence, we have the existence and uniqueness of the effective solution $\tildeub$ and of the correctors $\chi^1,\hdots,\chi^{\alpha+2}$.
Inductively, we can show that the correctors satisfy the following bound
\[
\norm{\chi^k_{i_1\hdots i_k}}_{\Hu(Y)} \leq C\norm{a}_{\Li(Y)},
\]
where the constant $C$ depends only on the ellipticity constant $\lambda$ and the reference cell $Y$.
Let us then investigate the regularity of $\tildeub$ and $f$.
As we assume $d\leq 3$, the embedding 
$\Hk{2}_\per(\Om)\contEmb \fcnC^0_\per(\closure\Om)$ 
is continuous and the regularity assumption ensures that 
\begin{gather*}	
f\in\Ld\Lti{\fcnC^{r(\alpha)-2}_\per(\closure\Om)},\\
\tildeub\in\Li\Lti{\fcnC^{r(\alpha)}_\per(\closure\Om)},
\quad
\dt\tildeub\in\Li\Lti{\fcnC^{r(\alpha)-1}_\per(\closure\Om)},
\quad
\dtt\tildeub\in\Li\Lti{\fcnC^{r(\alpha)-2}_\per(\closure\Om)},
\end{gather*}
where we recall that $r(\alpha) = \alpha+ 2\lfloor\alpha/2\rfloor + 2$ (as $\alpha\geq 2$ we have $r(\alpha)\geq \alpha+4$).

We are now able to define the adaptation operator.
Thanks to the regularity of $f$ and the correctors, $\mathcal{S}^\eps f$ defined in \eqref{eq:defSReps} belongs to $\Ld(0,\eps^{-\alpha} T;\calLdOm)$.
We can thus define $\eqc\varphi$ as the unique solution in $\Li(0,\eps^{-\alpha} T;\calWperOm)$ of 
\begin{equation*}	
\begin{array}{ll}
(\dtt+\Aeps)\eqc\varphi(t) = 
-\bbrackl \mathcal{S}^\eps f(t) \bbrackr
&\text{in }\calWperStarOm~ \text{a.e.}~~ t\in[0,\eps^{-\alpha}T],\\[2pt]
\eqc\varphi(0)=\dt\eqc\varphi(0) = \eqcbis{0},
\end{array}
\end{equation*}
where $\Aeps$ is defined in \eqref{eq:defSmallreps}.
The adaptation operator is then defined as
\[
\BBeps : \Li(0,\eps^{-\alpha}T; \fcnC^{r(\alpha)-2}_\per(\closure\Om)) \to \Ld(0,\eps^{-\alpha}T; \calWperStarOm),
\qquad \BBeps\tildeub = \eqcbis{\Beps\tildeub} + \eqc\varphi,
\]
where $\Beps\tildeub$ is the adaptation constructed in Section \ref{subsec:derivationFamily_subsec} (see \eqref{eq:adaptation1_alpha} and \eqref{eq:adaptation2_alpha}):
\[
\Beps\tildeub(t,x) = \sum_{k=1}^{\alpha+2}\eps^k \chi^k\ofxsureps \nabxn k\tildeub(t,x).
\]
Assumption \eqref{eq:assumptionOmega} ensures that $x\mapsto \chi^k\ofxsureps \nabxn k\tildeub(t,x)$ is $\Omega$-periodic.
Using the regularity of the correctors and $\tildeub$, we verify that the adaptation
\[
\BBeps\tildeub \in\Li(0,\eps^{-\alpha} T;\calWperOm),
\quad
\dt\BBeps\tildeub \in\Li(0,\eps^{-\alpha} T;\calLdOm),
\quad
\dtt\BBeps\tildeub \in\Li(0,\eps^{-\alpha} T;\calWperStarOm),
\]
(notice that the regularity of $\eqc\varphi$ and its derivatives are weaker than that of $\Beps\tildeub$).

Define then $\eqc \eta^\eps(t) = \BBeps\tildeub(t)-\eqcbiss{\ueps(t)}$
and recall that we developed $r^\eps=(\dtt+\Aeps) (\Beps\tildeub -\ueps)$ in Section \ref{subsec:derivationFamily_subsec}.
The development in \eqref{eq:waveAsymptoticDeveloppment_2_alpha} together with the cell problems \eqref{eq:cellProblemsWave_compact} ensure that
\begin{equation} \label{eq:proofapee_1}
(\dtt+\Aeps) \eqc\eta^\eps(t)
= \eqcbis{r^\eps(t)} +(\dtt+\Aeps)\eqc\varphi(t)
= \eqcbis{r^\eps(t) - \mathcal{S}^\eps f(t)}
= \eqcbis{\mathcal{R}_{\mathrm{ini}}^\eps\tildeub(t) + \mathcal{R}^\eps\tildeub(t)},
\end{equation}
where $\mathcal{R}_{\mathrm{ini}}^\eps\tildeub(t)$ is defined in \eqref{eq:Repsini}
and $\mathcal{R}^\eps\tildeub$ in \eqref{eq:defSReps}.
The following error estimate can then be used to quantify $\eqc\eta^\eps$ (see \cite[Corollary 2.2]{AbP16} or \cite[Corollary 4.2.2]{Pou17}).
\begin{lemma}	\label{lem:errorestimateindepoftimeanddomain}
	\sloppy
	Assume that $\eqc\eta\in\Li(0,\eps^{-\alpha}T;\calWperOm)$, with $\dt\eqc\eta\in\Li(0,\eps^{-\alpha}T;\calLdOm)$, $\dtt\eqc\eta\in\Ld(0,\eps^{-\alpha}T;\calWperStarOm)$ satisfies
	\begin{equation*}
	\begin{array}{l}
	\dtt\eqc\eta(t)+\Aeps \eqc\eta(t) = \eqc{r}(t) \quad \text{ in } \calWperStarOm~~\forae t\in[0,\eps^{-\alpha}T],\\[2pt]
	\eqc\eta(0) = \eqc\eta^0,\quad \dt\eqc\eta(0) = \eqc\eta^1,
	\end{array}
	\end{equation*}
	where $\eqc\eta^0\in\calWperOm$, $\eqc\eta^1\in\calLdOm$,
	and 
	$\eqc{r} \in\Lu(0,\eps^{-\alpha}T;\calLdOm)$.
	Then the following estimate holds
	\begin{equation*}
	\norm{\eqc\eta}_{\Li(0,\eps^{-\alpha}T;\calWnorm)} 
	\leq C(\lambda) \big( \norm{\eqc\eta^1}_\calLdOm + \norm{\eqc\eta^0}_\calLdOm 
	+ \norm{\eqc{r}}_{\Lu(0,\eps^{-\alpha}T;{\calLdOm})}\big),
	\end{equation*}
	where $C(\lambda)$ depends only on the ellipticity constant $\lambda$.
	If in addition, $\eqc{r} \in\Li(0,\eps^{-\alpha}T;\calLdOm)$, then
	\begin{equation*}
	\norm{\eqc\eta}_{\Li(0,\eps^{-\alpha}T;\calWnorm)} 
	\leq C(\lambda) \big( \norm{\eqc\eta^1}_\calLdOm + \norm{\eqc\eta^0}_\calLdOm 
	+ \eps^{-\alpha}T\norm{\eqc{r}}_{\Li(0,\eps^{-\alpha}T;{\calLdOm})}\big).
	\end{equation*}
\end{lemma}

In order to estimate the remainder terms in \eqref{eq:proofapee_1}, we also need the following result.

\begin{lemma}	\label{lem:estimate_dct}
	Let $\Yperfunc\in\Ldper(Y)$ and $v\in\Hk{2}_\per(\Omega)$.
	Then the following estimate holds
	\begin{equation}	\label{eq:lem:estimate_dct}
	\normbig{\Yperfunc\big(\sureps{\cdot}\big) v}_{\Ld(\Om)} \leq C \norm{\Yperfunc}_{\Ld(Y)} \norm{v}_{\Hk{2}(\Omega)},
	\end{equation}
	where the constant $C$ depends only on $Y$ and $d$.
\end{lemma}
\begin{proof}
	Recall that $Y=(0,\ell_1)\times\cdots\times(0,\ell_d)$ and $\Omega = (\omega_1^l,\omega_1^r)\times\cdots\times(\omega^l_d,\omega^r_d)$.
	As $\Omega$ satisfies \eqref{eq:assumptionOmega}, the numbers $N_i=\frac{\omega_i^l-\omega_i^r}{\ell_i\eps}$ are integers and the cells constituting $\Omega$ belongs to the set 
	$\{ \eps(n\cdot\ell + Y) : 0\leq n_i \leq N_i-1\}$.
	Denoting $\Xi = \{\xi = n\cdot \ell : 0\leq n_i \leq N_i-1\}$, the domain $\Omega$ satisfies
	\begin{equation}	\label{eq:estimateproof25_dct}
	\Omega = \mathrm{int}\bigg( \bigcup_{\xi\in\Xi} \eps (\xi + \closure{Y}) \bigg).
	\end{equation}
	Hence, almost every $x\in\Omega$ can be written as $x=\eps(\xi+y)$ for some $\xi\in\Xi, y\in Y$.
	For such triplet $(x,\xi,y)$, the $Y$-periodic function $\Yperfunc$ satisfies
	$\Yperfunc\ofxsureps = \Yperfunc(\xi+y) = \Yperfunc(y)$.
	Let $Z\subset\R^d$ be an open set with a $\fcnC^{1}$ boundary, that contains $Y$ 
	and is contained in its neighborhood, i.e., 
	\[
	Y\subset 
	Z \subset N_Y = (-\ell_1,2\ell_1)\times\cdots\times(-\ell_d,2\ell_d).
	\]
	As $Z$ has a $\fcnC^1$ boundary and $d\leq 3$, Sobolev embedding theorem ensures that the embedding $\Hk{2}(Z)\contEmb\fcnC^0(\closure{Z})$ is continuous.
	Hence, there exists a constant $C_{Y}$, depending only on $Y$, such that
	\begin{equation}	\label{eq:estimateproof275_dct}
	\norm{w}_{\fcnC^0(\closure{Y})}
	\leq \norm{w}_{\fcnC^0(\closure{Z})}
	\leq C_Y \norm{w}_{\Hk{2}(Z)}
	\leq C_Y \norm{w}_{\Hk{2}(N_Y)}
	\qquad\forall w\in\Hk{2}(N_Y).
	\end{equation}
	We now prove \eqref{eq:lem:estimate_dct}.
	Using \eqref{eq:estimateproof25_dct}, we have
	\begin{equation*}
	\normbig{\Yperfunc\big(\sureps{\cdot}\big) v}_{\Ld(\Om)}^2
	= \sum_{\xi\in\Xi} 	\int_{\eps(\xi+Y)} \Big| \Yperfunc\ofxsureps v(x)\Big|^2 \,\d x
	= \sum_{\xi\in\Xi} 	\int_{Y} \Big| \Yperfunc(y) v\big(\eps(\xi+y)\big)\Big|^2  \eps^d\,\d y,
	\end{equation*}
	where we made the change of variables
	$x=\eps(\xi+y)$.
	As $v\in\Hk{2}_\per(\Om)\contEmb\fcnC^0_\per(\closure\Om)$, we have
	\begin{equation}	\label{eq:estimateproof3_dct}
	\normbig{\Yperfunc\big(\sureps{\cdot}\big) v}_{\Ld(\Om)}^2
	\leq \norm{\Yperfunc}_{\LdY}^2 
	\sum_{\xi\in\Xi} \eps^d \norm{v_{\xi,\eps}}_{\fcnC^0(\closure{Y})}^2,
	\end{equation}
	where $v_{\xi,\eps}$ is the function of $\fcnC^0(\closure{Y})$ defined by $v_{\xi,\eps}(y) = v\big(\eps(\xi+y)\big)$.
	Using \eqref{eq:estimateproof275_dct} gives
	$\norm{v_{\xi,\eps}}_{\fcnC^0(\closure{Y})}\leq C_Y \norm{v_{\xi,\eps}}_{\Hk{2}(N_Y)}$.
	Furthermore, we have
	\[
	\eps^d\norm{v_{\xi,\eps}}_{\Hk{2}(N_Y)}^2
	= \int_{N_Y}{}  |v_{\xi,\eps}(y) |^2  \eps^d\d y
	+\int_{N_Y}{}  |\naby v_{\xi,\eps}(y) |^2  \eps^d\d y
	+\int_{N_Y}{}  |\nabla^2_{\!y} v_{\xi,\eps}(y) |^2  \eps^d\d y.
	\]
	As $\pa_{y_i}v_{\xi,\eps} = \eps \pa_{x_i}v$ and $\pa^2_{y_{ij}}v_{\xi,\eps} = \eps^2 \pa^2_{x_{ij}}v$, 
	the change of variable $x=\eps(\xi+y)$ leads to
	\begin{equation*}	
	\normbig{\Yperfunc\big(\sureps{\cdot}\big) v}_{\Ld(\Om)}^2
	\leq C\norm{\Yperfunc}_{\LdY}^2 
	\sum_{\xi\in\Xi} \norm{v}_{\Hk{2}(\eps(\xi+N_Y))}^2
	\leq C(6d-3) \norm{\Yperfunc}_{\LdY}^2 
	\sum_{\xi\in\Xi} \norm{v}_{\Hk{2}(\eps(\xi+Y))}^2,
	\end{equation*}
	where we used that every cell $\eps(\xi+Y)$ belongs to the neighborhoods of $(6d-3)$ cells (including itself).
	This proves \eqref{eq:lem:estimate_dct} and the proof of the lemma is complete.
\end{proof}

With Lemma \ref{lem:estimate_dct}, we can estimate the remainders in \eqref{eq:proofapee_1}.
Using \eqref{eq:defSReps} and estimate \eqref{eq:estimateR1eps}, we obtain
\begin{equation*}	
\norm{\mathcal{R}^\eps\tildeub}_\Ld
\leq \sum_{k=0}^{\alpha} \epsilon^{k} 
\norm{\chi^k}_\Ld  \norm{\nabxn{k}\mathcal{R}^\eps_{N(k)} \tildeub}_{\Hk{2}},
\leq C\sum_{k=0}^{\alpha} \epsilon^{k+N(k)+2} \!\!\!\!
\sum_{j=N(k)+k+2}^{2\lfloor\alpha/2\rfloor+N(k)+k} 
\Big(\norm{\nabxn{j+2}\tildeub}_{\Hk{2}} + \norm{\nabxn{j}\dtt\tildeub}_{\Hk{2}}	\Big).
\end{equation*}
From \eqref{eq:defSReps_Nk}, we verify that $\alpha-1\leq N(k)+k \leq \alpha$
and thus
\begin{equation}	\label{eq:proofapee_2}	
\norm{\mathcal{R}^\eps\tildeub}_\Ld
\leq C\epsilon^{\alpha+1} \sum_{j=\alpha+1}^{2\lfloor\alpha/2\rfloor+\alpha} 
\Big(\norm{\nabxn{j+2}\tildeub}_{\Hk{2}} + \norm{\nabxn{j}\dtt\tildeub}_{\Hk{2}}	\Big).
\end{equation}
Similarly, using \eqref{eq:estimateR1eps} to estimate 
$\mathcal{R}^\eps_{\mathrm{ini}}$ \eqref{eq:Repsini}
and $\mathcal{S}^\eps f$ \eqref{eq:defSReps}, we verify that
\begin{equation}	\label{eq:proofapee_3}
\norm{\mathcal{R}_{\mathrm{ini}}^\eps\tildeub}_\Ld
\leq 	C\eps^{\alpha+1} \sum_{j=\alpha+1}^{\alpha+2} \Big(\norm{\nabxn{j+2}\tildeub}_{\Hk{2}} + \norm{\nabxn{j}\dtt\tildeub}_{\Hk{2}}	\Big),
\qquad
\norm{\mathcal{S}^\eps f}_\Ld
\leq C\epsilon \sum_{j=1}^{2\lfloor\alpha/2\rfloor+\alpha}\norm{\nabxn{j}f}_{\Hk{2}}.
\end{equation}

\begin{remark} \label{rem:explanationQf4}
	From \eqref{eq:defSReps}, we verify that
	\[
	\mathcal{S}^\eps f 
	= \sum_{k=1}^\alpha \eps^k \chi^k\nabxn k f 
	+ \sum_{k=0}^\alpha \eps^k \chi^k \nabxn k \big( \mathcal{S}_{N(k)}^\eps f \big).
	\]	
	Referring to Remark \ref{rem:explanationQf2}, the definition of $Qf$ ensures that the second term has order $\mathcal{O}(\eps^{\alpha+1})$.
	Hence, $Qf$ ensures a lower constant in the second estimate in \eqref{eq:proofapee_3}.
\end{remark}

With these estimates, we are able to prove Theorem \ref{thm:est_epsm_alpha}.
\begin{proofof}{Theorem \ref{thm:est_epsm_alpha}}
	We have to estimate both error terms the right-hand side of \eqref{eq:errorSplit}.
	Using Lemma \ref{lem:errorestimateindepoftimeanddomain} and \eqref{eq:proofapee_3}, we verify that $\eqc\varphi$ satisfies the estimate
	\begin{equation*}	
	\norm{\eqc\varphi}_{\Li(0,\eps^{-\alpha}T;\calWnorm)} 
	\leq \norm{\eqcbiss{\mathcal{S}^\eps f}}_{\Lu(0,\eps^{-\alpha}T;\calLdOm)}
	\leq C \eps \norm{f}_{\Lu(0,\eps^{-\alpha}T;\HkOm{r(\alpha)})}.
	\end{equation*}
	where
	$r(\alpha) = \alpha+2\lfloor\alpha/2\rfloor+2$.
	Hence, from the definition of $\BBeps\tildeub$, we deduce that 
	\begin{equation}	\label{eq:trivialestimate}
	\norm{\BBeps\tildeub-\eqcbiss{\tildeub} }_{\Li(\calWnorm)}
	\leq C \eps \bigg( 
	\sum_{k=1}^{\alpha+2}\seminorm{\tildeub}_{\Li(0,\eps^{-\alpha}T;\HkOm{k})}
	+ \norm{f}_{\Lu(0,\eps^{-\alpha}T;\HkOm{r(\alpha)})}	\bigg).
	\end{equation}
	Next, applying Lemma \ref{lem:errorestimateindepoftimeanddomain} with $\eqc \eta^\eps(t) = \BBeps\tildeub(t)-\eqcbiss{\ueps(t)}$ (see \eqref{eq:proofapee_3}), we obtain
	\begin{equation}	\label{eq:proofest1_alpha}
	\norm{\eqc\eta}_{\Li(\calWnorm)} 
	\leq C \eps \bigg( 
	\sum_{k=\alpha+1}^{r(\alpha)+2}
	\seminorm{\tildeub}_{\Li(0,\eps^{-\alpha}T;\HkOm{k})} 
	+ \norm{\dtt\tildeub}_{\Li(0,\eps^{-\alpha}T;\HkOm{r(\alpha)})}
	+ \norm{\cini}_{\Hk{\alpha+2}(\Om)}  + \norm{\dini}_{\Hk{\alpha+2}(\Om)}  
	\bigg),
	\end{equation}
	where we used that for $\alpha\geq 2$, $r(\alpha) \geq \alpha+4$.
	Using \eqref{eq:trivialestimate} and \eqref{eq:proofest1_alpha} in \eqref{eq:errorSplit} proves estimate \eqref{eq:errorEstimateThm}
	and the proof of Theorem \ref{thm:est_epsm_alpha} is complete.
\end{proofof}


\subsection{A symmetrized tensor product of symmetric positive definite matrices is positive definite
	(proof of Lemma \ref{lem:tensorProdOfA0})}
\label{subsec:signOfEvenOrderTensor} 

Lemma \ref{lem:tensorProdOfA0} states that the tensor $S^{2n}(\otimes^n a^0)$ is positive definite, where $a^0$ is the homogenized tensor.
We prove here that this property is true for any symmetric, positive definite matrix.

A first important result is the following.
\begin{lemma}	\label{lem:posdeftensor_general_alpha}
	Let $R\in \Ten^{2n}(\R^d)$ be a positive definite tensor and let $A\in\Sym^2(\R^d)$ be a symmetric, positive definite matrix.
	Then the tensor of $\Ten^{2n+2}(\R^d)$ defined by $A_{i_1i_{2n+2}} R_{i_2\cdotss i_{2n+1}}$ is positive definite.
\end{lemma}
\begin{proof}
	As $A$ is symmetric positive definite, the Cholesky factorization provides an invertible matrix $H$ such that $A=H^TH$.
	For $\xi\in\Sym^{n+1}(\R^d)$, we thus have
	\begin{equation}	\label{eq:proof1_lemmaTensorPosDef_alpha}
	A_{i_1i_{2n+2}} R_{i_2\cdotss i_{2n+1}}	\xi_{i_1\cdotss i_{n+1}} \xi_{i_{n+1}\cdotss i_{2n+2}}
	= R_{i_2\cdotss i_{2n+1}}	\big( H_{rj}\xi_{j i_2\cdotss i_{n+1}}\big)  \big( H_{rj} \xi_{j i_{n+2}\cdotss i_{2n+2}}\big)
	\geq 0.
	\end{equation}
	As $R$ is positive definite, the equality holds if and only if
	$H_{rj}\xi_{j i_2\cdotss i_{n+1}}=0$ for all $r,i_2,\hdots, i_{n+1}\in\{1,\hdots,d\}$.
	Let $i_2,\cdots,i_{n+1}$ be arbitrarily fixed and denote $v_j = \xi_{j i_2\cdotss i_{n+1}}$.
	Hence, we have $H_{rj}v_j=0$ for all $r$, which is equivalent to $H^Tv=0$.
	As $H^T$ is regular, we obtain that $v=0$.
	We have proved that the equality in \eqref{eq:proof1_lemmaTensorPosDef_alpha} implies $\xi=0$.
	Hence the tensor is positive definite and the proof of the lemma is complete.
\end{proof}

With Lemma \ref{lem:posdeftensor_general_alpha} at hand, we are able to prove the following result, which implies Lemma \ref{lem:tensorProdOfA0}.

\begin{lemma}	\label{lem:posdeftensor_general2_alpha}
	If $A\in\Sym^2(\R^d)$ is a symmetric, positive definite matrix, then the tensor $S^{2n}(\otimes^n A)\in\Sym^{2n}(\R^d)$ is positive definite.
\end{lemma}
\begin{proof}
	We proceed by induction on $n$. The case $n=1$ is ensured by \cite[Lemma 4.1]{AbP16} (it can also be deduced from Lemma \ref{lem:posdeftensor_general_alpha}).
	We assume that the result holds for $1,\hdots,n-1$ and prove it for $n$.
	Let $\xi\in\Sym^{n}(\R^d)\backslash \{0\}$.
	First, assume that $n$ is odd.
	Then, the product $S^{2n}(\otimes^n A)\xi : \xi$ is composed of terms of the form
	\begin{equation}\label{eq:prooflem:posdeftensor_general2_alpha}
	A_{jk} A_{i_1i_2}\cdots A_{i_{2n-3}i_{2n-2}} 
	\xi_{j i_1\cdots i_{n-1}}\xi_{k i_{n}\cdots i_{2n-2}},
	\end{equation}
	i.e., one of the factor $A_{i_ri_s}$ share on index with both entities of $\xi$.
	The induction hypothesis combined with Lemma \ref{lem:posdeftensor_general_alpha} ensure that all these terms are strictly positive and thus $S^{2n}(\otimes^n A)$ is positive definite.
	Second, we assume that $n$ is even.
	Then, the product $S^{2n}(\otimes^n A)\xi : \xi$ is composed of terms of two forms.
	First, there are terms of the form \eqref{eq:prooflem:posdeftensor_general2_alpha}.
	By the same induction argument as before, they are strictly positive.
	Second, terms of the form
	\[
	A_{i_1i_2}\cdots A_{i_{n-1}i_{n}} A_{i_{n+1}i_{n+2}} \cdots A_{i_{2n-1}i_{2n}} 
	\xi_{i_1\cdots i_{n}}\xi_{i_{n+1}\cdots i_{2n}}
	= \big( A_{i_1i_2}\cdots A_{i_{n-1}i_{n}}\xi_{i_1\cdots i_{n}}\big)^2
	\geq 0. 
	\]
	Altogether, we verify that $S^{2n}(\otimes^n A)\xi : \xi>0$ 
	and the proof of the lemma is complete.
\end{proof}


\subsection{Proof of the new relation between the correctors (Theorem \ref{thm:lemmawpeven_alpha})}	
\label{subsec:prooflemma_even}

We prove the result for $r\geq 1$.
Note that in the case $r=1$, we adopt the convention that empty sums vanish, i.e., $\sum_{k=1}^0 x_k=0$.
For the sake of clarity, we assume that $|Y|=1$ so that $\ps{\cdot,1}_Y = \intMean{\cdot}_Y$.

In Section \ref{subsec:derivationFamily_subsec}, we explained that the cell problems \eqref{eq:cellProblemsWave_compact} are well-posed if and only if $\Pten^{j} =_S \Dten^{j}$.
Let us then replace $\Pten^{j}$ by $\Dten^{j}$ in the expression of the cell problems \eqref{eq:cellProblemsWave_compact}.
For $r\geq 1$, let $\chi_1,\hdots,\chi^{2r+1}$ be the $2r+1$ first zero mean correctors.
We define the tensors $A^{k},B^k,C^k\in\Ten^{2r+2}(\R^d)$ as
\begin{align*}
A^{k}_{i_1\cdotss i_{2r+2}}
&= \psbig{ a\naby\chi^{k}_{i_1\cdotss i_{k}}, \naby\chi^{2r+2-k}_{i_{k+1}\cdotss i_{2r+2}} }_Y
&& 1\leq k\leq 2r+1,\displaybreak[0]\\
B^k_{i_1\cdotss i_{2r+2}}
&= \psbig{a e_{i_1} \chi^k_{i_2\cdotss i_{k+1}}, \naby\chi^{2r+1-k}_{i_{k+2}\cdotss i_{2r+2}} }_Y
&& 0\leq k\leq 2r,\displaybreak[0]\\
C^k_{i_1\cdotss i_{2r+2}}
&= \psbig{a e_{i_1} \chi^k_{i_3\cdotss i_{k+2}}, e_{i_2}\chi^{2r-k}_{i_{k+3}\cdotss i_{2r+2}} }_Y
&& 0\leq k\leq 2r,\displaybreak[0]\\
D^k
&= \sum_{j=0}^{k-1} \intMeanbig{\Dten^{k-j} \otimes \chi^k \otimes \chi^{2r-k} }_Y
&& 0\leq k\leq 2r.
\end{align*}
Note that the symmetry of $a$ ensures the following symmetry relations for $A^k$ and $C^k$:
\begin{equation}	\label{eq:proofRel_1}
A^{k+1} =_S A^{2r+1-k},
\qquad
C^k =_S C^{2r-k}
\quad
0\leq k\leq 2r.
\end{equation}
Furthermore, using the test function $w=\chi^{2r+2-k}_{i_{k+1}\cdotss i_{2r+2}}$ in the cell problem for $\chi^k_{i_1\cdotss i_k}$ in \eqref{eq:cellProblemsWave_compact} and the symmetry of $a$, we obtain the following relations
\begin{equation}	\label{eq:proofRel_2}
\begin{aligned}
A^1 &=_S -B^0,\\
A^{k+1} &=_S -B^{k} + B^{2r+1-k} + C^{k-1} - D^{k-1}. \\
\end{aligned}
\end{equation}
Define then the tensor $T\in\Ten^{2r+2}(\R^d)$ as
\begin{equation} 	\label{eq::proof0_lemmawpeven_alpha}
T \coloneqq \sum_{k=0}^{2r} \sigma^k A^{k+1},
\qquad
\sigma^k = \left\{\begin{array}{ll} (-1)^{k+1} & \text{if } k\leq r \\ (-1)^{k}  & \text{if } k> r \end{array}\right..
\end{equation}
Using the definition of $\sigma^k$ and \eqref{eq:proofRel_1}, we verify that
\begin{equation} 	\label{eq::proof1_lemmawpeven_alpha}
T 
= \sum_{k=0}^{r} (-1)^{k+1} A^{k+1} + \sum_{m=r+1}^{2r} (-1)^{m} A^{m+1}
= \sum_{k=0}^{r} (-1)^{k+1} A^{k+1} + \sum_{k=0}^{r-1} (-1)^{2r-k} A^{2r+1-k}
=_S (-1)^{r+1} A^{r+1},
\end{equation}
where in the second equality we changed the index $k=2r-m$.
Using \eqref{eq:proofRel_2}, we decompose the tensor $T$ as
\begin{equation} 	\label{eq::proof2_lemmawpeven_alpha}
T = \sigma^0A^1 + \sum_{k=1}^{2r} \sigma^k A^{k+1}
=_S B^0 + U^1 + U^2 + U^3,
\end{equation}
where the tensors $U^1$, $U^2$ and $U^3$ are
\begin{equation*}
U^1 = \sum_{k=1}^{2r}\sigma^k \big(-B^{k} + B^{2r+1-k}\big),
\qquad
U^2 = \sum_{k=1}^{2r}\sigma^k C^{k-1},
\qquad
U^3 = \sum_{k=1}^{2r}(-\sigma^k) D^{k-1}.
\end{equation*}
The rest of the proof relies on the following result.
\def\rela{\textit{(i)}}
\def\relb{\textit{(ii)}}
\def\relc{\textit{(iii)}}
\begin{lemma}	\label{lem:proofRel_relU}
	The tensors $U^i$ defined above satisfy the following relations:
	\begin{equation*} 
	\begin{aligned}&
	\rela
	\quad
	U^1 = 0,
	\qquad
	\relb\quad
	U^2 =_S C^0 + (-1)^{r+1}C^r,
	\qquad
	\relc\quad
	U_3 =_S - \Hten^r,
	\end{aligned}
	\end{equation*}
	where $\Hten^r$ is the tensor defined in \eqref{eq:defpq_alpha}.
\end{lemma}

While the proof of \rela{} and \relb{} is direct, the proof of \relc{} requires preliminary work, done in the two following lemmas.
\def\HHten{H}
\begin{lemma}	\label{lem:symmetryD}
	The tensors $D^k$ satisfy 
	\begin{equation}	\label{eq:symmetryD}
	\sum_{k=1}^{r} (-1)^{k} D^{2r-k} =_S \sum_{k=1}^{r-1} (-1)^{k} D^{k} - \HHten,
	\qquad\text{where }~\HHten\coloneqq \sum_{j=1}^{r}\sum_{k=1}^{r} (-1)^{k+1} \intMeanbig{ \Gten^{2r-j-k} \otimes \chi^j\otimes\chi^k}_Y.
	\end{equation}
	where for $r=1$ the sum on the right-hand side vanishes.
\end{lemma}
\begin{lemma}	\label{lem:Gisgr}
	The tensor $\HHten$ defined in \eqref{eq:symmetryD} satisfies $\HHten =_S \Hten^r$,
	where $\Hten^r$ is the tensor defined in \eqref{eq:defpq_alpha}.
\end{lemma}
\begin{proofof}{Lemma \ref{lem:symmetryD}}
	Let us denote 
	$z^{j,k}= \intMeanbig{\Dten^{k-j} \otimes \chi^j \otimes \chi^{2r-k} }_Y$.
	As $\chi^0=1$ and the correctors have zero mean, we verify that 
	$z^{0,k} = z^{j,2r}= 0$.
	Hence, $D^k$ can be written as $D^k = \sum_{j=1}^k z^{j,k}$.
	Splitting the sum and making the change of index $\ell=2r-j$, we write for $1\leq k\leq r-1$
	\[
	D^{2r-k}
	=\sum_{j=1}^{2r-k} z^{j,2r-k}
	=\sum_{j=1}^{r} z^{j,2r-k}
	+\sum_{j=r+1}^{2r-k} z^{j,2r-k}
	=\sum_{j=1}^{r} z^{j,2r-k}
	+\sum_{\ell=k}^{r-1} z^{2r-\ell,2r-k}.
	\]
	As we verify that $z^{2r-\ell,2r-k} = _S z^{k,\ell}$, we can write
	\[
	\sum_{k=1}^{r} (-1)^{k} D^{2r-k}
	=_S (-1)^rD^r + V^1 + V^2,\qquad
	V^1 = \sum_{k=1}^{r-1} \sum_{j=1}^{r} (-1)^{k} z^{j,2r-k}
	,\quad
	V^2 = \sum_{k=1}^{r-1} \sum_{\ell=k}^{r-1}(-1)^{k} z^{k,\ell}.
	\]
	As $\Dten^{s}$ is nonzero only for even index (see \eqref{eq:definitionPk}), we have for all $s$,
	$(-1)^{s}\Dten^{s} = \Dten^{s}$.
	In particular, for all $s=k-\ell$, we have
	$(-1)^{k}\Dten^{k-\ell} = (-1)^{\ell}\Dten^{k-\ell}$
	which ensures that $(-1)^{k} z^{k,\ell} = (-1)^{\ell} z^{k,\ell}$.
	Hence, changing the summation order, we find
	\[
	V^2 
	= \sum_{k=1}^{r-1} \sum_{\ell=k}^{r-1}(-1)^{k} z^{k,\ell}
	= \sum_{\ell=1}^{r-1} \sum_{k=1}^{\ell} (-1)^{k} z^{k,\ell}
	= \sum_{\ell=1}^{r-1}(-1)^{\ell}  \sum_{k=1}^{\ell}  z^{k,\ell}
	= \sum_{\ell=1}^{r-1}(-1)^{\ell}  D^{\ell}.
	\]
	Furthermore, we verify that
	\[
	V^1+(-1)^rD^r = \sum_{k=1}^{r-1} \sum_{j=1}^{r} (-1)^{k} z^{j,2r-k} + \sum_{j=1}^{r} (-1)^{r} z^{j,r}
	= \sum_{k=1}^{r} \sum_{j=1}^{r} (-1)^{k} z^{j,2r-k} 
	= - \HHten.
	\]
	Summing the two last equalities we obtain \eqref{eq:symmetryD} and the proof is complete.
\end{proofof}
\begin{proofof}{Lemma \ref{lem:Gisgr}}
	Lemma \ref{lem:lemmawpodd_alpha} implies that $\Dten^{2s+1}=_S0$.
	Hence, the only nonvanishing terms in the double sum
	\[
	\HHten= \sum_{m=1}^{r}\sum_{n=1}^{r} (-1)^{n+1} \intMeanbig{ \Dten^{2r-(m+n)} \otimes \chi^m\otimes\chi^n}_Y,
	\]
	are $\{1\leq m,n\leq r : m+n \text{ is even} \}$.
	These terms are exactly given by
	\[
	\big\{m=2j,\, n=2k \,:\, 1\leq j,k\leq \lfloor r/2\rfloor \big\}
	\sqcup
	\big\{m=2j-1,\, n=2k-1 \,:\, 1\leq j,k\leq \lceil r/2\rceil \big\}.
	\]
	We verify that the expression of $\Hten^r$ in \eqref{eq:defpq_alpha} corresponds precisely to the sum on these two sets of index, i.e., $\HHten=_S\Hten^r$.
\end{proofof}

We now prove Lemma \ref{lem:proofRel_relU}.
\begin{proofof}{Lemma \ref{lem:proofRel_relU}}
	Let us start by proving \rela{}.
	Using the definition of $\sigma^k$, we	have
	\begin{align*}
	U^1 
	&= 	\sum_{k=1}^{r}(-1)^{k+1} \big(-B^{k} + B^{2r+1-k}\big)
	+\sum_{m=r+1}^{2r}(-1)^{m} \big(-B^{m} + B^{2r+1-m}\big)
	\\&= \sum_{k=1}^{r}(-1)^{k+1} \big(-B^{k} + B^{2r+1-k}\big)
	+\sum_{k=1}^{r}(-1)^{2r+1-k} \big(-B^{2r+1-k} + B^{k}\big),
	\end{align*}
	where in the second equality we changed the index $k=2r+1-m$.
	As $(-1)^{2r+1-k}=(-1)^{k+1}$, we obtain equality \rela{}.
	
	Next, we prove \relb{}.
	Using again the definition of $\sigma^k$, we find	
	\begin{equation*}
	U^2 
	= \sum_{k=1}^{r}(-1)^{k+1} C^{k-1} + \sum_{m=r+1}^{2r}(-1)^{m} C^{m-1}
	= \sum_{k=1}^{r}(-1)^{k+1} C^{k-1} + \sum_{k=2}^{r+1}(-1)^{2r+2-k} C^{2r+1-k},
	\end{equation*}
	where we changed the index $k=2r+2-m$.
	The symmetry \eqref{eq:proofRel_1} implies that $C^{2r+1-k} =_S C^{k-1}$.
	Hence, noting that $(-1)^{2r+2-k}= (-1)^{k}$, we obtain \relb{}.
	
	Finally, we prove \relc.
	With the definition of $\sigma^k$, we find	
	\begin{equation*}
	U^3
	= \sum_{n=1}^{r}(-1)^{n} D^{n-1} + \sum_{m=r+1}^{2r}(-1)^{m+1} D^{m-1}
	= \sum_{k=0}^{r-1}(-1)^{k+1} D^{k} + \sum_{k=1}^{r}(-1)^{2r+2-k} D^{2r-k},
	\end{equation*}
	where we made the changes of index $k=n-1$ and $k=2r+1-m$, respectively.
	As $(-1)^{2r+2-k} = (-1)^{k}$, 
	using Lemma \ref{lem:symmetryD} we obtain
	\[
	U^3 
	=_S \sum_{k=0}^{r-1}(-1)^{k+1} D^{k} + \sum_{k=1}^{r-1}(-1)^{k} D^{k} - \HHten
	= -D^{0}- \HHten.
	\] 
	As $\chi^0=1$ and the correctors have zero mean, we verify that $D^{0}=0$.
	Using Lemma \ref{lem:Gisgr}, we obtain $U^3=-\Hten^r$ and \relc{} is proved.
\end{proofof}

With Lemma \ref{lem:proofRel_relU} at hand we can prove Theorem \ref{thm:lemmawpeven_alpha}.
\begin{proofof}{Theorem \ref{thm:lemmawpeven_alpha}}
	Combining \eqref{eq::proof1_lemmawpeven_alpha}, \eqref{eq::proof2_lemmawpeven_alpha}
	and Lemma \ref{lem:proofRel_relU} yields
	\begin{equation}	\label{eq:proofRel_last}
	(-1)^{r+1} A^{r+1} =_S B^0 + U^1+U^2+U^3 =_S B^0+ C^0 +(-1)^{r+1}C^{r} - \Hten^r.
	\end{equation}
	We verify that the tensors $\Gten^{2r}$ and $\Kten^r$, defined in \eqref{eq:defbis_hr} and \eqref{eq:defpq_alpha}, respectively, can be written as
	\[
	\Gten^{2r} = B^0+C^0,
	\qquad
	\Kten^r = -A^{r+1} + C^{r}.
	\]
	Hence, we deduce from \eqref{eq:proofRel_last} that
	\begin{equation*}	
	\Gten^{2r} =_S (-1)^{r} \big( - A^{r+1} + C^{r}\big)  + \Hten^r
	= (-1)^{r}\Kten^r  + \Hten^r.
	\end{equation*}
	This equality matches the decomposition \eqref{eq:decompostionhr}
	and the proof of Theorem \ref{thm:lemmawpeven_alpha} is complete.
\end{proofof}


\section{Conclusion} \label{sec:conclusion}

In this paper, we presented a family of effective equations for wave propagation in periodic media for arbitrary timescales.
In particular, for any given $\alpha\geq0$, our main result (Theorem \ref{thm:est_epsm_alpha}) ensures the effective solutions to be close to $\ueps$ in the $\Li(0,\eps^{-\alpha} T;\Wnorm)$ norm.
As emphasized, the effective equations are well-posed without requiring any regularization process.
In addition, we described a numerical procedure to compute the effective tensors of equations in the family.
We showed that the computational cost of this procedure can be significantly reduced by using a new relation between the correctors.
This relation should also be used to compute the tensors of the alternative effective model available in the literature. 

One question that is raised is how to find the best effective equation in the family?
More precisely, can we find a criterion to find an optimal equation in the family and can we build such an equation explicitly?
These interrogations also concern the effective models from \cite{ALR18} and \cite{BeG17}.
Answers to these questions are exciting topics left for future research.


\def\cprime{$'$}


\end{document}